\documentclass[11pt, a4paper]{article}
\usepackage{amsmath,amsthm,amssymb,amsfonts}
\usepackage{a4wide}
\newtheorem{theorem}{Theorem}[section]
\newtheorem{lemma}[theorem]{Lemma}
\newtheorem{proposition}[theorem]{Proposition}
\newtheorem{definition}[theorem]{Definition}

\newtheorem{corollary}[theorem]{Corollary}

\theoremstyle{remark}

\newtheorem{remark}[theorem]{Remark}

\def\Xint#1{\mathchoice
{\XXint\displaystyle\textstyle{#1}}%
{\XXint\textstyle\scriptstyle{#1}}%
{\XXint\scriptstyle\scriptscriptstyle{#1}}%
{\XXint\scriptscriptstyle\scriptscriptstyle{#1}}%
\!\int}
\def\XXint#1#2#3{{\setbox0=\hbox{$#1{#2#3}{\int}$ }
\vcenter{\hbox{$#2#3$ }}\kern-.6\wd0}}

\def\dashint{\Xint-}

\usepackage{graphicx}
\usepackage[usenames,dvipsnames]{color}
\usepackage[colorlinks=true, pdfstartview=FitV, linkcolor=black, citecolor=blue, urlcolor=blue]{hyperref}

\makeatletter

\@addtoreset{equation}{section}
\makeatother

\newcommand{\E}{{\mathbb E}}

\newcommand{\N}{{\mathbb N}}
\newcommand{\Z}{{\mathbb Z}}
\newcommand{\Q}{{\mathbb Q}}
\newcommand{\R}{{\mathbb R}}

\newcommand{\mI}{\mathcal{I}}
\newcommand{\mF}{\mathcal{F}}

\newcommand{\mE}{\mathcal{E}}
\newcommand{\mA}{\mathcal{A}}
\newcommand{\CL}{\mathrm{CL}}

\newcommand{\pa}{{\partial}}
\newcommand{\na}{{\nabla}}

\newcommand{\eps}{{\varepsilon}}

\def\div{\hbox{div \!}}

\def\sspace{\smallskip \noindent}
\def\mspace{\medskip \noindent}
\def\bspace{\bigskip\noindent}

\def\Xint#1{\mathchoice
{\XXint\displaystyle\textstyle{#1}}%
{\XXint\textstyle\scriptstyle{#1}}%
{\XXint\scriptstyle\scriptscriptstyle{#1}}%
{\XXint\scriptscriptstyle\scriptscriptstyle{#1}}%
\!\int}
\def\XXint#1#2#3{{\setbox0=\hbox{$#1{#2#3}{\int}$ }
\vcenter{\hbox{$#2#3$ }}\kern-.6\wd0}}

\def\dashint{\Xint-}

\usepackage{mathtools}
\usepackage{pst-all}
\usepackage{pstricks}
\usepackage{etex}
\usepackage{float} 
\DeclareMathOperator*{\argmin}{arg\,min}
\newcommand{\mC}{\mathcal{C}}
\renewcommand{\P}{\mathcal{P}}
\newcommand{\dx}{\textbf{$ \ \mathrm{d}x$}}
\newcommand{\Ed}{\mathrm{Ed}}

\newcommand{\CC}{\mathrm{CC}}

\usepackage{tikz,pgfplots}
\usetikzlibrary{patterns,matrix,backgrounds, calc, fit,decorations.pathreplacing}
\newcommand{\diam}{\text{diam}}

\title{Homogenization of stiff inclusions\\
through network approximation} 

\author{David G\'erard-Varet, Alexandre Girodroux-Lavigne}
\begin{document}
\maketitle
\begin{abstract}
We investigate the homogenization of inclusions of infinite conductivity, randomly stationary distributed inside a homogeneous conducting medium. 
A now classical result by Zhikov shows that, under a logarithmic moment bound on the minimal distance between the inclusions, an effective model with finite homogeneous conductivity exists. Relying on ideas from network approximation, we provide a relaxed criterion ensuring homogenization. Several examples not covered by the previous theory are  discussed. 
\end{abstract}

\section{Introduction}
The classical theory of homogenization deals with conductivity matrices $A^\eps = A(\frac{x}{\eps})$ having uniform lower and upper bounds:  $\alpha \mathrm{Id} \le A \le \beta \mathrm{Id}$, with $0 < \alpha < \beta < +\infty$. 
Degenerate cases, for which $A$ is allowed to vanish or to take infinite values, still lead to open questions. A problem typical of the first situation is as follows. Given a union   $F = \cup_{I \in \mI} \, I$ of closed connected domains $I$ (the {\em inclusions}), either periodic or random stationary, one considers the following system: 
\begin{equation} \label{soft_inclusion}
\left\{
\begin{aligned}
- \Delta u^\eps  & = f \quad \text{in } \: U\setminus  F^\eps, \\
u^\eps\vert_{\pa U} & = 0, \\
\pa_\nu u^\eps\vert_{\pa F^\eps} & = 0.
\end{aligned}
\right. 
\end{equation} 
Here, $U$ is a smooth bounded domain of $\R^3$, and $F^\eps = \cup  \, \eps I $ where the union is restricted to inclusions satisfying $\eps I \subset U$. The Neumann condition corresponds to zero conductivity inside the inclusions, while conductivity is normalized to $1$ outside. The problem is to determine a homogenized limit model. Namely, one tries to understand under which conditions the  solution $u^\eps$ of \eqref{soft_inclusion} converges as $\eps \rightarrow 0$ to the solution $u^0$ of  
\begin{equation} \label{target_system}
\left\{
\begin{aligned}
- \div(A^0 \na u^0)  & = (1-\lambda) f \quad \text{in } \: U, \\
u^0\vert_{\pa U} & = 0, \\
\end{aligned}
\right.
\end{equation}
for some non-degenerate effective conductivity matrix $A^0$. The constant $\lambda$ refers to the density of the inclusions, and  is the weak limit of $1_{F^\eps}$. This homogenization problem has become standard, and is discussed extensively in the classical book \cite{book:ZKO}. In brief, homogenization holds under some mild but uniform regularity requirement on the geometry of the inclusions, and in a  general random stationary ergodic setting, as long as $\R^3\setminus F$ is connected. One strategy to show this result is through an adaptation of the classical div-curl approach. With regards to this strategy, a nice feature is that the first equation and the Neumann condition are equivalent to the single equation 
$$ \div (1_{F^\eps} \na u^\eps) = 1_{F^\eps} f \quad \text{ in U}$$
meaning that the divergence of the flux $1_{F^\eps} \na u^\eps$ is relatively compact in $H^{-1}(U)$. On the contrary, one difficulty is that $\na u^\eps$ is defined only on $U \setminus F^\eps$ and there is no canonical extension as a gradient of an $H^1$ function over the whole domain $U$. Still, construction of such a potential field extension inside the inclusions is possible under mild requirements. Let us stress that such extension is a local process, in the sense that extension in one inclusion can be considered independently from the others.  

\mspace
As regards the opposite case of inclusions with infinite conductivity, the theory of homogenization is less complete.  The analogue of \eqref{soft_inclusion} reads: 
\begin{equation} \label{main_system}
\left\{
\begin{aligned}
- \Delta u^\eps  & = f \quad \text{in } \: U\setminus  F^\eps, \\
\na u^\eps & = 0 \quad \text{in } \: F^\eps, \\
u^\eps\vert_{\pa U} & = 0, \\
\int_{\pa I^\eps} \pa_\nu u^\eps & = 0,  \quad \forall I^\eps \in \mathrm{CC}(F^\eps).
\end{aligned}
\right. 
\end{equation} 
Here, we denote $I^\eps := \eps I$, that belongs to the set $\mathrm{CC}(F^\eps)$  of connected components of $F^\eps$. On such inclusions, the condition $\na u^\eps = 0$  ensures that the potential is constant, which corresponds to infinite conductivity. The integral condition $\int_{\pa I^\eps} \pa_\nu u^\eps  = 0$ corresponds to the fact that the total flux through the boundary of the inclusion is zero (a continuous  analogue of Kirchoff's  nodal rule). Let us note that system \eqref{main_system} has an important extension to modelling of suspensions in fluid mechanics, in which the Laplace operator is replaced by the Stokes operator,  the electric field being replaced by the Newtonian stress tensor of the fluid. The analogue question of the  effective viscosity of passive suspensions has attracted a lot of attention recently, {\it cf.} \cite{Haines&Mazzucato,AGKL,MR4102716,MR4098775,GVH, GVRH,gerard2020correction,gerard2021derivation,DuerinckxGloria,duerinckx2020effective,DueGloria} among many. 

\mspace
Again, the point is to show convergence to the effective system \eqref{target_system}. However, contrary to the case of {\em soft inclusions} \eqref{soft_inclusion}, the case of {\em stiff inclusions} (in the terminology of \cite{book:ZKO})  requires more than the connectedness of $\R^3\setminus F$. Some condition on the distance between particles is needed. The main difference between \eqref{soft_inclusion} and \eqref{main_system}, responsible for extra assumptions in the latter case, can be seen when trying to adapt the div-curl approach. The difficulty is reversed. While the potential field $\na u^\eps$ has a natural extension as a potential field over $U$ (just extend $u^\eps$ by its constant value inside each inclusion), the divergence of the  flux $1_{F^\eps} \na u^\eps$ is no longer controlled in $H^{-1}(U)$. The point is then to find a nice extension inside the inclusions of a field (namely $\na u^\eps$) with given divergence (namely $f$).  This constraint on the divergence makes such extension process non-local, and implies global assumptions on the configuration of the inclusions. 

\mspace
Up to our knowledge, homogenization was so far only established by Zhikov for a random spherical structure $F = \cup_{i \in \N} B_i$,  under the assumption that almost surely, 
$$ \limsup_{N \rightarrow +\infty} \frac{1}{N^3} \sum_{B_i  \subset (-N,N)^3} \mu_{i} < + \infty, \quad \text{ where } \quad \mu_{i} := |\ln(d(B_i, F\setminus B_i))|.$$
See Theorem \ref{Zhikov} below, or \cite[chapter 8]{book:ZKO} for more.  Obviously, in the case where there is a minimal distance $\delta > 0$  between the spheres, such condition is satisfied. As will be discussed later in the paper,  the logarithmic term $\mu_i$  comes from the fact that given any smooth Dirichlet data $\varphi_i$  on $B_i$, there is an  $H^1$ function $\phi$ with $\phi = \varphi_i$ on $B_i$  and  $\varphi = 0$ on all other $B_j$'s, exploding at most  like $\mu_{i}$. This gives a glimpse on how such assumption may help to build the extension operator alluded to above.

\mspace
In the opposite very dense setting  where for all $i$, $|\mu_i| \le \mu \ll 1$,  homogenization is not always possible. Roughly, one can then find a sequence $\mu = \mu(\eps)$  and  boundary data problems of the form 
 \begin{equation} \label{boundary_data_system}
\left\{
\begin{aligned}
- \Delta u^{\eps, \mu}  & = 0 \quad \text{in } \: U\setminus  F^\eps, \\
\na u^{\eps, \mu} & = 0 \quad \text{in } \: F^\eps, \\
u^{\eps, \mu} \vert_{\pa U} & = \varphi, \\
\int_{\pa I^\eps} \pa_\nu u^{\eps, \mu} & = 0,  \quad \forall I^\eps \in \mathrm{CC}(F^\eps) 
\end{aligned}
\right. 
\end{equation} 
for which
$$  \liminf_{\eps \rightarrow 0} \int_{U} |\na u^{\eps, \mu(\eps)}|^2 = +\infty. $$
Such negative results for dense settings rely notably on the so-called  {\em network approximation method}, as described in the monograph \cite{BeKoNo}. The idea behind this method  is that when the spheres get close to one another, the analysis of systems of type \eqref{boundary_data_system}  can be simplified:  the asymptotic behaviour of the system (notably of its energy), can be deduced from  the properties of an underlying weighted graph where: 
\begin{itemize}
\item  nodes of the graph are the spheres of infinite conductivity
\item edges are pairs of spheres close to one another, in the sense that they belong to adjacent Voronoi cells
\item each edge $e = \{B_i, B_j\}$ has weight $\mu_e := |\ln d(B_i, B_j)|$.
\end{itemize}
For instance, the energy of the system can be approximated as $\mu \rightarrow 0$ by  a reduced discrete energy associated to the graph. Such energy may then be analyzed by tools of graph theory, and its divergence as $\mu \rightarrow 0$ established. We refer to \cite{BoPa,MR1857272,BeMi,berlyand2002error}, and again to the monograph \cite{BeKoNo} and its bibliography. See also \cite{BeBoPa,MR2525112,MR2319721} in the context of fluid mechanics and suspensions.  

\mspace
The goal of this paper is to investigate some intermediate situations, in which the condition on the distance between the inclusions is not necessarily satisfied, but homogenization is still possible. To identify  situations of this kind, we will rely on two notions: 
\begin{itemize}
\item the notion of {\em multigraph of inclusions}, reminiscent of the network approximation just mentioned. 
\item the notion of {\em short of inclusions},  reminiscent of the study of electrical circuits. 
\end{itemize}
Thanks to these notions, we will formulate two assumptions \eqref{H1}-\eqref{H2}. The weaker one, \eqref{H1},  allows to define the homogenized matrix $A^0$, while the stronger one, \eqref{H2}, allows for homogenization. This set of two assumptions is implied by a logarithmic moment bound on the minimal distance,  but is much more general. For instance, our homogenization theorems apply to the case of inclusions or clusters of  inclusions of arbitrairily large size, with a mere moment bound on their diameter. Further examples, like inclusions with anisotropic structure, will be discussed. 

\mspace
The outline of the paper is as follows. In Section \ref{sec1},  after a brief reminder on stationary closed sets, we describe the class of inclusions  under consideration. For such class, we define multigraphs and shorts of  inclusions. We conclude the section by the statement of our homogenization results. In Section \ref{section_preliminary}, we show certain extension properties for potential  and solenoidal vector fields, crucial to the proof of homogenization.  This proof is then given in Section \ref{sec_main_results}. Eventually, Section \ref{section_disc_assumptions}  provides complements on our main hypotheses \eqref{H1}-\eqref{H2}.

\section{Statements of the main results} \label{sec1}

\subsection{Reminder on stationary closed sets.}
 We follow here the definition in the lecture notes \cite[chapter 13]{bart}. A random closed set is  a random variable $F = F(\omega) \subset \R^3$ from a probability space  to the set 
$$\CL = \{F \subset \R^3, \: F \: \text{closed} \}$$
equipped with the borelian $\sigma$-algebra $B(\CL)$. We remind that the topology on $\CL$ is the topology induced by the sets $\CL_K = \{F, F \cap K \neq \emptyset \}$,
 $K$ describing the compact subsets of $\R^3$. Note that by considering  the law $P_F$  of $F$ on $\CL$, we can always assume that the probability space is $(\CL, B(\CL), P_F)$  and that $F(\omega) = \omega$.  This is the canonical representation of the random closed set. We now introduce the  shift $\tau_x : \CL \rightarrow \CL$, $\tau_x(\omega) = \omega - x$, $x \in \R^3$. We say that the random closed set  $F$ is  $\R^3$-stationary, resp. $\Z^3$ stationary, if $ P_F \circ \tau_x^{-1} = P_F$ for all $x \in \R^3$, resp. $x \in \Z^3$. We say that it is ergodic if under the assumption that $\tau_x(\mA) = \mA$ for all $x \in \R^3$, resp.   $\tau_x(\mA) = \mA$ for all $x \in \Z^3$,  one has $P_F(\mA) \in \{0,1\}$.

\mspace
Note that there is a trick to turn a $\Z^3$-stationary ergodic random closed set into an $\R^3$-stationary ergodic random closed set. Namely, given $F$ a $\Z^3$-stationary random closed set, one defines 
$$ \tilde F : \left( \CL \times (0,1)^3 , P_F \otimes \text{Leb} \right) \rightarrow \CL , \quad (\omega,x) \rightarrow \omega - x. $$
Then, one considers on $\CL$ the measure image of $P_F \otimes \text{Leb}$ by $\tilde F$, and denoting $P_{\tilde F}$ such a measure, one can check that 
$P_{\tilde F} \circ \tau_x^{-1} = P_{\tilde F}$ for all $x \in \R^3$. Indeed, given a borelian $\mA$ of $\CL$, by Fubini's theorem, 
\begin{align*}
 P_{\tilde F} \circ \tau_x^{-1}(\mA) & = P_F \times \text{Leb} \left( \{(\omega,y), \omega-y  \in \mA+x \}\right) \\
 & = \int_{(0,1)^3}  \int_{\CL}   1_{\mA+x}(\omega-y) dP_F(\omega) dy =\int_{(0,1)^3} P_F(\tau_{-x-y}(\mA)) dy. 
 \end{align*}
By $\Z^3$-stationarity of $P_F$, the integrand is $\Z^3$-periodic, and the result follows by the change of variable $y' = y+x$. Moreover, $\tilde F$ is ergodic if $F$ is. Indeed, if  $\tau_y(\mA) = \mA$ for all $y \in \R^3$, we find (take $x=0$ in the previous identity): 
\begin{align*}
 P_{\tilde F}(\mA) &  =\int_{(0,1)^3} P_F(\tau_{-y}(\mA)) dy  = P_F(\mA)  \in \{0,1\}. 
 \end{align*}
Note that a property that is almost sure with respect to $P_{\tilde F}$ will hold for realizations of the form  $\omega - x$, that is $F(\omega) - x$, for $P_F$-almost every $\omega$ and almost every $x$.  As the choice  $x = 0$ characterizing $F$  made at the beginning is irrelevant, this gives results about the original random closed set $F$. 

\mspace
More generally, let $(F_i)_{i \in I}$ a family of random closed sets defined on the same probability space, and $P$ its joint law on $\CL^I$. We define the shift  $\tau_x : \CL^I \rightarrow \CL^I$, $\tau_x((\omega_i)_{i \in I}) = (x+\omega_i)_{i \in I}$, $x \in \R^3$.  We say that  $(F_i)_{i \in I}$  is $\R^3$-stationary, resp. $\Z^3$-stationary, if $P \circ \tau_x^{-1} = P$ for all $x \in \R^3$, resp. $x \in \Z^3$. We say that it is ergodic if under the assumption that $\tau_x(\mA) = \mA$ for all $x \in \R^3$, resp.   $\tau_x(\mA) = \mA$ for all $x \in \Z^3$,  one has $P(\mA) \in \{0,1\}$.
Of course, stationarity of the family is stronger than the stationarity of each of its elements, and is the right notion as soon as one deals with events involving intersections, unions of several random closed sets $F_i$. 

\mspace
A convenient unified description, that we adopt from now on, is the following. We consider a probability space  $(\Omega, \mA, P)$, equipped with a family of  maps  $\tau_x : \Omega \rightarrow \Omega$, $x \in \R^3$, which satisfies: 
\begin{itemize}
\item[i)] $(x,\omega) \rightarrow \tau_x(\omega)$ measurable
\item[ii)] $\forall x, y, \quad \tau_{x+y} = \tau_x \circ \tau_y$ (shift)
\item[iii)]  $\forall x, \quad P = P \circ \tau_x^{-1}$  (measure preserving)
\item[iv)]  $\big( \forall x, \: \tau_x(\mA) = \mA \big)  \: \Rightarrow \:  P(\mA) \in \{0,1\}$ (ergodicity). 
\end{itemize}
Under this description, an (ergodic) stationary closed set is then a r. v. $F : \Omega \rightarrow \CL$ s.t. 
\begin{equation} \label{def:stationaryset} 
F(\tau_x(\omega)) = F(\omega) - x, \quad \forall \omega \in \Omega, \: \forall x \in \R^3.
\end{equation}

\mspace
Introducing  the subset $\mF$ of  $\Omega$ defined by 
\begin{equation} \label{defi_rmF}
\mF = \{ \omega \in \Omega, \: 0 \in F(\omega) \} 
\end{equation}
one can notice that 
$$ F(\omega) = \{x, \tau_x(\omega) \in \mF\}. $$
This is the point of view taken in \cite{book:ZKO}.


\subsection{Geometry of the inclusions.} Our homogenization results apply to a class of ergodic stationary closed sets $F = F(\omega)$ that we now describe. Let $F$ be a closed set  and $\mathrm{CC}(F)$ its  family of connected components, so that $F = \cup_{I \in \mathrm{CC}(F)} I$.  Each of this connected component is of course closed. We introduce the following geometric conditions:      

\mspace
\begin{itemize} 
\item[(G1)] Regularity of the inclusions : there exists $d > 0$, such that all $I \in \mathrm{CC}(F)$ is the closure of a $C^2$ bounded domain satisfying an interior and exterior ball condition with uniform  radius $d$. 
 
  \item[(G2)] Geometry of the {\em gaps} : there exists $\delta, a > 0$, such that   for all  $I \in \mathrm{CC}(F)$, the set 
  $$I \cap \{d(x, F\setminus I) \le  \delta \}$$
   has a finite number of connected components $I_\alpha$, $\alpha = 1, \dots, N_I$,  with $\sup_I N_I < +\infty$,  and with at most one couple $(J, \beta) \neq (I, \alpha)$  such that $d(I_\alpha, J_\beta) \le 2\delta$. Moreover, for an appropriate local system of cylindrical coordinates $(r,\theta,z)$:  
\begin{equation} \label{paraboloid}
I_\alpha  \subset \{ z \ge d(I_\alpha,I_\beta)/2 +  a r^2  \}, \quad  I_\beta  \subset \{ z \le - d(I_\alpha,I_\beta)/2 - a r^2  \},
\end{equation}
that is $I_\alpha$ and $I_\beta$ are separated by paraboloids with uniform curvature. 
 \end{itemize}
 
 \begin{remark}
By  \eqref{paraboloid}, there is a unique couple of points $(x_{I,\alpha}, x_{J,\beta})  \in I_\alpha   \times I_\beta$ such that 
\begin{equation} \label{xialpha}
|x_{I,\alpha} - x_{J,\beta}| =  d(I_\alpha, J_\beta). 
\end{equation}
 See Figure \ref{Geometry}. 

\mspace
\definecolor{gray1}{gray}{0.85}
\definecolor{gray2}{gray}{0.67}
\definecolor{gray3}{gray}{0.55}

\begin{figure}[h!]
    \centering
    \psset{unit=1.3cm}
\begin{pspicture}(4.3,-1.4)(3,3)

\begin{psclip}{\pscustom[linewidth=2pt,fillstyle=solid,
fillcolor=gray]{
\pscurve[showpoints=false,linewidth=1.5pt,linecolor=black] (-0.05,-0.15) (0.15,-0.6) (-0.5,-1) (-1,-1.5) (-1.3,-0.75) (-0.6,0.1) (-0.05,-0.15)
}}
\pscircle[fillstyle=solid,fillcolor=gray1](0,-0.05){0.25}
\end{psclip}

\begin{psclip}{\pscustom[linewidth=2pt,fillstyle=solid,
fillcolor=gray]{
\pscurve[showpoints=false,linewidth=1.5pt,linecolor=black] (-0.3,1.4) (0,2) (0.6,1.2) (1.1,0.7) (1.5,0.4) (2,0) (1.9,-0.5) (1.5,-0.6) (0,0) (-0.3,1.4)
}}
\pscircle[fillstyle=solid,fillcolor=gray1](0,-0.05){0.25}
\pscircle[fillstyle=solid,fillcolor=gray1](0,2.05){0.30}
\pscircle[fillstyle=solid,fillcolor=gray1](2.05,0){0.30}
\end{psclip}

\begin{psclip}{\pscustom[linewidth=2pt,fillstyle=solid,
fillcolor=gray]{
\pscurve[showpoints=false,linewidth=1.5pt,linecolor=black] (0,2.1)   (2.5,1.7) (2.1,0) (2.2,-0.1) (2.3,-0.1) (3.2,3) (-0.2,2.4) (-0.3,2.2) (0,2.1))}}
\pscircle[fillstyle=solid,fillcolor=gray1](0,2.05){0.30}
\pscircle[fillstyle=solid,fillcolor=gray1](2.05,0){0.30}
\end{psclip}
\psframe[linewidth=1.5pt,linecolor=black,linestyle=dashed] (1.6,-0.4)(2.45,0.40)

\pspolygon[fillstyle=solid,linewidth=0,
fillcolor=gray] (6.96,2.98) (8.48,2.98) (8.48,1.26)

\begin{psclip}{\pscustom[linewidth=2pt,fillstyle=solid,
fillcolor=gray]{
\pscurve[showpoints=false,linewidth=1.5pt,linecolor=black] (7,3) (6.6,1.2) (8.5,1.3)
}}
\pscircle[fillstyle=solid,fillcolor=gray1](6.5,1){1}
\end{psclip}

\pspolygon[fillstyle=solid,linewidth=0,
fillcolor=gray] (4.52,2.54) (6.03,-0.98) (4.52,-0.98)

\begin{psclip}{\pscustom[linewidth=2pt,fillstyle=solid,
fillcolor=gray]{
\pscurve[showpoints=false,linewidth=1.5pt,linecolor=black] (4.5,2.5) (6.3,0.9) (6,-1)
}}
\pscircle[fillstyle=solid,fillcolor=gray1](6.5,1){1}
\end{psclip}

\rput(6.26,1.02){$\bullet$}
\rput(6.1,0.83){\small $x_{I,\alpha}$}

\rput(6.6,1.20){$\bullet$}
\rput(7,1.15){\small $x_{J,\beta}$}

\psframe[linewidth=1.5pt,linecolor=black,linestyle=dashed] (4.5,-1)(8.5,3)
\psline[linewidth=1pt,linecolor=black,linestyle=dashed]  (2.45,-0.4) (4.5,-0.96)
\psline[linewidth=1pt,linecolor=black,linestyle=dashed]  (2.45,0.40) (4.5,2.96)
\rput(0.5,0.5){$I$}
\rput(6.6,0.3){$I_\alpha$}
\rput(6.3,1.7){$I_\beta$}
\rput(5,-0.5){$I$}
\rput(8,1.5){$J$}
\end{pspicture}
    \caption{Geometry of the inclusions with a close-up on a gap.}
    \label{Geometry}
\end{figure}
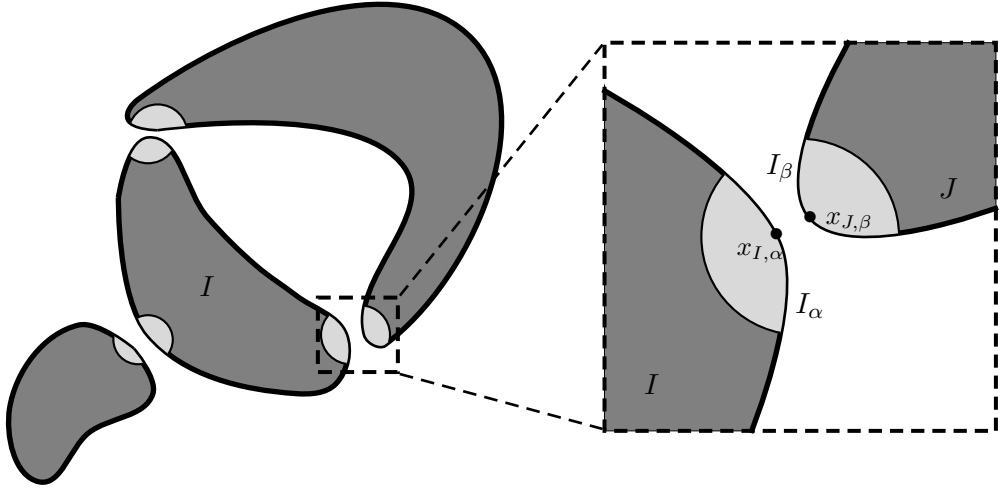
\end{remark}
\begin{definition} {\bf (Admissible set of inclusions) } \label{rem_G1G2}

\sspace
We say that an ergodic stationary closed set  $F = F(\omega)$ is an admissible set of inclusions if  it satisfies (G1)-(G2)  almost surely, with $d,\delta,a$, and $\sup_I N_I$ bounded by a deterministic constant. 
 \end{definition}
\begin{remark}
At this stage, we do not assume anything on the diameter of the inclusions. 
\end{remark}

\subsection{Multigraph of inclusions}
We now associate to a closed set satisfying (G1)-(G2)  an unoriented  multigraph  that we  call  multigraph of  inclusions. Roughly, the nodes of the multigraph are the connected components of the closed set, and we link pieces of these connected components that are $\delta$-close. See Figure \ref{Multigraph} for an illustration. In a more formal way:  
\begin{definition} \label{defi:multigraph} {\bf (Multigraph of  inclusions)}

\sspace
Let $F$ a closed set satisfying (G1)-(G2). For $\delta$ the constant in (G2),  the $\delta$-multigraph of  inclusions associated to $F$, called multigraph of inclusions for brevity, is the unoriented multigraph $Gr(F) = (\mathrm{CC}(F),\mathrm{Ed}(F))$ with set of nodes $\mathrm{CC}(F)$ and set of edges $\mathrm{Ed}(F)$  made of elements of the form  
 $$ e = [x_{I,\alpha}, x_{J,\beta}], \quad \text{ with } \:  I \neq J \in \mathrm{CC}(F), \quad  1 \le \alpha \le N_I, \quad 1 \le \beta \le N_J, \quad d(I_\alpha, J_\beta)  \le \delta.$$
\end{definition}
\noindent
 We say that  an edge $e$ as above connects $I$ to $J$, and  denote it  $I \overset{e}{\leftrightarrow} J$. If there exists $e \in \mathrm{Ed}(F)$ such that $I \overset{e}{\leftrightarrow} J$, we simply note $I \leftrightarrow J$.  Note that several edges can connect the same pair of nodes, hence the multigraph structure. This corresponds to multiple gaps for a given pair of inclusions.   To each of the edges, we associate a weight $\mu_e$, by the formula 
 \begin{equation} \label{weight}
  \mu_e := \ln |e|  = \ln |d(I_\alpha, J_\beta)| \quad \text{ for } \quad e = [x_{I,\alpha}, x_{J,\beta}].
  \end{equation}
As explained in the introduction, there is a strong analogy between our multigraph of inclusions and the network approximation of \cite{BeKoNo}.

\mspace
For later use, we further define 
\begin{definition} \label{defi:cluster} {\bf (Cluster of  inclusions)}

\sspace
Let $F$ satisfying (G1)-(G2). A  {\em cluster} of $F$ is a union of all the inclusions  that are the nodes of a connected component of the 
$\delta$-multigraph $Gr(F)$ (not to be confused with a connected component of $F$ itself, which corresponds to a single inclusion/node).   
\end{definition}

\begin{remark} \label{rem_I0_C0}
 For $F$ a set of inclusions, and $y \in \R^3$, we shall denote: 
 \begin{itemize}
 \item  $I_{y,F} \in  \mathrm{CC}(F)$ the inclusion containing $y$, with the convention that $I_{y,F} = \emptyset$ if $y \not\in F$. 
 \item $C_{y,F} \in  \mathrm{CC}(F)$ the cluster containing $y$, with the convention that $C_{y,F} = \emptyset$ if $y \not\in F$. 
 \end{itemize}
 In the case $F = F(\omega)$ is stationary, one has clearly 
 $$ I_{y,F(\omega)} = I_{0,F(\tau_y(\omega))}, \quad C_{y,F(\omega)} = C_{0,F(\tau_y(\omega))}. $$
 This implies that $(y,\omega) \mapsto \diam(I_{y,F(\omega)})$,  $(y,\omega) \mapsto \diam(C_{y,F(\omega)})$, $(y,\omega) \mapsto \sharp(C_{y,F(\omega)})$ are stationary random fields.  We will apply the ergodic theorem several times  to these fields. 
\end{remark}

\definecolor{gray4}{gray}{0.90}

\begin{figure}[h!]
    \centering
    \psset{unit=1.3cm}
\begin{pspicture}(4.3,-0.9)(3,3)
\begin{psclip}{\psframe[fillstyle=solid,linewidth=2pt,linecolor=black,fillcolor=gray4] (0,-1)(7.5,3)}
\pscircle[fillstyle=solid,fillcolor=gray,linecolor=black,linewidth=1.5pt](0,0){0.7}
\pscircle[fillstyle=solid,fillcolor=gray,linecolor=black,linewidth=1.5pt](1.15,0.7){0.6}
\psline[linewidth=8pt,linecolor=white] (0,0) (1.15,0.7)
\pscircle[fillstyle=solid,fillcolor=gray,linecolor=black,linewidth=1.5pt](0.55,1.9){0.7}
\psline[linewidth=8pt,linecolor=white] (1.15,0.7) (0.55,1.9)
\psline[linewidth=8pt,linecolor=white] (-2,1.9) (0.55,1.9)
\pscircle[fillstyle=solid,fillcolor=gray,linecolor=black,linewidth=1.5pt](1,3.2){0.6}
\psline[linewidth=8pt,linecolor=white]  (0.55,1.9) (1,3.2)
\pscircle[fillstyle=solid,fillcolor=gray,linecolor=black,linewidth=1.5pt](2.6,0.5){0.8}
\psline[linewidth=8pt,linecolor=white]  (2.6,0.5) (1.15,0.7)
\pscircle[fillstyle=solid,fillcolor=gray,linecolor=black,linewidth=1.5pt](2.8,1.95){0.6}
\psline[linewidth=8pt,linecolor=white]  (2.6,0.5) (2.8,1.95)
\pscircle[fillstyle=solid,fillcolor=gray,linecolor=black,linewidth=1.5pt](2.4,3.1){0.5}
\psline[linewidth=8pt,linecolor=white] (2.4,3.1) (2.8,1.95)
\pscircle[fillstyle=solid,fillcolor=gray,linecolor=black,linewidth=1.5pt](3.2,-0.85){0.6}
\psline[linewidth=8pt,linecolor=white] (2.6,0.5) (3.2,-0.85)
\psline[linewidth=8pt,linecolor=white]  (3.2,-0.85) (3.2,-1)
\pscircle[fillstyle=solid,fillcolor=gray,linecolor=black,linewidth=1.5pt](4.20,0.2){0.7}
\psline[linewidth=8pt,linecolor=white]  (4.20,0.2) (3.2,-0.85)
\psline[linewidth=8pt,linecolor=white]  (4.20,0.2) (2.6,0.5)
\pscircle[fillstyle=solid,fillcolor=gray,linecolor=black,linewidth=1.5pt](4.15,1.65){0.7}
\psline[linewidth=8pt,linecolor=white]  (4.15,1.65) (4.20,0.2)
\psline[linewidth=8pt,linecolor=white]  (4.15,1.65)  (2.8,1.95)
\pscircle[fillstyle=solid,fillcolor=gray,linecolor=black,linewidth=1.5pt](5.71,0.95){0.9}
\psline[linewidth=8pt,linecolor=white]  (5.71,0.95)  (4.15,1.65)
\psline[linewidth=8pt,linecolor=white]  (5.71,0.95)  (4.20,0.2)
\pscircle[fillstyle=solid,fillcolor=gray,linecolor=black,linewidth=1.5pt](5.25,2.45){0.6}
\psline[linewidth=8pt,linecolor=white]  (5.25,2.45)  (5.71,0.95)
\psline[linewidth=8pt,linecolor=white]  (5.25,2.45)  (4.15,1.65)
\psline[linewidth=8pt,linecolor=white]  (5.25,2.45)  (5.65,4)
\pscircle[fillstyle=solid,fillcolor=gray,linecolor=black,linewidth=1.5pt](5.9,-1){0.65} 
\pscircle[fillstyle=solid,fillcolor=gray,linecolor=black,linewidth=1.5pt](7.2,0.5){0.6} 
\psline[linewidth=8pt,linecolor=white]  (7.2,0.5) (5.71,0.95)
\psline[linewidth=8pt,linecolor=white]  (7.2,0.5) (8.8,0)
\pscircle[fillstyle=solid,fillcolor=gray,linecolor=black,linewidth=1.5pt](6.52,2.3){0.6} 
\psline[linewidth=8pt,linecolor=white]  (6.52,2.3) (5.25,2.45)
\psline[linewidth=8pt,linecolor=white]  (6.52,2.3) (5.71,0.95)
\pscircle[fillstyle=solid,fillcolor=gray,linecolor=black,linewidth=1.5pt](7.55,2.9){0.5}
\psline[linewidth=8pt,linecolor=white]  (7.55,2.9) (6.52,2.3)
\end{psclip}
\psframe[linewidth=1.4pt,linecolor=black] (0,-1)(7.5,3)
\end{pspicture}
    \caption{Spherical set up. The graph obtained with the whites lines is isomorphic to the multigraph of inclusions.}
    \label{Multigraph}
\end{figure}
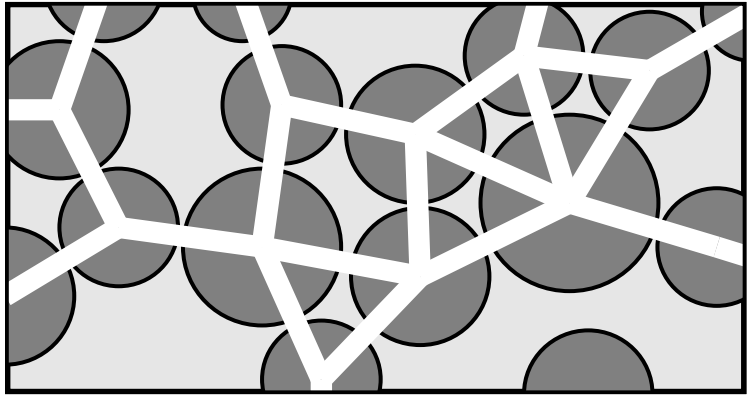

\subsection{Short of inclusions}
A last notion we need to explain before stating  our main results is the notion of short of inclusions. It is directly inspired from the study of electrical networks and their associated multigraphs:  

\begin{definition} {(\bf Short of a multigraph)}

\sspace
Let $G = (V,E)$ a multigraph.  Let $I,J \in V$. The short of $G$ at $\{I,J\}$  is the multigraph obtained by identifying nodes $I$ and $J$ (and suppressing all edges joining $I$ and $J$).  
 More generally, given a set  of pairs of nodes $S$, the short of $G$  at $S$ is the multigraph obtained from $G$ by identifying all pairs in $S$ (and suppressing all edges joining these nodes): 
 
 \sspace
 We say that $G'$ is a short of $G$ if there exists  a set $S$  of pairs of nodes  such that $G'$ is the short of $G$ at $S$. 
 \end{definition}

\mspace
For electrical circuits,   it is well-known that nodes with the same potential can be shorted, without changing the values of the currents through the remaining edges. This is a useful fact, as one can short the multigraph of an electrical circuit to simplify calculations.  In other words, shorting two nodes is the same as imposing the same potential on each of the node, which  can be interpreted as having zero resistance between those nodes.  We refer to \cite[chapter 2]{Bollobas} for more. 

\mspace
At the level of the inclusions, an analogue of a short between two nodes  consists in bridging the gap between two inclusions, so as to obtain one single connected component out of the two, with a single potential. 
We introduce the following 
\begin{definition} {(\bf Short of inclusions)} \label{defi:short_inclusions}

\sspace
Let $F$ a closed set  satisfying (G1)-(G2) and  $Gr(F)$ its multigraph of inclusions, see Definition \ref{defi:multigraph}.  We say that a closed set $F' \supset F$ is a short of $F$ if 
\begin{itemize}
\item $F'$ satisfies (G1)-(G2) (with some  constants $d', \delta', a', N'_{I'}$ instead of $d, \delta, a,N_{I}$)
\item $\mathrm{Ed}(F') \subset \mathrm{Ed}(F)$ 
\item there exists $\eta > 0$, such that  for all  $e \in \mathrm{Ed}(F) \setminus  \mathrm{Ed}(F')$, one has 
$$\{ d(x,e) < \eta\} \subset  F' \subset F \cup \cup_{e \in \mathrm{Ed}(F) \setminus  \mathrm{Ed}(F')} \{ d(x,e) < 2 \eta\}.$$
\end{itemize}

\sspace
Furthermore, if $F$ is an admissible set of inclusions, see Definition  \ref{rem_G1G2}, $F'$ is an admissible short  of $F$ if it is itself an admissible set, and a short of $F$ almost surely for some deterministic $\eta$. \end{definition}
 \noindent
As explained above, a short of $F$ is obtained by filling some gaps of $F$. See Figure  \ref{fig:Short of a multigraph}.  If $F'$ is a short of $F$, $Gr(F')$ is isomorphic to a short of $Gr(F)$. 
\definecolor{gray1}{gray}{0.85}
\definecolor{gray2}{gray}{0.67}
\definecolor{gray3}{gray}{0.55}
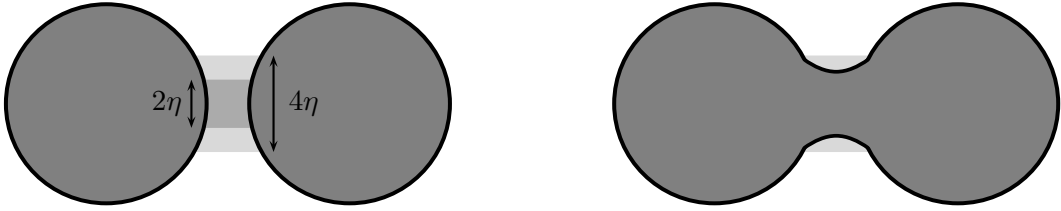
\begin{figure}[h!]
    \centering
    \psset{unit=0.4cm}
\begin{pspicture}(-4,-4)(-17,4)

\psframe*[linecolor=gray1](-4,-1.6) (4,1.6)
\psccurve[fillstyle=solid,fillcolor=gray,linewidth=2.8pt,linecolor=black](-4,-3) (0,-1) (4,-3) (4,3) (0,1) (-4,3) (-4,-3)

\pscircle[fillstyle=solid,fillcolor=gray,linewidth=1.5pt,linecolor=black](-4,0){3.3}

\pscircle[fillstyle=solid,fillcolor=gray,linewidth=1.5pt,linecolor=black](4,0){3.3}
\psccurve*[linecolor=gray](-4,-3) (0,-1) (4,-3) (4,3) (0,1) (-4,3) (-4,-3)

\psframe*[linecolor=gray1](-24,-1.6) (-16,1.6)
\psframe*[linecolor=gray2](-24,-0.8) (-16,0.8)
\pscircle[fillstyle=solid,fillcolor=gray,linewidth=1.5pt,linecolor=black](-24,0){3.3}
\pscircle[fillstyle=solid,fillcolor=gray,linewidth=1.5pt,linecolor=black](-16,0){3.3}

\psline{<->}(-18.5,-1.6) (-18.5,1.6)
\psline{<->}(-21.2,-0.8) (-21.2,0.8)
\rput(-22,0){$2\eta$}
\rput(-17.5,0){$4\eta$}

\end{pspicture}
    \caption{Two inclusions configuration, shorted on the right.}
    \label{fig:Short of a multigraph}
\end{figure}


\subsection{Homogenization results}
We now turn to the homogenization problem described roughly in the introduction. For the rest of this section, we assume that $F = F(\omega) = \cup_{I \in \mathrm{CC}(F(\omega))} I(\omega)$ is an admissible set of inclusions, see Definition \ref{rem_G1G2}.  Given a parameter $\eps > 0$, $\delta_0 > 0$  and a bounded domain $U$,  we denote 
\begin{equation} \label{defi:rescaled} 
 \quad I^\eps := \eps I \quad  \forall I \in \mathrm{CC}(F),     \quad  F^\eps := \bigcup_{\substack{I^\eps \subset U, \\ d(I^\eps, \pa U) > \eps \delta_0}} I^\eps . 
 \end{equation}
Let $f \in L^{6/5}(U)$,  and $u^\eps = u^\eps(\omega) \in H^1_0(U)$ the solution of \eqref{main_system}. Existence and uniqueness of $u^\eps$ is standard, and comes with the estimate 
$$ \| \na u^{\eps} \|_{L^2(U)} \le C \|f\|_{L^{6/5}(U)}.$$
In particular, $u^\eps$ has a subsequence that  weakly converges in $H^1_0(U)$. The point is to identify the limit $u^0$, and to determine under which conditions it satisfies a system like \eqref{target_system}, where $\lambda = \E 1_F$ is the average density of the set of inclusions.  

\mspace
As mentioned in the introduction, Zhikov has tackled the homogenization problem, for a collection of random spheres of unit radius, under a logarithmic moment bound  on  the minimal  distance between the spheres. The result extends easily to admissible sets of inclusions, under a uniform bound on their diameter. It can further  be expressed in terms of  multigraphs of inclusions. Namely, for all $N > 0$, let
\begin{equation} \label{defi:FN}
Q_N := (-N,N)^3, \quad    F_N := \bigcup_{\substack{I \in \mathrm{CC}(F), \\ I \subset Q_N}} I .
\end{equation}
Then  the homogenization theorem of Zhikov reads 
\begin{theorem} {\bf (Zhikov, \cite{book:ZKO})} \label{Zhikov} 

\sspace
Let $F$ an admissible set of inclusions satisfying almost surely: $\sup_{I \in \mathrm{CC}(F)} \:  \diam(I) < +\infty$, and  such that
\begin{equation} \label{assumption_Jikov_type}
\limsup_{N \rightarrow +\infty} \frac{1}{|Q_N|}  \sum_{e \in \mathrm{Ed}(F_N)}   \mu_e < +\infty 
\end{equation}
Then, almost surely,  the whole sequence $u^\eps$ converges weakly in $H^1_0(U)$ to the solution $u^0$ of a system of type \eqref{target_system}.  
\end{theorem}

\begin{remark} \label{rem_Antoine}
As pointed out to us by A. Gloria, it can be seen from  the ergodic theorem that: 
$\text{ $\sup_{I \in \mathrm{CC}(F)} \:  \diam(I) < +\infty$ a.s. $\Rightarrow$   $\sup_{I \in \mathrm{CC}(F)} \:  \diam(I) \le D$ a.s. for some deterministic  $D$.}$
\end{remark}

\noindent
As usual in stochastic  homogenization, the effective conductivity matrix $A^0$ can be expressed in terms of a variational problem in probability. Following  \cite[chapter 8]{book:ZKO}, we first introduce the effective resistance matrix $B^0$ :  it is the symmetric matrix defined by 
$$ \forall \xi \in \R^3, \quad B^0 \xi \cdot \xi := \inf_z \E \int_{Q_1 \setminus F} |\xi + z|^2, $$
where the infimum is taken over the set of vector fields $z = z(y,\omega) \in L^2(\Omega, L^2_{loc}(\R^3))$ satisfying 
\begin{itemize}
\item[i)] $z$ solenoidal, that is ${\rm div}_y \, z = 0$    
\item[ii)] $z$ stationary, that is $z(y+y',\omega)= z(y,\tau_y'(\omega))$ 
\item[iii)] $\E z = 0$. 
\end{itemize}
Then, part of the proof of Theorem \ref{Zhikov} is to show that $B^0$ is invertible, and that the effective conductivity matrix is given by 
\begin{equation} \label{defi_Ahom}
 A^0 = (B^0)^{-1}.   
\end{equation}

\mspace
Our goal is to relax the assumption in the previous theorem, notably to identify  configurations for which no condition on the minimal distance is needed. Our criterion for homogenization will be expressed again through multigraphs of inclusions. A key role will be played by the following discrete energy functional: for  $F$ a closed set  satisfying (G1)-(G2),  and two families 
$\{u_I\}$, indexed by $I \in \mathrm{CC}(F)$, and $\{b_{IJe}\}$ indexed by triplets with $I,J \in  \mathrm{CC}(F)$, and $e \in \mathrm{Ed}(F)$ s.t.  $I \overset{e}{\leftrightarrow} J$, we introduce the energy functional 
\begin{equation} \label{defi:En}
\mE\Big(F, \{u_I\}, \{b_{IJe}\}\Big) =   \sum_{I,J \in \mathrm{CC}(F)} \sum_{\substack{e \in \mathrm{Ed}(F), \\ I \overset{e}{\leftrightarrow} J}}  \mu_e  |b_{IJe} - b_{JIe} + u_I - u_J |^2 + \sum_{I \in \mathrm{CC}(F)}| I | \,  |u_I|^2.
\end{equation}

\mspace
We start with the definition of the homogenized matrix. We remind the notation  $I_{0,F}$ from Remark \ref{rem_I0_C0}. We denote $x_S$ the center of mass of a subset $S$ of $\R^3$. 
\begin{proposition} {\bf (Existence of the homogenized matrix)} \label{prop_corrector}

\sspace
Let $F$ an admissible set of inclusions. If  $\E \,  \diam(I_{0,F})^2 < +\infty$  and if almost surely 
\begin{equation}  \label{H1} \tag{H1}
\limsup_{N \rightarrow +\infty} \:   \inf_{\{u_I\}} \: \frac{1}{|Q_N|} \mE\Big(F_N, \{u_I\}, \{b_{IJe}  = \xi \cdot x_I\} \Big)  < +\infty 
\end{equation}
where $\mE$ is given in \eqref{defi:En}, and $F_N$ in \eqref{defi:FN}, then the matrix $A^0$ introduced in \eqref{defi_Ahom} is well-defined. 
 \end{proposition}

\mspace
Compared to the assumptions in Theorem \ref{Zhikov}, those of Proposition \ref{prop_corrector} are better in two regards. First,  the uniform bound on the diameter of the inclusions (see Remark \ref{rem_Antoine})  is replaced by a moment bound. This possibility of considering inclusions of arbitrary large size, interesting in its own, will turn very useful when combined to our homogenization result involving shorts, {\it cf.} Theorem \ref{thm_homog}. Second, as mentioned in the introduction, \eqref{H1} is weaker than  \eqref{assumption_Jikov_type}. Indeed, for inclusions satisfying  $\sup_{I \in \mathrm{CC}(F)} \:  \diam(I) < +\infty$, the latter clearly implies 
$$ \limsup_{N \rightarrow +\infty}\frac{1}{|Q_N|}  \mE(F_N, \{0\}, \{\xi \cdot x_I\}) < +\infty.$$
 Further discussion of \eqref{H1} will be provided in Section \ref{section_disc_assumptions}, notably its relation to the graph Laplacian or to the subadditive ergodic theorem. We will also prove that \eqref{H1} is satisfied by sets $F$  satisfying  the moment bound 
$$ \E  \, \diam(C_{0,F})^2 < +\infty $$
where $C_{0,F}$ is the cluster of $F$ containing $0$, {\it cf. }Remark \ref{rem_I0_C0}.

 \mspace
It is an interesting open problem to be able to perform homogenization under the mere assumption \eqref{H1}. We must here strengthen it a little.

\begin{theorem} {(\bf Homogenization of stiff inclusions)} \label{thm_homog}

\sspace
Let $F$ an admissible set of inclusions. We assume that there exists an admissible short  $F'$ of $F$ and  some $s  \in (3,6)$ ($s  \in (3,\infty)$  in the case $F' = F$)   such that, almost surely:
\begin{equation}  \label{H2} \tag{H2}
\limsup_{N \rightarrow +\infty} \:    \:   \frac{1}{|Q_N|}\sup_{\{b_{IJe}\}}   \inf_{\{u_I\}} \: \mE\Big(F'_N, \{u_I\}, \{b_{IJe}\}\Big) / \Big( \frac{1}{|Q_N|} \sum_{I,J \in \mathrm{CC}(F'_N)} \sum_{\substack{e \in \mathrm{Ed}(F'_N), \\ I \overset{e}{\leftrightarrow} J}} |b_{IJe}|^s  \Big)^{2/s} < +\infty 
\end{equation}
Then,  there exists $p = p(s)$ such that under the additional condition: 
\begin{equation}  \label{moment_bound_F}
\E \, \diam(I_{0,F'})^p < +\infty, 
\end{equation}
the solution $u^\eps$ of \eqref{main_system} converges a.s. weakly in $H^1_0(U)$ to the solution $u^0$ of \eqref{target_system}, with $A^0$ defined by \eqref{defi_Ahom}. 
\end{theorem}

%

\mspace
Here are a few remarks, to be complemented in Section \ref{section_disc_assumptions}.

\mspace
i) Of course, in previous statements, it is enough that  all assumptions involving the $\delta$-multigraph of $F$ hold for {\em some} $\delta > 0$. In practice, one should take $\delta$  small, to have a reduced number of edges in the multigraph.

\mspace
ii) In practice, the following reformulation of \eqref{H2}  will be used:

\sspace
One can find  almost surely an $M = M(\omega) >0$ satisfying:  for all $N>0$,  for any family $\{b_{IJe} \}$ indexed by $I,J,e \in \CC(F_N)^2 \times \Ed(F_N)$ with $I \overset{e}{\leftrightarrow} J$, there exists a family $\{ u_{I} \}_{I \in CC(F_N)}$ such that 
\begin{equation} \label{reformulation_H2}
\mE\Big(F_N, \{u_I\}, \{b_{IJe}\}\Big)   \leq M |Q_N|\Big( \frac{1}{|Q_N|} \sum_{I,J \in \CC(F_N)} \sum_{\substack{e \in \Ed(F_N), \\ I \overset{e}{\leftrightarrow} J}} |b_{IJe}|^s  \Big)^{2/s}.
\end{equation}

\mspace
iii)  Let $F$ an admissible set of inclusions, $F'$ an admissible short of $F$. If 
$$ \E \diam(I_{0,F'})^s < +\infty $$
then \eqref{H2} implies \eqref{H1}.   See Lemma \ref{lem_H2_implies_H1}.  In particular, by Proposition \ref{prop_corrector}, the matrix $A^0$ is well-defined

\mspace
iv)  \eqref{H2} is implied by the following logarithmic moment bound, see Lemma  \ref{lemma_logH2}: 
\begin{equation} \label{log_bound_mue}
  \limsup_{N \rightarrow +\infty} \frac{1}{|Q_N|} \sum_{e \in \mathrm{Ed}(F_N)}   \mu_e^{k} < +\infty, \quad k = \big(\frac{s}{2}\big)' = \frac{s}{s-2}. 
 \end{equation}
In the case where $F' = F$,  $s$ can be taken arbitrarily large in Theorem \ref{thm_homog}, so that this sufficient condition is almost the standard one (one can take any $k >1$). For a general short $F'$, we are limited to $s < 6$ for technical reasons, hence to $k > \frac{3}{2}$.  


\mspace
iv) A corollary  of Theorem \ref{thm_homog},  to be established in Section  \ref{section_disc_assumptions} and illustrated in Figure \ref{exemple_2}, is the following: 

\begin{corollary} \label{corollary_cyclefree_new}
Let $F$ an admissible set of inclusions with a.s. $\sup_{I \in \CC(F)} \diam(I) < +\infty$. Assume that  $Gr(F)$ is cycle-free, and that 
$$\E \, (\sharp C_{0,F})^p < +\infty,  \quad \text{ for some } \: p > 2,$$
where $C_{0,F}$ is the cluster of $F$ containing $0$, {\it cf. }Remark \ref{rem_I0_C0}.   Then, \eqref{H2} is satisfied with $F' = F$ and  $s=\frac{2p}{p-2}$, so that homogenization holds. 
\end{corollary}

\definecolor{gray1}{gray}{0.85}
\definecolor{gray2}{gray}{0.67}
\definecolor{gray3}{gray}{0.55}

\begin{figure}[h!]
    \centering
    \psset{unit=0.85cm}
\begin{pspicture}(-5,-4)(4,2.5)
\begin{psclip}{\psframe[fillstyle=solid,linewidth=1.5pt,linecolor=black,fillcolor=gray4](-7,-4)(5.5,2.5)}
\pscircle[fillstyle=solid,fillcolor=gray,linecolor=black,linewidth=1pt](1,0){0.20}
\pscircle[fillstyle=solid,fillcolor=gray,linecolor=black,linewidth=1pt](1.33,0.35){0.20}
\pscircle[fillstyle=solid,fillcolor=gray,linecolor=black,linewidth=1pt](1.66,0.70){0.20}
\pscircle[fillstyle=solid,fillcolor=gray,linecolor=black,linewidth=1pt](1.99,1.05){0.20}
\pscircle[fillstyle=solid,fillcolor=gray,linecolor=black,linewidth=1pt](2.33,1.40){0.20}
\pscircle[fillstyle=solid,fillcolor=gray,linecolor=black,linewidth=1pt](2.65,1.75){0.20}
\pscircle[fillstyle=solid,fillcolor=gray,linecolor=black,linewidth=1pt](-5,-2.05){0.20}
\pscircle[fillstyle=solid,fillcolor=gray,linecolor=black,linewidth=1pt](-5.35,-2.35){0.20}
\pscircle[fillstyle=solid,fillcolor=gray,linecolor=black,linewidth=1pt](-5.80,-2.35){0.20}
\pscircle[fillstyle=solid,fillcolor=gray,linecolor=black,linewidth=1pt](-6.25,-2.35){0.20}
\pscircle[fillstyle=solid,fillcolor=gray,linecolor=black,linewidth=1pt](-6.6,-2.70){0.20}
\pscircle[fillstyle=solid,fillcolor=gray,linecolor=black,linewidth=1pt](-6.6,-3.15){0.20}
\pscircle[fillstyle=solid,fillcolor=gray,linecolor=black,linewidth=1pt](-6.15,-3.35){0.20}
\pscircle[fillstyle=solid,fillcolor=gray,linecolor=black,linewidth=1pt](-5.7,-3.55){0.20}
\pscircle[fillstyle=solid,fillcolor=gray,linecolor=black,linewidth=1pt](-5.25,-3.75){0.20}
\pscircle[fillstyle=solid,fillcolor=gray,linecolor=black,linewidth=1pt](-4.8,-3.95){0.20}
\pscircle[fillstyle=solid,fillcolor=gray,linecolor=black,linewidth=1pt](0,-2){0.20}
\pscircle[fillstyle=solid,fillcolor=gray,linecolor=black,linewidth=1pt](0.30,-1.65){0.20}
\pscircle[fillstyle=solid,fillcolor=gray,linecolor=black,linewidth=1pt](0.6,-1.32){0.20}
\pscircle[fillstyle=solid,fillcolor=gray,linecolor=black,linewidth=1pt](0.9,-1.65){0.20}
\pscircle[fillstyle=solid,fillcolor=gray,linecolor=black,linewidth=1pt](1.2,-2){0.20}
\pscircle[fillstyle=solid,fillcolor=gray,linecolor=black,linewidth=1pt](1.5,-1.65){0.20}
\pscircle[fillstyle=solid,fillcolor=gray,linecolor=black,linewidth=1pt](1.80,-1.32){0.20}
\pscircle[fillstyle=solid,fillcolor=gray,linecolor=black,linewidth=1pt](2.1,-1.65){0.20}
\pscircle[fillstyle=solid,fillcolor=gray,linecolor=black,linewidth=1pt](2.4,-2){0.20}
\pscircle[fillstyle=solid,fillcolor=gray,linecolor=black,linewidth=1pt](2.7,-1.65){0.20}
\pscircle[fillstyle=solid,fillcolor=gray,linecolor=black,linewidth=1pt](3,-1.32){0.20}
\pscircle[fillstyle=solid,fillcolor=gray,linecolor=black,linewidth=1pt](3.3,-1.65){0.20}
\pscircle[fillstyle=solid,fillcolor=gray,linecolor=black,linewidth=1pt](3.6,-2){0.20}
\pscircle[fillstyle=solid,fillcolor=gray,linecolor=black,linewidth=1pt](3.9,-1.65){0.20}
\pscircle[fillstyle=solid,fillcolor=gray,linecolor=black,linewidth=1pt](4.2,-1.32){0.20}
\pscircle[fillstyle=solid,fillcolor=gray,linecolor=black,linewidth=1pt](4.5,-1.65){0.20}
\pscircle[fillstyle=solid,fillcolor=gray,linecolor=black,linewidth=1pt](4.8,-2){0.20}
\pscircle[fillstyle=solid,fillcolor=gray,linecolor=black,linewidth=1pt](-3,1){0.20}
\pscircle[fillstyle=solid,fillcolor=gray,linecolor=black,linewidth=1pt](-3,1.45){0.20}
\pscircle[fillstyle=solid,fillcolor=gray,linecolor=black,linewidth=1pt](-3.4,1.60){0.20}
\pscircle[fillstyle=solid,fillcolor=gray,linecolor=black,linewidth=1pt](-3.8,1.40){0.20}
\pscircle[fillstyle=solid,fillcolor=gray,linecolor=black,linewidth=1pt](-4,1){0.20}
\pscircle[fillstyle=solid,fillcolor=gray,linecolor=black,linewidth=1pt](-3.9,0.6){0.20}
\pscircle[fillstyle=solid,fillcolor=gray,linecolor=black,linewidth=1pt](-3.5,0.45){0.20}
\pscircle[fillstyle=solid,fillcolor=gray,linecolor=black,linewidth=1pt](-3.2,0.1){0.20}
\pscircle[fillstyle=solid,fillcolor=gray,linecolor=black,linewidth=1pt](-1.53,-1){0.20}
\pscircle[fillstyle=solid,fillcolor=gray,linecolor=black,linewidth=1pt](-1.05,-1){0.20}
\pscircle[fillstyle=solid,fillcolor=gray,linecolor=black,linewidth=1pt](-1.9,-1.25){0.20}
\pscircle[fillstyle=solid,fillcolor=gray,linecolor=black,linewidth=1pt](-2,-1.7){0.20}

\pscircle[fillstyle=solid,fillcolor=gray,linecolor=black,linewidth=1pt](-2.5,-4){0.20}
\pscircle[fillstyle=solid,fillcolor=gray,linecolor=black,linewidth=1pt](-2.1,-3.8){0.20}
\pscircle[fillstyle=solid,fillcolor=gray,linecolor=black,linewidth=1pt](-1.75,-3.55){0.20}

\pscircle[fillstyle=solid,fillcolor=gray,linecolor=black,linewidth=1pt](-6,1){0.20}
\pscircle[fillstyle=solid,fillcolor=gray,linecolor=black,linewidth=1pt](-5.55,1.1){0.20}
\end{psclip}
\psframe[linewidth=1.5pt,linecolor=black](-7,-4)(5.5,2.5)
\end{pspicture}
    \caption{Cycle-free configuration set up.}
    \label{exemple_2}
\end{figure}
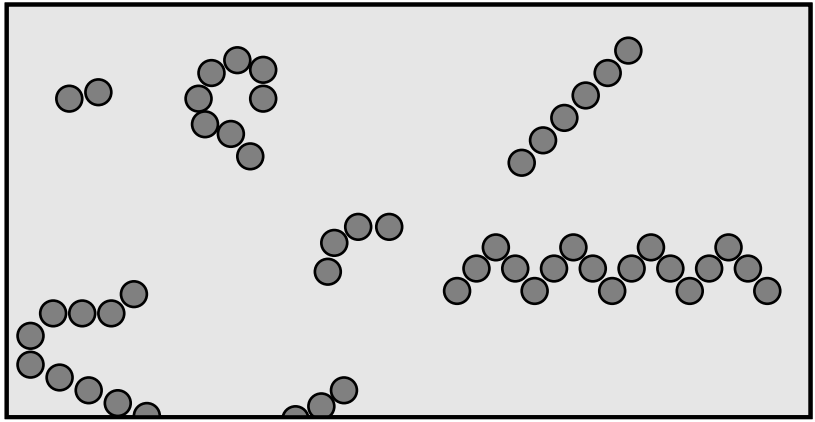

\definecolor{gray4}{gray}{0.90}

\mspace
v) Another important corollary  of Theorem \ref{thm_homog} is:
\begin{corollary} \label{coro_short_clusters}
Let $F$ an admissible set of inclusions.  Let $s \in (3,6)$ and $p=p(s)$ as in Theorem \ref{thm_homog}. If 
$$ \E \, \diam(C_{0,F})^p < +\infty $$
then homogenization holds. 
\end{corollary}
\noindent
Indeed, if we simply bridge all pairs of nodes $I,J$ with $I \leftrightarrow J$, we obtain in this way an admissible short $F'$ of $F$ with $\E \, \diam(I_{0,F'})^p < +\infty$, and with $\Ed(F') = \emptyset$, hence trivially satisfying \eqref{H2}  and  \eqref{moment_bound_F}.  See Figure \ref{exemple_1} for an illustration. 
 Let us point out that the sharper bound 
 $$ \E \,  \exp(\diam(C_{0,F})^\gamma) < +\infty, \quad \forall \gamma > 1 $$
has been shown  to hold for several examples of admissible sets of inclusions, satisfying strong mixing assumptions, below the critical percolation threshold.  We refer to the recent article \cite{DuGloperco} for details, where, moreover, the homogenization of both Laplace and Stokes equation is obtained in this special case by a different method.

\definecolor{gray1}{gray}{0.85}
\definecolor{gray2}{gray}{0.67}
\definecolor{gray3}{gray}{0.55}

\begin{figure}[h!]
    \centering
    \psset{unit=0.85cm}
\begin{pspicture}(10,-2)(3,5)

\pscircle[linewidth=1pt,fillstyle=solid,
fillcolor=gray](0.05,0.2) {0.25}
\pscircle[linewidth=1pt,fillstyle=solid,
fillcolor=gray](1,0.75) {0.25}
\pscircle[linewidth=1pt,fillstyle=solid,
fillcolor=gray](0.45,0.60) {0.25}
\pscircle[linewidth=1pt,fillstyle=solid,
fillcolor=gray](-0.5,-1) {0.25}
\pscircle[linewidth=1pt,fillstyle=solid,
fillcolor=gray](0.08,-1) {0.25}
\psline[linewidth=1.5pt,linecolor=white] (1,0.75)  (0.45,0.60)
\psline[linewidth=1.5pt,linecolor=white] (0.05,0.2)  (0.45,0.60)
\psline[linewidth=1.5pt,linecolor=white] (0.08,-1) (-0.5,-1) 
\pscircle[linewidth=1pt,fillstyle=solid,
fillcolor=gray](2.65,-0.3) {0.25}
\pscircle[linewidth=1pt,fillstyle=solid,
fillcolor=gray](3.15,-0.5) {0.25}
\pscircle[linewidth=1pt,fillstyle=solid,
fillcolor=gray](3.8,0.9) {0.25}
\pscircle[linewidth=1pt,fillstyle=solid,
fillcolor=gray](0.05,3.2) {0.25}
\pscircle[linewidth=1pt,fillstyle=solid,
fillcolor=gray](0.1,2.65) {0.25}
\pscircle[linewidth=1pt,fillstyle=solid,
fillcolor=gray](-0.40,2.87) {0.25}
\pscircle[linewidth=1pt,fillstyle=solid,
fillcolor=gray](-0.46,3.41) {0.25}
\pscircle[linewidth=1pt,fillstyle=solid,
fillcolor=gray](3.3,3.3) {0.25}
\pscircle[linewidth=1pt,fillstyle=solid,
fillcolor=gray](3.7,2.9) {0.25}
\pscircle[linewidth=1pt,fillstyle=solid,
fillcolor=gray](3.85,3.45) {0.25}
\pscircle[linewidth=1pt,fillstyle=solid,
fillcolor=gray](2.75,3.3) {0.25}

\pscircle[linewidth=1pt,fillstyle=solid,
fillcolor=gray](9.55,0.2) {0.25}
\pscircle[linewidth=1pt,fillstyle=solid,
fillcolor=gray](10.5,0.75) {0.25}
\pscircle[linewidth=1pt,fillstyle=solid,
fillcolor=gray](9.95,0.60) {0.25}
\pscircle[linewidth=1pt,fillstyle=solid,
fillcolor=gray](9,-1) {0.25}
\pscircle[linewidth=1pt,fillstyle=solid,
fillcolor=gray](9.58,-1) {0.25}

\pscircle[linewidth=1pt,fillstyle=solid,
fillcolor=gray](12.15,-0.3) {0.25}
\pscircle[linewidth=1pt,fillstyle=solid,
fillcolor=gray](12.65,-0.5) {0.25}
\pscircle[linewidth=1pt,fillstyle=solid,
fillcolor=gray](13.3,0.9) {0.25}
\psline[linewidth=1.5pt,linecolor=white] (2.65,-0.3) (3.15,-0.5) 
\pscircle[linewidth=1pt,fillstyle=solid,
fillcolor=gray](9.55,3.2) {0.25}
\pscircle[linewidth=1pt,fillstyle=solid,
fillcolor=gray](9.6,2.65) {0.25}
\pscircle[linewidth=1pt,fillstyle=solid,
fillcolor=gray](9.1,2.87) {0.25}
\pscircle[linewidth=1pt,fillstyle=solid,
fillcolor=gray](9.04,3.41) {0.25}
\psline[linewidth=1.5pt,linecolor=white] (0.05,3.2) (0.1,2.65)
\psline[linewidth=1.5pt,linecolor=white] (0.05,3.2) (-0.46,3.41)
\psline[linewidth=1.5pt,linecolor=white] (-0.46,3.41) (-0.4,2.87) 
\psline[linewidth=1.5pt,linecolor=white] (-0.4,2.87) (0.1,2.65)
\pscircle[linewidth=1pt,fillstyle=solid,
fillcolor=gray](12.8,3.3) {0.25}
\pscircle[linewidth=1pt,fillstyle=solid,
fillcolor=gray](13.2,2.9) {0.25}
\pscircle[linewidth=1pt,fillstyle=solid,
fillcolor=gray](13.35,3.45) {0.25}
\pscircle[linewidth=1pt,fillstyle=solid,
fillcolor=gray](12.25,3.3) {0.25}
\psline[linewidth=1.5pt,linecolor=white] (3.3,3.3) (3.7,2.9)
\psline[linewidth=1.5pt,linecolor=white] (3.7,2.9) (3.85,3.45)
\psline[linewidth=1.5pt,linecolor=white] (3.3,3.3) (3.85,3.45)
\psline[linewidth=1.5pt,linecolor=white] (3.3,3.3) (2.75,3.3)

\pscircle[linewidth=1pt,fillstyle=solid,
fillcolor=gray](9.55,0.2) {0.25}
\pscircle[linewidth=1pt,fillstyle=solid,
fillcolor=gray](10.5,0.75) {0.25}
\pscircle[linewidth=1pt,fillstyle=solid,
fillcolor=gray](9.95,0.60) {0.25}
\pscircle[linewidth=1pt,fillstyle=solid,
fillcolor=gray](9,-1) {0.25}
\pscircle[linewidth=1pt,fillstyle=solid,
fillcolor=gray](9.58,-1) {0.25}
\psline[linewidth=2pt,linecolor=gray] (10.5,0.75)  (9.95,0.60)
\psline[linewidth=2pt,linecolor=gray] (9.55,0.2)  (9.95,0.60)
\psline[linewidth=2pt,linecolor=gray] (9.58,-1) (9,-1) 
\pscircle[linewidth=1pt,fillstyle=solid,
fillcolor=gray](12.15,-0.3) {0.25}
\pscircle[linewidth=1pt,fillstyle=solid,
fillcolor=gray](12.65,-0.5) {0.25}
\pscircle[linewidth=1pt,fillstyle=solid,
fillcolor=gray](13.3,0.9) {0.25}
\psline[linewidth=2pt,linecolor=gray] (12.15,-0.3) (12.65,-0.5) 
\pscircle[linewidth=1pt,fillstyle=solid,
fillcolor=gray](9.55,3.2) {0.25}
\pscircle[linewidth=1pt,fillstyle=solid,
fillcolor=gray](9.6,2.65) {0.25}
\pscircle[linewidth=1pt,fillstyle=solid,
fillcolor=gray](9.1,2.87) {0.25}
\pscircle[linewidth=1pt,fillstyle=solid,
fillcolor=gray](9.04,3.41) {0.25}
\psline[linewidth=2pt,linecolor=gray] (9.55,3.2) (9.6,2.65)
\psline[linewidth=2pt,linecolor=gray] (9.55,3.2) (9.04,3.41)
\psline[linewidth=2pt,linecolor=gray] (9.04,3.41) (9.1,2.87) 
\psline[linewidth=2pt,linecolor=gray] (9.1,2.87) (9.6,2.65)
\pscircle[linewidth=1pt,fillstyle=solid,
fillcolor=gray](12.8,3.3) {0.25}
\pscircle[linewidth=1pt,fillstyle=solid,
fillcolor=gray](13.2,2.9) {0.25}
\pscircle[linewidth=1pt,fillstyle=solid,
fillcolor=gray](13.35,3.45) {0.25}
\pscircle[linewidth=1pt,fillstyle=solid,
fillcolor=gray](12.25,3.3) {0.25}
\psline[linewidth=2pt,linecolor=gray] (12.8,3.3) (13.2,2.9)
\psline[linewidth=2pt,linecolor=gray] (13.2,2.9) (13.35,3.45)
\psline[linewidth=2pt,linecolor=gray] (12.8,3.3) (13.35,3.45)
\psline[linewidth=2pt,linecolor=gray] (12.8,3.3) (12.25,3.3)

\psline[linewidth=1pt, linestyle= dashed] (-1.5,-1.5)(4.5,-1.5)
\psline[linewidth=1pt, linestyle= dashed] (-1.5,4.5)(4.5,4.5)
\psline[linewidth=1pt, linestyle= dashed] (4.5,-1.5)(4.5,4.5)
\psline[linewidth=1pt, linestyle= dashed] (-1.5,-1.5)(-1.5,4.5)

\psline[linewidth=1pt, linestyle= dashed] (8,-1.5)(14,-1.5)
\psline[linewidth=1pt, linestyle= dashed] (8,4.5)(14,4.5)
\psline[linewidth=1pt, linestyle= dashed] (14,-1.5)(14,4.5)
\psline[linewidth=1pt, linestyle= dashed] (8,-1.5)(8,4.5)
\end{pspicture} 
    \caption{On the left, all clusters are far away from the others. On the right, groups of inclusions joined by a grey line  form a short $F'$  that verifies $\Ed(F') = \emptyset$.}
    \label{exemple_1}
\end{figure}
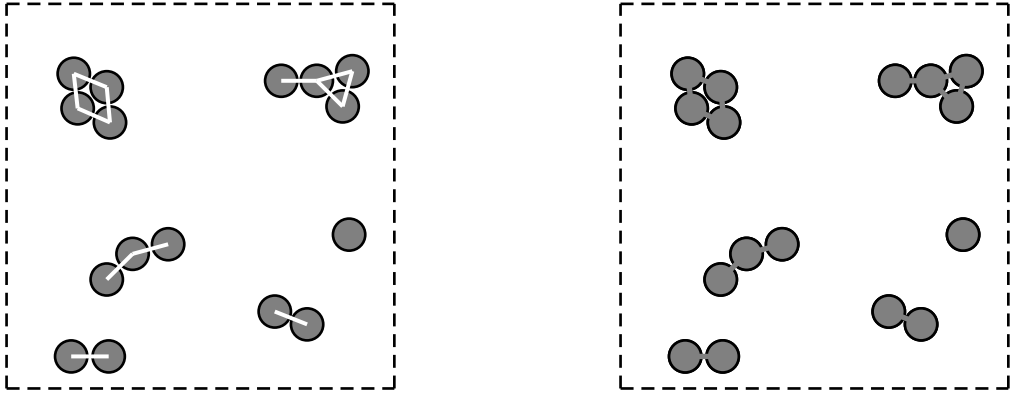


\section{Preliminary extension theorems} \label{section_preliminary}
In this whole section, $F$ is an admissible set of inclusions, $U$ is a bounded domain, and  $F^\eps$ is defined as in \eqref{defi:rescaled}. 
We will show how assumptions \eqref{H1} or \eqref{H2} allow to construct suitable extensions of $H^1$ fields given in the inclusions, that is in $F^\eps$,  resp. divergence-free $L^2$ fields given outside the inclusions, that is in $U\setminus F^\eps$. By suitable, we mean that the extension operator will be bounded uniformly in $\eps$. These extension results will be central to the homogenization process. 

\subsection{Extension outside the inclusions}
The two main results of this paragraph are 
\begin{proposition}  \label{prop_extension_H1}
Let $\xi \in \R^3$. Assume  \eqref{H1}. If  $\E \, \diam(I_{0,F})^2 < +\infty$,   
one can find almost surely  $C > 0$ independent of $\eps$ and  a field $\phi^\eps \in H^1_0(U)$ with 
$$ \na \phi^\eps = \xi \quad \text{ in } \:  F^\eps, \quad \eps^{-1} \|\phi^\eps\|_{L^2(U)} +  \|\na  \phi^\eps \|_{L^2(U)} \le C |\xi|.  $$
\end{proposition}

\begin{proposition}  \label{prop_extension_H2}
Assume  \eqref{H2} with $F' = F$ and $s > 3$. For any $\tilde s >  s$,  there exists $p = p(\tilde{s})$ such that  if 
$$\: \E \, \diam(I_{0,F})^p  < +\infty,$$ 
one can find almost surely  $C > 0$ independent of $\eps$  satisfying:  for all $\varphi^\eps \in W^{1,\tilde{s}}(F^\eps)$, there exists  a field  $\phi^\eps \in H^1_0(U)$ with
$$ \na \phi^\eps = \na  \varphi^\eps \quad \text{ in } \:  F^\eps, \quad \eps^{-1} \|\phi^\eps\|_{L^2(U)} + \|\na  \phi^\eps \|_{L^2(U)} \le C \|\na \varphi^\eps \|_{L^{\tilde{s}}(F^\eps)}  $$
\end{proposition}

\noindent
We will focus on the proof of the latter proposition, as the former requires only minor modifications. 

\mspace
\begin{proof}
In all the proof, the realization $\omega$ is fixed and does not show up in the notations. We will use the following Poincaré-Wirtinger inequality: there exists $r >0$ such that for all  $t \in [1,+\infty)$, one can find $C > 0$ satisfying: 
\begin{equation} \label{PW}
 \forall I \in \CC(F), \quad  \forall u \in W^{1,t}(I), \quad  \|u - (u)_I\|_{L^t(I)}  \le C \, \diam(I)^r \|\na u \|_{L^t(I)},
\end{equation}
with $(u)_I = \frac{1}{|I|} \int_I u$. A main point here is that the constant is bounded by a power of the diameter. This inequality is known to be true with $r=1$ for convex domains,  and more generally for star-shaped domains, see \cite[chapter 12]{Leoni}. We indicate in Appendix \ref{appendix_PW} how to show that this inequality holds with $r = 12$. Obviously, in the special case where $u(x) = \xi \cdot x$, relevant to Proposition \ref{prop_extension_H1}, the exponent $r=1$ is enough.

\mspace
Let 
$$N := \eps^{-1}, \quad \tilde{F}_N := N F^\eps= \bigcup_{\substack{I \subset F \cap N U, \\ d(I,N\pa U)>\delta_0}} I$$ 
the set of all the connected components of $F$ included in $N U$ and $\delta_0$ far from the boundary. By a scaling argument, it is enough to prove that there exists a constant $C>0$ independent of $N$ such that for any $\varphi \in W^{1,\tilde{s}}(\tilde{F}_N)$, one can find $\phi_N \in H^1_0(N U)$ with
\begin{equation} \label{rescaled_inequality_extension}
\na \phi_N= \na \varphi \text{ in } \:  \tilde{F}_N, \quad ||\phi_N||_{H^1(N U)} \leq C N^{\frac{3\tilde{s}-6}{2\tilde{s}}} ||\na \varphi||_{L^{\tilde{s}}(\tilde{F}_N)} .
\end{equation}
Let $\varphi \in W^{1,\tilde{s}}(\tilde{F}_N)$ and denote $\varphi_I:=\varphi|_{I}$ for any $I \in \CC(\tilde{F}_N)$. The proof is split into two main steps: 

\begin{enumerate}
    \item Given an arbitrary family $(u_I)_{I  \in \CC(\tilde{F}_N)}$, and any    $e=[x_{I,\alpha}, x_{J,\beta}] \in \Ed(\tilde{F}_N)$, we build a local extension $\phi_e$, defined  in a neighborhood $V_e$ of $e$ that contains $I_\alpha$ and $J_\beta$,  such that 
\begin{equation} \label{prop_phi_e_1}
\phi_e\vert_I  = \varphi_I - \dashint_I \varphi_I  + u_I, \quad \phi_e\vert_J = \varphi_J - \dashint_J \varphi_J + u_J 
\end{equation}
\begin{equation} \label{prop_phi_e_2}
\begin{aligned}
\| \phi_e\|^2_{H^1(V_e)} &  \le C \mu_e |\varphi_I(x_{I,\alpha}) - \dashint_I \varphi_I  - \varphi_J(x_{J,\beta}) +  \dashint_J \varphi_J  +  u_I - u_J|^2 + C ( |u_I|^2 + |u_J|^2)  \\ 
&  + C (\diam(I)^{2r} \|\na \varphi_I\|^2_{L^2(I)}  +  \diam(J)^{2r}  \|\na \varphi_J\|_{L^2(J)}^2) \\
& +  C (\diam(I)^{2r} \|\na \varphi_I\|^2_{L^s(I)}  +  \diam(J)^{2r}  \|\na \varphi_J\|_{L^s(J)}^2) 
\end{aligned}
\end{equation}
    
    \item We show that for a proper choice of the family $(u_I)$ and with the help of the previous local extensions, there exists a global extension $\phi^N$ satisyfing \eqref{rescaled_inequality_extension}. 
\end{enumerate}

\mspace
{\em Step 1}. Let $e=[x_{I,\alpha}, x_{J,\beta}] \in \Ed(\tilde{F}_N)$.  Clearly, to show the existence of  $\phi_e$ satisfying  \eqref{prop_phi_e_1}-\eqref{prop_phi_e_2}, one can restrict to the case where 
\begin{equation} \label{mean_zero_I}
\dashint_I \varphi_I = 0  \quad \forall I, \quad \text{ so that, for all $t \ge 2$,}  \quad \|\varphi_I\|_{W^{1,t}(I)} \le  C  \, \diam(I)^r \,   \|\na \varphi_I\|_{L^t(I)} 
\end{equation}
thanks to \eqref{PW}. Let now  $P_\alpha$ and $P_\beta$  two paraboloids that enclose $I_\alpha$ and $J_\beta$. By assumption (G2), for appropriate local cylindrical coordinates $(r,\theta,z)$ centered at $\frac{x_{I,\alpha} + x_{J,\beta}}{2}$ and of axis along $e$, we can write
 \begin{equation*} 
P_\alpha  =  \Big\{ z \ge  \frac{|e|}{2} +  a r^2   \Big\}, \quad  P_\beta = \Big\{ z \le - \frac{|e|}{2} - a r^2  \Big\}.
\end{equation*}
By Lemma \ref{lemma::bridge function}, given $d > 0$ there exists a function $w_{e} \in H^1(P_\alpha \cup P_\beta \cup F_{e}(d))$ where in the same system of local coordinates: 
$$F_e(d):= \Big\{ r^2 \leq  d^2, \quad a r^2 + \frac{|e|}{2}  \ge z \ge -\frac{|e|}{2}  -a r^2  \Big\}$$
and where 
\begin{align*}
 & 0 \leq w_e\leq 1, \quad  w_{e}|_{P_\alpha}=1, \quad  w_e|_{P_\beta}=0, \\
    & \int_{F_{e}(d)} |\na w_e|^2 \dx \leq C \mu_e, \quad  \int_{F_{e}(d)} |\na w_e|^2 |x-x_{I,\alpha}|^{2 \gamma} \dx \leq C_\gamma \quad \forall \gamma > 0.
\end{align*}
We take $d > 0$ large enough so that $I_\alpha \cup J_\beta$ lies in the interior of $P_\alpha \cup P_\beta \cup F_{e}(d)$.


\mspace
Now, using the Stein extension operator from $W^{1,t}(I)$ to $W^{1,t}(\R^3)$,  see  \cite[Chapter 6]{stein1970singular}, one can find an extension $\tilde{\varphi}_I$ of $\varphi_I$  on $\R^3$ such that for all $t \in [2,\tilde{s}]$, 
\begin{equation}
    || \tilde{\varphi}_I||_{W^{1,t}(\R^3)} \leq c \,  ||\varphi_I||_{W^{1,t}(I)} \le c' \, \diam(I)^r ||\na \varphi_I||_{L^t(I)}  \ \forall I \in \CC(F) 
    \label{stein_operator_inequality}
\end{equation}
We remark that the analysis in \cite[Chapter 6]{stein1970singular} provides constants $c$ and $c'$ that are independent of the size of the inclusions. This is however not a crucial point here, as we could handle constants diverging polynomially in the diameter of the inclusions.  Eventually, we define 

\begin{equation}
    \phi_e := w_e (\tilde{\varphi}_I + u_I) + (1-w_e) (\tilde{\varphi}_J + u_J) 
    \label{definition relevement local u_IJ}
\end{equation}
that we consider as a function  of $W^{1,s}(V_e)$, with
\begin{equation} \label{defi_tilde_Ve} 
 V_e \: := \: P_\alpha \cup P_\beta  \cup F_{e}(d) \cap \{ |z| \le R \}
 \end{equation}
where $R$ is large enough so that  $I_\alpha \cup J_\beta$ lies in the interior of $V_e$. The intersection with   $\{ |z| \le R \}$ is just here to make $V_e$ bounded.  Clearly, \eqref{prop_phi_e_1} is satisfied. We then compute

\begin{align*}
    \int_{F_{e}(d)} |\na \phi_e|^2 \dx 
    & \leq  \int_{F_{e}(d)} |\na w_e (\tilde{\varphi}_I+u_I -\tilde{\varphi}_J - u_J ) + w_e \na \tilde{\varphi}_I + (1-w_e)  \na \tilde{\varphi}_J |^2 \dx.
\end{align*}
Hence, 
\begin{align*}
   \int_{F_{e}(d)} |\na \phi_e|^2 \dx  \le & \:    C \, \Big( \int_{F_{e}(d)} |\na w_e|^2 |\tilde{\varphi}_I+u_I -\tilde{\varphi}_J - u_J|^2 \dx + \int_{F_{e}(d)} |\na \tilde{\varphi}_I|^2 + |\na \tilde{\varphi}_J|^2 \dx \Big) \\
   \leq & \:  C'   \, \Big(  \int_{F_{e}(d)} |\na w_e|^2 |\varphi_I(x_{I,\alpha})+u_I -\varphi_J(x_{J,\beta}) - u_J|^2 \dx \\
  + & \:  \int_{F_{e}(d)} |\na w_e|^2 |\tilde{\varphi}_I - \varphi_I(x_{I,\alpha}) + \varphi_J(x_{J,\beta}) -\tilde{\varphi}_J |^2 \dx  \\
+ & \diam(I)^{2r} ||\na \varphi_I||^2_{L^2(I)} + \diam(J)^{2r} ||\na \varphi_J||^2_{L^2(J)} \Big).
\end{align*}
Thanks to Morrey's inequality $W^{1,s}(\R^3) \xhookrightarrow{} C^{\gamma}(\R^3)$,  $\gamma=1-\frac{3}{s} \in (0,1)$, one can write   $$|\tilde{\varphi}_I (x) - \varphi_I(x_{I,\alpha})| \leq C  |x-x_{I,\alpha}|^\gamma ||\na \tilde{\varphi}_I ||_{L^s(\R^3)}   \leq C' |x-x_{I,\alpha}|^\gamma  \diam(I)^r ||\na \varphi_I ||_{L^s(I)}$$ 
for any $x$ in the gap $F_{e}(d)$. Thus, we have the following  

\begin{align*}
& \int_{F_{e}(d)} |\na w_e|^2 |\tilde{\varphi}_I - \varphi_I(x_{I,\alpha}) + \varphi_J(x_{J,\beta}) -\tilde{\varphi}_J |^2 \dx \\
 \leq  & \,  C  \, \diam(I)^{2r} ||\na \varphi_I ||_{L^s(I)}^2  \int_{F_{e}(d)} |\na w_e|^2 |x-x_{I,\alpha}|^{2\gamma} \dx  \\
  +     & \,  C \, \diam(J)^{2r} ||\na \varphi_J ||_{L^s(J)}^2  \int_{F_{e}(d)} |\na w_e|^2 |x-x_{J,\beta}|^{2\gamma} \dx \\
 \leq &  \,  C' \,  \big(  \diam(I)^{2r}  ||\na \varphi_I ||_{L^s(I)}^2  + \diam(J)^{2r} ||\na \varphi_J ||_{L^s(J)}^2 \big). 
  \end{align*}
 Finally, combining the previous inequalities  entails
\begin{equation*}
\begin{aligned}
    \int_{F_{e}(d)} |\na  \phi_e|^2 \dx & \le  C \Big(   \mu_e   |\varphi_I(x_{I,\alpha})+u_I -\varphi_J(x_{J,\beta}) - u_J|^2 \\
& \hspace{1cm} + \diam(I)^{2r}   ||\na \varphi_I ||_{L^2(I)}^2 + \diam(J)^{2r}  ||\na \varphi_J||_{L^2(J)}^2 \\ 
& \hspace{1cm} + \diam(I)^{2r}   ||\na \varphi_I ||_{L^s(I)}^2 + \diam(J)^{2r}  ||\na \varphi_J||_{L^s(J)}^2 \Big).
\end{aligned}
\label{gradient of u_ij in the contact zone}
\end{equation*}
It is even simpler to show that 
\begin{equation*}
   \int_{V_e \setminus F_{e}(d)} |\na  \phi_e|^2 \dx +    \int_{V_e} | \phi_e|^2 \dx \le C \Big(   |u_I|^2 + |u_J|^2 + \diam(I)^{2r}  ||\na \varphi_I ||_{L^2(I)}^2 + \diam(J)^{2r} ||\na \varphi_J||_{L^2(J)}^2 \Big)
\label{ u_ij in the contact zone}
\end{equation*}
which concludes the derivation of \eqref{prop_phi_e_2}, and  the first step. 

\bspace
{\em Step 2}. We now explain how to construct a global extension $\phi_N \in H^1_0(N U)$ with 
\begin{equation} \label{prop_phi_eps}
\phi_N\vert_I = \varphi_I - \dashint_I \varphi_I  + u_I,  \quad \forall I \in \CC(\tilde{F}_N).
\end{equation}
 For all  $e=[x_{I,\alpha}, x_{J,\beta}] \in \Ed(\tilde{F}_N)$, we first introduce a function $\chi_e \in C^\infty_c(\R^3)$, with values in $[0,1]$, satisfying 
\begin{itemize}
\item $\chi_e = 1$ in a neighborhood of $I_\alpha \cup J_\beta$
\item $\chi_e = 0$ in a neighborhood of  $\R^3 \setminus V_e$, where $V_e$ was introduced in \eqref{defi_tilde_Ve}
\item $\chi_e = 0$ in a $\delta_0/2$ neighborhood of $N\pa U$
\item the supports of $\chi_e$ and $\chi_{e'}$ are disjoint for all $e \neq e' \in \Ed(\tilde{F}_N)$
\item $|\na \chi_e| \le C$ for some constant $C$ that is uniform in $e$ and $N$.
\end{itemize}
Existence of such functions is easily deduced from our geometric assumptions (G1)-(G2). We now set 
$$ \phi_{N,1} :=  \sum_{e \in  \Ed(\tilde{F}_N)} \chi_e \phi_e.$$ 
By our choice of functions $\chi_e$ and by property \eqref{prop_phi_e_1},  one has $\phi_{N,1} \in H^1_0(N U)$, and  
\begin{equation} \label{prop_phi_eps_1} 
\phi_{N,1}\vert_{I_\alpha} = \varphi_I - \dashint_I \varphi_I  + u_I,  \quad \forall I \in \CC(\tilde{F}_N), \quad \forall \alpha = 1, \dots, N_I. 
\end{equation}
Moreover, by estimate \eqref{prop_phi_e_2}, 
\begin{equation*}
\begin{aligned} 
& \| \phi_{N,1}\|_{H^1(NU)}^2 \\
&  \le C  \sum_{e = [x_{I,\alpha}, x_{J,\beta}] \in  \Ed(\tilde{F}_N)}  \Big(  \, \mu_e \, \big|\varphi_I(x_{I,\alpha}) - \dashint_I \varphi_I  - \varphi_J(x_{J,\beta}) +  \dashint_J \varphi_J  +  u_I - u_J\big|^2  +   |u_I|^2 + |u_J|^2 \\ 
&\hspace{4cm} +  \, \diam(I)^{2r}  \|\na \varphi_I\|^2_{L^2(I)}  + \,   \diam(J)^{2r}  \|\na \varphi_J\|_{L^2(J)}^2 \\
& \hspace{4cm} +  \,  \diam(I)^{2r}  \|\na \varphi_I\|^2_{L^s(I)}  + \, \diam(J)^{2r}  \|\na \varphi_J\|_{L^s(J)}^2  \Big)
\end{aligned}
\end{equation*}
It remains to  construct some $\phi_{N,2}$ satisfying 
\begin{equation} \label{prop_phi_eps_2}
\phi_{N,2}\vert_{I} = \psi_I, \quad \psi_I :=  \varphi_I - \dashint_I \varphi_I  + u_I - \phi_{N,1},  \quad \forall I \in \CC(\tilde{F}_N),
\end{equation}
in order for $\phi_N := \phi_{N,1} + \phi_{N,2}$ to satisfy \eqref{prop_phi_eps}. By \eqref{prop_phi_eps_1}, $\psi_I$ is zero on $I_\alpha$ for all $I \in \CC(\tilde{F}_N)$, for all   $1 \le \alpha \le N_I$.  Thanks to this property and (G1)-(G2), one can find a constant $\nu > 0$ independent of $I$ or $N$ such that for all $I$, there exists $\tilde{\psi}_I \in H^1(\R^3)$ satisfying
$$  \tilde{\psi}_I\vert_I = \psi_I, \quad \|\tilde{\psi}_I\|_{H^1(\R^3)} \le C \|\psi_I\|_{H^1(I)} $$
and for all $(J,\beta)$ with $e = [x_{I,\alpha}, x_{J,\beta}] \in \Ed(\tilde{F}_N)$, for the same local coordinates  around the edge $e$ as seen before
$$  \tilde{\psi}_I = 0 \quad \text{on } \: J_{\beta,\nu} := J_\beta \cap \{|z| \le \nu \}. $$
Now, for each $I \in  \CC(\tilde{F}_N)$, we introduce $\nu' > 0$ and a smooth  function $\chi_I$ which is $1$ in a $\nu'$-neighborhood of $I$ and $0$ outside a  $2\nu'$-neighborhood of $I$. Thanks to our geometric assumptions, by taking $\nu'$ small enough (but independent of $I$ and $N$), we can ensure that for all $J \neq I$ connected by an edge $e \in \Ed(\tilde{F}_N)$, 
$$\text{Supp}(\chi_J) \cap I \: \subset \:  \cup_{1 \le \alpha \le N_I} \: I_{\alpha,\nu}.$$ 
We finally set 
$$  \phi_{N,2} = \sum_{I \in \CC(\tilde{F}_N)} \chi_I \,  \tilde{\psi}_I.$$
The keypoint in the definition of $\phi_{N,2}$ is that for a given $I$, and for any $J \neq I$, the term $\chi_J \, \tilde{\psi}_J$  is zero on $I$: indeed, for all $x \in I$,  either $\chi_J(x) = 0$, or $\chi_J(x) \neq 0$ which implies that $x \in \cup_\alpha I_{\alpha,\nu}$  so that 
$\tilde{\psi}_J(x) = 0$. Hence, \eqref{prop_phi_eps_2} is satisfied, as expected. Moreover, it is easily seen that  
\begin{align*}
 \| \phi_{N,2}\|_{H^1(N U)}^2  & \le  C \Big( \| \phi_{N,1}\|_{H^1(N U)}^2 \: + \: \sum_{I \in \CC(\tilde{F}_N)} \Big( \diam(I)^{2r} \|\na \varphi_I\|_{L^2(I)}^{2r} \, + \, | I | \, |u_I|^2 \Big) \Big) 
 \end{align*}  
 so that eventually 
\begin{equation*}
\begin{aligned} 
& \| \phi_N\|_{H^1(N U)}^2 \\
& \le C  \sum_{e = [x_{I,\alpha}, x_{J,\beta}] \in  \Ed(\tilde{F}_N)} \Big(  \, \mu_e \, \big|\varphi_I(x_{I,\alpha}) - \dashint_I \varphi_I  - \varphi_J(x_{J,\beta}) +  \dashint_J \varphi_J  +  u_I - u_J\big|^2 \\
&  + C  \sum_{I \in \CC(\tilde{F}_N}  | I | \, |u_I|^2  +  \sum_{I \in \CC(\tilde{F}_N)}   \diam(I)^{2r} (  \|\na \varphi_I\|^2_{L^2(I)} +  \|\na \varphi_I\|^2_{L^s(I)} )  
\end{aligned}
\end{equation*}

\mspace
The final step of the proof is to show that for a proper choice of the family $(u_I)_{I \in \CC(\tilde{F}_N)}$, $\phi_N$ satisfies the bound \eqref{rescaled_inequality_extension}. This is done using assumption \eqref{H2}. 
Namely, we denote $b_{IJe} := \varphi_I(x_{I,\alpha})- \dashint_I \varphi_I$, where $e = [x_{I,\alpha}, x_{J,\beta}] \in \Ed(\tilde{F}_N)$. Remembering Definition \ref{defi:En}, one has clearly 
\begin{equation*}
\begin{aligned} 
\| \phi_N\|_{H^1(NU)}^2 & \le C \, \mE\Big(\tilde{F}_N, \{u_I\}, \{b_{IJe}\}\Big)  \: + \:  C  \sum_{I \in \CC(\tilde{F}_N)}  \diam(I)^{2r} (  \|\na \varphi_I\|^2_{L^2(I)} +  \|\na \varphi_I\|^2_{L^s(I)} )   \\
& \le C \, \mE\Big(\tilde{F}_N, \{u_I\}, \{b_{IJe}\}\Big)  \: + \:  C  \sum_{I \in \CC(\tilde{F}_N)}  \diam(I)^{2r} (|\, I \,|^{\frac{\tilde{s}-2}{\tilde{s}}} + | \, I \, |^{\frac{2(\tilde{s} - s)}{s \tilde{s}}})  
  \|\na \varphi_I\|^2_{L^{\tilde{s}}(I)}   \\
& \le C \, \mE\Big(\tilde{F}_N, \{u_I\}, \{b_{IJe}\}\Big) \:   + \:  C'   \, \Big( \sum_{I \in \CC(\tilde{F}_N)}  \diam(I)^{\frac{2r\tilde{s}}{\tilde{s}-2}} |\, I \,|  \Big)^{\frac{\tilde{s}-2}{\tilde{s}}}   \|\na \varphi\|_{L^{\tilde{s}}(\tilde{F}_N)}^2  
\end{aligned}
\end{equation*}
Now, taking $N' = 2 N \sup_{x \in U} |x| $, one has $N U \subset Q_{N'}$ and $|Q_{N'}| = C |Q_N|$ for a constant $C$ independent of $N$. Furthermore, we can write $F_{N'} = \tilde{F}_N \cup G$, where the union is disjoint and $Gr(F_{N'})$ is deduced from $Gr(\tilde{F}_N)$ by the addition of nodes and edges. Using the property (\ref{prop_E_1}) proved in Section  \ref{section_disc_assumptions}, we have 

$$\mE\Big(\tilde{F}_N, \{u_I\}, \{b_{IJe}\}\Big) \leq  \mE\Big(F_{N'}, \{ \bar{u}_I\}, \{\bar{b}_{IJe}\}\Big)$$
 for any extensions $\{\overline{u}_I\}, \{\overline{b}_{IJe}\}$ of $\{u_I\},  \{b_{IJe}\}$. We make the choice $\bar{b}_{IJe} = 0$ if $e \in \Ed(F_{N'}) \setminus \Ed(\tilde{F}_N) $.
Now, using property \eqref{H2}  (in the form mentioned in  Remark i)  after Theorem \ref{thm_homog}), there exists, almost surely, a family $\{\bar{u}_I\}_{I \in \CC(F_{N'})}$ such that 

\begin{align*}
    \mE\Big(F_{N'}, \{ \bar{u}_I\}, \{\bar{b_{IJe}}\}\Big) & \leq M |Q_{N'}|\Big( \frac{1}{|Q_{N'}|} \sum_{I,J \in \CC(F_{N'})} \sum_{\substack{e \in \Ed(F_{N'}), \\ I \overset{e}{\leftrightarrow} J}} |\bar{b}_{IJe}|^s \Big)^{2/s}\\
    & \le M' |Q_{N}|\Big( \frac{1}{|Q_N|} \sum_{I,J \in \CC(\tilde{F}_N)} \sum_{\substack{e \in \Ed(\tilde{F}_N), \\ I \overset{e}{\leftrightarrow} J}} |b_{IJe}|^s \Big)^{2/s}.
\end{align*}
Using one last time the Morrey injection yields
$$|b_{IJe}|^s = |\varphi_I(x_{I,\alpha})- \varphi_I(x_I)|^s \leq |x_I - x_{I,\alpha}|^{s-3} \, \| \na \tilde \varphi_I \|^s_{L^s(\R^3)} \le C  \diam(I)^{s-3+rs} ||\na \varphi_I||_{L^s(I)}^s. $$
Setting $u_I  = \overline{u}_I$ for $I \in \CC(\tilde{F}_N)$, we get,  back to  $\phi_N$: 
\begin{align*}
   ||\phi_N||^2_{H^1(N U)} & \leq M' |Q_{N}| \Big( \frac{1}{|Q_{N}|} \sum_{I,J \in \CC(\tilde{F}_N)}   \diam(I)^{s-3+rs}  ||\na \varphi_I||_{L^s(I)}^s \Big)^{\frac{2}{s}} \\ 
   & + C' \, \Big( \sum_{I \in \CC(\tilde{F}_N)} \diam(I)^{\frac{2r \tilde{s}}{\tilde{s}-2}}   |\, I \,|  \Big)^{\frac{\tilde{s}-2}{\tilde{s}}}   \,  \|\na \varphi\|_{L^{\tilde{s}}(\tilde{F}_N)}^2   \\
   & \le  C'' N^{\frac{3\tilde{s}-6}{\tilde{s}}} \, \Big(\frac{1}{|Q_{N'}|} \sum_{I \in \CC(\tilde{F}_N)}  \diam(I)^{p} |\, I \,|  \Big)^{\frac{\tilde{s}-2}{\tilde{s}}} \, \|\na \varphi\|_{L^{\tilde{s}}(\tilde{F}_N)}^2   
   \end{align*}
   where $p = \max((s-3 + rs) \frac{\tilde{s}}{\tilde{s} - s}, \frac{2r \tilde{s}}{\tilde{s}-2})$.  
 As 
 $$   \frac{1}{|Q_{N'}|} \sum_{I \in \CC(\tilde{F}_N)}  \diam(I)^{p} |\, I \,|  \le \frac{1}{|Q_{N'}|} \int_{Q_{N'}} \diam(I_{y,F})^{p} dy $$
 we find by the ergodic theorem that 
 $$ \limsup_{N \rightarrow + \infty} \frac{1}{|Q_{N'}|} \sum_{I \in \CC(\tilde{F}_N)}  \diam(I)^p |\, I \,|  \le \E \, \diam(I_{0,F})^{p} < +\infty $$
   which concludes the proof. 
\end{proof}
\subsection{Extension inside the inclusions preserving divergence}
Our proof of homogenization, based on the div-curl lemma, will require proper extensions of solenoidal vector fields, or of fields with given divergence, inside the inclusions.  This is the purpose of  
\begin{proposition} \label{prop_div_extensions}

\sspace
Assume that \eqref{H2} holds with $F' = F$ and $s > 3$. Let 
$f^\eps \in L^{6/5}(U)$, $\: p^\eps \in L^2(U\setminus F^\eps)$ such that  $\:  \div p^\eps = f^\eps \: \text{ in }  \: U\setminus F^\eps$, satisfying the following compatibility conditions: 
$$  \int_{\pa I^\eps} p^\eps \cdot \nu = \int_{I^\eps} f^\eps, \quad \forall   I^\eps \in \CC(F^\eps).  $$
Then,  there exists a field $P^\eps \in L^2(U)$ satisfying 
$P^\eps\vert_{U\setminus F^\eps}  = p^\eps$,   $\div P^\eps = f^\eps$ in $U$. Moreover, given any $t < s' = \frac{s}{s-1}$, there exists $p = p(t)$ such that under the additional hypothesis $\E \, \diam(I_{0,F})^p < +\infty$,   one can choose $P^\eps$ satisfying  the uniform estimate 
$$ \| P^\eps\|_{L^t(U)} \le C_t \Big( \|p^\eps\|_{L^2(U\setminus F^\eps)} + \|f^\eps\|_{L^{6/5}(U)} \Big)$$ 
\end{proposition}

\begin{proof}
Let $t < s'$. We introduce $\tilde{s}$ such that $s < \tilde{s} < t'$. We also introduce the solution $w$ of
$$ \Delta w =  f^\eps  \: \text{ on } \: U,  \quad w\vert_{\pa U} = 0. $$ 
It satisfies the estimate 
$$  \|\na w\|_{L^q(U)}  \le C_q \|f^\eps\|_{W^{-1,q}(U)} \le C'_q \|f^\eps\|_{L^{3q/(3+q)}(U)} \quad \forall  q \in (1, 2] $$
Denoting $s^\eps = p^\eps - \na w$, it remains to find $S^\eps$ satisfying 
$S^\eps\vert_{U\setminus F^\eps}  = s^\eps$,   $\div S^\eps =0$ in $U$, and
$$  \| S^\eps\|_{L^t(U)} \le C_t  \|s^\eps\|_{L^2(U\setminus F^\eps)}. $$
Then $P^\eps := S^\eps + \na w$ will meet all requirements. The idea  is to search for $S^\eps$ in the form of a gradient in the inclusions. Strictly speaking, for a fixed realization, we introduce a field $v^\eps$ (depending on $\omega$) defined on $F^\eps$ that verifies in each inclusion $I^\eps$ the Neumann problem 
\begin{equation*}
\left\{
\begin{aligned}
 \Delta v^\eps & = 0 \quad \text{in } \mathring{I}^\eps, \\
\partial_{\nu} v^\eps & = s^\eps \cdot \nu \quad \text{in } \: \partial I^\eps.
\end{aligned}
\right. 
\end{equation*}
This Neumann problem is well-posed thanks to the compatibility condition 
$$ \int_{\pa I^\eps} s^\eps \cdot \nu =  \int_{\pa I^\eps} p^\eps \cdot \nu - \int_{\pa I^\eps} \pa_\nu w  = \int_{I^\eps} f^\eps   - \int_{\mathring{I}^\eps} \Delta w = 0, \quad \forall  I^\eps \in \CC(F^\eps).  $$
 We then define the random field $S^\eps$ by $S^\eps=s^\eps$ in $U \setminus F^\eps$ and $S^\eps = \na v^\eps$ in $F^\eps$. It is divergence-free on $U$ thanks to the continuity of its normal component through each $\pa I^\eps$. To establish the uniform estimate on $S^\eps$ in $L^t(U)$, we proceed by duality.   Let $\Phi \in L^{t'}(U)$, with $t'$ the conjugate of $t$.  It admits the following Helmholtz decomposition in each inclusion: 
$$\Phi|_{I^\eps}:= \mathbb{P}_{\mathring{I}^\eps}  \Phi + \nabla \varphi^\eps,  \quad \forall    I^\eps \in \CC(F^\eps). $$
where for any open set $\mathcal{O}$,  $\mathbb{P}_{\mathcal{O}} $ is the  Leray projector, continuous  over $L^{\tilde{s}}(\mathcal{O})$. More precisely, we claim that 
$$ \|\mathbb{P}_{\mathring{I}^\eps}  \Phi\|_{L^{\tilde{s}}(I^\eps)} \: + \:  ||\na \varphi^\eps ||_{L^{\tilde{s}}(I^\eps)} \leq \, C_{\tilde{s}} \, \diam(I)^R  \, ||\Phi||_{L^{\tilde{s}}(I^\eps)}, \quad  \text{ for some }  R > 0.  $$
Indeed, by scaling, it is enough to show this inequality for $\eps=1$. To show that  the operator norm of $\mathbb{P}_{\mathring{I}}$ (or equivalently $\mathrm{Id} - \mathbb{P}_{\mathring{I}}$)  
 is bounded by a power of $\diam(I)$, one writes 
$(\mathrm{Id} - \mathbb{P}_{\mathring{I}})  f = \na u_f$, where  
$$ \Delta u_f = \div f \quad \text{ on } \mathring{I}, \quad \pa_\nu u\vert_{\pa I} = f \cdot \nu.  $$
One must  then  look carefully  at the proof  of the inequality  $\|\na u_f\|_{L^{\tilde{s}}(\mathring{I})} \le C_I \|f\|_{L^{\tilde{s}}(\mathring{I})}$, and track the dependence of $C_I$ with respect to $I$. The derivation of this inequality follows the usual scheme:  by local charts and straightening of the boundary, one can use that the inequality holds in  $\R^3$ and  $\R^3_+$. A tedious verification shows that the constant in the inequality involves the constant $d$ in (G1), the number of charts and the constant in the Poincaré inequality.  Under our regularity assumptions, it is  controlled by $\diam(I)^R$ for large enough $R$. We skip the details for brevity.  

\mspace
We now introduce the function $\phi^\eps \in H^1_0(U)$ associated to $\varphi^\eps$ in Proposition \ref{prop_extension_H2}. In particular, $\phi^\eps$ and $\varphi^\eps$ coincide on each $I^\eps$ up to a constant function. We compute 
\begin{align*}
 \int_{F^\eps} \na v^\eps \cdot \Phi  & = \sum_{I^\eps \in \CC(F^\eps)} \int_{I^\eps} \na v^\eps \cdot \na \varphi^\eps  \underbrace{- \int_{I^\eps} v^\eps \, \div \, \mathbb{P}_{\mathring{I}^\eps} \Phi + \int_{\partial I^\eps} v^\eps \,  \mathbb{P}_{\mathring{I}^\eps}  \Phi \cdot \nu}_{=0}\\
    & =  \sum_{I^\eps \in \CC(F^\eps)}\int_{\partial I} s^\eps \cdot \nu  \, \varphi^\eps \text{ using the equation on $v^\eps$}\\
    & =   \sum_{I^\eps \in \CC(F^\eps)}\int_{\partial I} s^\eps \cdot \nu  \, \phi^\eps  \text{ using that $\int_{\partial_I} s^\eps \cdot \nu =0$}\\
    & =  \int_{U \setminus F^\eps} s^\eps \cdot \na \phi^\eps  \le \|s^\eps\|_{L^2(U\setminus F^\eps)} \|\na\phi^\eps\|_{L^2(U\setminus F^\eps)}   \le C \|s^\eps\|_{L^2(U\setminus F^\eps)} \| \na \varphi^\eps\|_{L^{\tilde{s}}(F^\eps)} 
\end{align*}
where the last inequality comes from Proposition \ref{prop_extension_H2}. Now,  
\begin{align*}
\| \na \varphi^\eps\|_{L^{\tilde{s}}(F^\eps)}^{\tilde{s}} & = \sum_{I^\eps \in \CC(F^\eps)}  \| \na \varphi^\eps\|_{L^{\tilde{s}}(I^\eps)}^{\tilde{s}}  \le C  \sum_{I^\eps \in \CC(F^\eps)} \diam(I)^{R\tilde{s}} 
  \| \Phi \|_{L^{\tilde{s}}(I^\eps)}^{\tilde{s}} \\
& \le C \Big( \sum_{I^\eps \in \CC(F^\eps)}  |I^\eps|  \diam(I)^{p}  \Big)^{\frac{t'- \tilde{s}}{t'}} \| \Phi \|_{L^{t'}(F^\eps)}^{\tilde{s}}, \quad p = \frac{R \tilde{s} t'}{t'-\tilde{s}}. 
\end{align*}
Again, 
$$  \sum_{I^\eps \in \CC(F^\eps)} |I^\eps|  \diam(I)^{p} \le C \int_U \diam(I_{x/\eps,F})^p \: \xrightarrow[\eps \rightarrow 0]{} C \,  |U| \,  \E \,  \diam(I_{0,F}^p) < +\infty $$
so that we end up with 
\begin{align*}
\int_{U} S^\eps  \cdot \Phi & =  \int_{U\setminus F^\eps} s^\eps  \cdot \Phi + \int_{F^\eps} \na v^\eps  \cdot \Phi    \le C \|s^\eps\|_{L^2(U\setminus F^\eps)}  \big( \|\Phi\|_{L^2(U\setminus F^\eps)} +  \|\Phi\|_{L^{t'}(F^\eps)} \Big)  \\
& \le C'   \|s^\eps\|_{L^2(U\setminus F^\eps)} \|\Phi\|_{L^{t'}(U)} 
\end{align*}
This concludes the proof. 
\end{proof}

\section{Proofs of the main results} \label{sec_main_results}
Here, again, $F$ is an admissible set of inclusions. 
\subsection{Existence of the homogenized matrix - Proposition \ref{prop_corrector}}

The goal of this section is to define properly the matrix $A^0$, describing the effective viscosity of the conducting medium. We follow here the approach developped in \cite[chapter 8]{book:ZKO}. We first introduce the so-called resistance matrix, that is the symmetric matrix defined by:  
$$ \forall \xi \in \R^3, \quad B^0 \xi \cdot \xi := \inf_z \E \int_{Q_1 \setminus F} |\xi + z|^2, $$
where the infimum is taken over vector fields $z = z(y,\omega) \in L^2(\Omega, L^2_{loc}(\R^3))$ that are solenoidal, stationary and mean-free. An equivalent formulation of the variational problem is 
\begin{equation} \label{VF_B0}
 \forall \xi \in \R^3, \quad B^0 \xi \cdot \xi := \inf_{Z \in \mathcal{V}^2_{sol}(\Omega)}   \int_{\Omega\setminus \mF}  |\xi + Z|^2, 
 \end{equation}
where:
\begin{itemize}
\item $\mathcal{F}$ is the subset of $\Omega$ defined in \eqref{defi_rmF}, so that $F(\omega) = \{ x \in \R^3, \tau_x(\omega) \in \mF\}$. 
\item  $\mathcal{V}^2_{sol}(\Omega) = \{ Z \in L^2(\Omega), \quad \E Z = 0, \quad y \rightarrow Z(\tau_y(\omega)) \: \text{solenoidal vector field}\}$.  
\end{itemize}
We remind that introducing the other subspace of vector fields 
$$L^2_{pot}(\Omega) = \{ U \in L^2(\Omega), \quad  y \rightarrow U(\tau_y(\omega)) \: \text{potential vector field}\}$$
one has the orthogonal decomposition $L^2(\Omega) =  \mathcal{V}^2_{sol}(\Omega) \oplus L^2_{pot}(\Omega)$.  

\mspace
Still following \cite[chapter 8]{book:ZKO}, if we now  denote 
\begin{equation} \label{def_X}
 \text{$X$ : the closure in $L^2(\Omega\setminus \mF)$ of the space $\{Z\vert_{\Omega\setminus \mF}, \: Z \in \mathcal{V}^2_{sol}(\Omega)\}$} 
 \end{equation}
 then there exists a unique minimizer $Z \in X$ attaining the infimum, and it satisfies the Euler-Lagrange equation: 
 $$ \int_{\Omega \setminus \mathcal{F}}  (\xi + Z) \cdot Z' = 0, \quad \forall Z' \in \mathcal{V}^2_{sol}(\Omega).   $$
 In particular, $1_{\mF} (\xi + Z) \in L^2_{pot}(\Omega)$, and $B^0 \xi = \E 1_{\mF} (\xi + Z)$. 
  
\mspace
The last step of proof of Proposition \ref{prop_corrector} is showing that the matrix $B^0$ above is invertible. Therefore, we use Lemma 8.7 of \cite{book:ZKO}, which provides a sufficient condition: 
\begin{lemma} {\bf \cite[Lemma 8.7]{book:ZKO}}
Assume that for any $\xi \in \R^3$, and for any $\omega$ in a subset of positive measure,  there exists a sequence of potential vector fields $v^\eps \in L^2(U)$ satisfying 
$$ v^\eps\vert_{F^\eps(\omega)} = 0, \quad v^\eps \rightarrow \xi \: \text{weakly in $L^2(U)$}, \quad \limsup_{\eps \rightarrow 0} \|v^\eps\|_{L^2}    \le C |\xi| \: \text{ for some $C > 0$.} $$
Then,  $B^0$ is positive definite.  
\end{lemma}

\mspace
The keypoint is that under \eqref{H1}, the assumptions of the lemma are granted by Proposition \ref{prop_extension_H1}: one can  take $v^\eps = \xi - \na \phi^\eps = \na (\xi \cdot x - \phi^\eps)$, with $\phi^\eps$ as in Proposition  \ref{prop_extension_H1}. This concludes the proof of Proposition  \ref{prop_corrector}.

\subsection{Homogenization without short} \label{subsec_without}
We prove in this section part of Theorem \ref{thm_homog}. Namely, we focus on the case where \eqref{H2} is satisfied with $F' = F$, for some $s> 3$. The reason for treating this special case separately is that it is much easier :  indeed, the arguments of  \cite[chapter 8]{book:ZKO} rely on the existence of  proper extensions of solenoidal  vector fields, or of fields with given divergence,  inside the inclusions.  As such extensions are granted by  Proposition  \ref{prop_div_extensions}, they  adapt straightforwardly. The proof of the general case, given in the next section, will be more involved (and due to technical difficulty limited to $s < 6$). 

\mspace
First, by Remark iii) after Theorem \ref{thm_homog}, \eqref{H2} implies \eqref{H1}. Hence, we can apply Proposition \ref{prop_corrector}, so that $A^0 = (B^0)^{-1}$ is well-defined. Let $f \in L^{6/5}(U)$, and  $u^\eps  \in H^1_0(U)$ the solution of \eqref{main_system} (with implicit dependence on $\omega$), where domain $F^\eps$ is defined in \eqref{defi:rescaled}. From a simple energy estimate, $u^\eps$ is bounded in $H^1_0(U)$ uniformly in $\omega$ and $\eps$. Hence, almost surely, $u^\eps$ has a subsequence  that converges weakly to some $u^0$. The goal is to show that $u^0$ satisfies \eqref{target_system}. By uniqueness of this accumulation point, this will mean  that the whole sequence converges to $u^0$. From now on, for the sake of brevity, we denote $u^\eps$ the converging subsequence. 

\mspace
  Let now $\xi \in \R^3$,  and $Z$ the minimizer of problem \eqref{VF_B0}. 
As $Z \in X$, {\it cf.} \eqref{def_X}, there exists a sequence  $Z^\nu \in \mathcal{V}^2_{sol}(\Omega)$ and $\|Z^\nu - Z\|_{L^2(\Omega\setminus \mF)} \rightarrow 0$ as $\nu \rightarrow 0$. Let $\bar{Z}$, resp. $\bar{Z}^\nu$, the extension of $Z$, resp. of $Z^\nu\vert_{\Omega\setminus\mF}$, by $-\xi$ on $\mF$. We  remind that $\xi + \bar{Z} \in L^2_{pot}(\Omega)$, with $\E \, (\xi + \bar{Z}) = B^0 \xi$. We finally set 
$$ \bar{z}(y,\omega) = \bar{Z}(\tau_y(\omega)), \quad  \bar{z}^\nu(y,\omega) = \bar{Z}^\nu(\tau_y(\omega)), \quad z^\nu(y,\omega) = Z^\nu(\tau_y(\omega)). 
$$
  Let $p^\eps = \na u^\eps$. By Proposition \ref{prop_div_extensions}, for any $t < s'$, assuming \eqref{moment_bound_F} for large enough $p$, one can extend $p^\eps$ into a field   $P^\eps$ such that 
$$ P^\eps\vert_{U\setminus F^\eps} = p^\eps, \quad \div P^\eps = f 1_{U\setminus F^\eps} \: \text{ in } U, \quad \|P^\eps\|_{L^t(U)} \le C_t $$
The last bound implies weak convergence of (a subsequence of)  $P^\eps$ towards some $P^0$ in $L^t(U)$. By the ergodic theorem, $f 1_{U\setminus F^\eps}$ converges weakly to $f (1-\lambda)$, with $\lambda =\E 1_{F}$,  in $L^{6/5}(U)$. Hence, 
$\div P^0 = (1-\lambda)f$ in  $U$.  
Let now $\varphi \in C^\infty_c(U)$, and $u^0$ the weak limit of (a subsequence of) $u^\eps$  in $H^1_0(U)$. The point is to show that, as $\eps \rightarrow 0$:
\begin{equation} \label{divcurl1}  
\int_{U} \varphi(x) \na u^\eps(x) \cdot (\xi +  \bar{z}(x/\eps)) dx  \: \rightarrow \: \int_{U}  \varphi(x) \na u^0(x)\cdot  \xi dx 
\end{equation}
as well as 
\begin{equation} \label{divcurl2}  
\int_{U} \varphi(x) \na u^\eps(x) \cdot (\xi +  \bar{z}(x/\eps))  dx  \: \rightarrow \: \int_{U}  \varphi(x) P^0 \cdot B^0 \xi dx 
\end{equation}
Identifying the limits, it follows that $B^0 P^0 = \na u^0$, so that $P^0 = A^0 \na u^0$ and as $\div P^0 = (1-\lambda f)$, we recover system \eqref{target_system}. 

\mspace
The proof of \eqref{divcurl1}-\eqref{divcurl2} is an adaptation of the one in \cite{book:ZKO}, so that we indicate only the main elements and the changes that are needed. As regards \eqref{divcurl1}, we write 
\begin{equation} \label{decompoz}
\begin{aligned}
\int_{U} \varphi(x) \na u^\eps(x) \cdot (\xi + \bar{z}(x/\eps)) dx & = \int_{U} \varphi(x) \na u^\eps(x) \cdot (\bar{z}(x/\eps) - \bar{z}^\nu(x/\eps)) dx \\
& +  \int_{U} \varphi(x) \na u^\eps(x) \cdot  (\bar{z}^\nu(x/\eps) - z^\nu(x/\eps))  dx  \\
& +  \int_{U} \varphi(x) \na u^\eps(x) \cdot  (\xi + z^\nu(x/\eps))  dx  
\end{aligned}
\end{equation}
The first term at the r.h.s. satisfies 
\begin{align*} 
\big| \int_{U} \varphi(x) \na u^\eps(x) \cdot (\bar{z}(x/\eps) - \bar{z}^\nu(x/\eps))  dx \big| & \le \| \varphi\|_\infty \, \|\na u^\eps \|_{L^2(U)} \, \| \bar{z}(x/\eps) - \bar{z}^\nu(x/\eps)  \|_{L^2(U)}
 \end{align*}
 so that, by the uniform $L^2$ bound on $\na u^\eps$ and the  ergodic theorem: 
 $$ \big| \limsup_{\eps}  |\int_{U} \varphi(x) \na u^\eps(x) \cdot (\bar{z}(x/\eps) - \bar{z}^\nu(x/\eps)) dx \big| \le C \|\bar{Z} - \bar{Z}^\nu\|_{L^2(\Omega)}= 
  C \|Z - Z^\nu\|_{L^2(\Omega\setminus\mF)} $$
 and finally 
 $$ \limsup_{\nu}  \limsup_{\eps}   |\int_{U} \varphi(x) \na u^\eps(x) \cdot (\bar{z}(x/\eps) - z^\nu(x/\eps)) dx \big| = 0 $$
 For the second term at the r.h.s. of \eqref{decompoz}, we notice that  $\na u^\eps \cdot  (\bar{z}^\nu(\cdot/\eps) - z^\nu(\cdot/\eps))$ is zero in $F^\eps$, because $\na u^\eps$ is zero there, and in $U \setminus \eps F$, because $\bar{z}^\nu = z^\nu$ there. However, it does not {\it a priori} vanish in $(\eps F) \cap U \setminus F^\eps$. This corresponds to inclusions $I^\eps$ in $\eps F$ that intersect $U_{\delta_0 \eps}$, where 
 $$ U_\eta  := \{ x \in U, \text{d}(x, \pa U) \le \eta \}, \quad  \eta > 0. $$
 A crucial point is that, under the moment condition $\E \, \diam(I_{0,F})^3 < +\infty$, by a direct adaptation of the proof of  Lemma \ref{lemmaH2H2kappa} and  \eqref{uniform_bound_inclusion} below,   one has almost surely,
 $$  \sup \, \{ \diam(I^\eps), \: I^\eps \in \CC(\eps F), \quad I^\eps \cap U_{\delta_0 \eps} \neq \emptyset \} = o(1) \quad \text{ as } \eps \rightarrow 0. $$
 Hence,  for any $\eta > 0$, for $\eps$ small enough, one has  $(\eps F \cap U) \setminus F^\eps \subset U_\eta$, so that 
\begin{align*}
&  \Big|  \int_{U} \varphi(x) \na u^\eps(x) \cdot  (\bar{z}^\nu(x/\eps) - z^\nu(x/\eps))  dx \Big|   =   \Big|  \int_{(\eps F \cap U) \setminus F^\eps} \varphi(x) \na u^\eps(x) \cdot  (\xi + z^\nu(x/\eps))  dx \Big| \\
 \le & \:  \|\varphi\|_{L^\infty} \| \na u^\eps \|_{L^2(U_\eta)} \| \xi + z^\nu(\cdot/\eps) \|_{L^2(U_\eta)} \:   \le \:    C \| \xi + z^\nu(\cdot/\eps) \|_{L^2(U_\eta)} 
\end{align*}
and by the ergodic theorem 
$$ \limsup_\eps \,   \Big|  \int_{U} \varphi(x) \na u^\eps(x) \cdot  (\bar{z}^\nu(x/\eps) - z^\nu(x/\eps))  dx \Big|  \le C \| 1_{U_\eta}\|_{L^2} \|  \xi + Z^\nu \|_{L^2(\Omega)}  \le C' \eta^{1/2}$$
As $\eta$ is arbitrary, it follows that 
$$ \limsup_\eps \,   \Big|  \int_{U} \varphi(x) \na u^\eps(x) \cdot  (\bar{z}^\nu(x/\eps) - z^\nu(x/\eps))  dx \Big|  = 0. $$
Finally, as regards the third term at the r.h.s. of \eqref{decompoz}, by the div-curl lemma and the ergodic theorem, for any given $\nu$, 
 $$ \lim_{\eps} \int_{U} \varphi(x) \na u^\eps(x) \cdot (\xi +  z^\nu(x/\eps)) dx =  \int_{U} \varphi(x) \na u^0(x) \cdot (\xi +  \E Z^\nu) dx = \int_{U} \varphi(x) \na u^0(x) \cdot \xi dx  $$
 where the last equality comes from the property  $\E Z^\nu =  0$. Combining all previous relations yields  \eqref{divcurl1}. 
 
 \mspace
As regards \eqref{divcurl2}, we want again to rely on the div-curl lemma but switching the potential and solenoidal vector fields. Therefore, we write 
\begin{align*}
 \int_{U} \varphi(x) \na u^\eps(x) \cdot (\xi + \bar{z}(x/\eps)) dx  
 & =  \int_U  \varphi(x)  P^\eps(x)  \cdot  \na w^\eps dx
\end{align*} 
 taking into account that $(\xi + \bar{z}(x/\eps))$  is a potential vector field, hence can be written $\na w^\eps$. Moreover, 
 $$ \na w^\eps \:  \rightarrow \: \E \, (\xi + \bar{Z}) = B^0 \xi \quad \text{weakly in $L^2(U)$, almost surely}. $$ 
 If $(P^\eps)_{\eps > 0}$ was bounded in $L^2(U)$, one could conclude directly by the div-curl lemma. As it is only bounded in $L^t(U)$ for $t < s'$, one must use an approximation of $w^\eps$ by the truncation 
 $$w^{\eps,l(x)} = w^\eps(x) \quad \text{if} \: |w^\eps(x)| \le l, \quad  w^{\eps,l(x)} =  l \quad \text{if} \: w^\eps(x) \ge l, \quad w^{\eps,l(x)} =  -l \quad \text{if} \: w^\eps(x) \le -l. $$
 We refer to \cite[chapter 8, page 286]{book:ZKO} for implementation of this argument. 

\subsection{Homogenization with short - Theorem \ref{thm_homog}} 
We tackle the proof of Theorem \ref{thm_homog}  in the general case where $F'$ is an admissible short of $F$. 

\mspace
First, we introduce the sequence of admissible shorts  $(F^\kappa)_{\kappa \in (0,1]}$, defined by the following properties : 
 for all $\kappa \in (0,1)$, $F'$ is a short of $F^\kappa$ and $F^\kappa$ is a short of $F$, with 
$$\mathrm{Ed}(F^\kappa) = \mathrm{Ed}(F') \cup  \Big\{ e  \in \mathrm{Ed}(F)\setminus \mathrm{Ed}(F') , \quad |e| \ge \kappa \}.$$
In other words, $F^\kappa$ is deduced from $F'$ by removing bridges corresponding to gaps of size larger than $\kappa$. Obviously, almost surely, for every closed ball $B$,  $F^\kappa \cap B = F \cap B$ for $\kappa \le \kappa_B$ small enough. 

\mspace

\begin{lemma} \label{lemmaH2H2kappa} 
 If $F'$ satisfies \eqref{H2} and the moment bound  \eqref{moment_bound_F} for $p=3$, then $F^\kappa$ satisfies \eqref{H2} for all $\kappa > 0$.
\end{lemma}

\mspace
{\em Proof.} We will first show that, 
\begin{equation} \label{uniform_bound_inclusion_kappa}
 \text{ almost surely, } \quad \forall \kappa \in \mathbb{Q} \cap  (0,1], \:   \sup_{I \in CC(F^\kappa), I \cap Q_N \neq \emptyset } \diam(I) = o(N)
 \end{equation}
 Indeed, let $\kappa \in \mathbb{Q} \cap  (0,1]$. Clearly, $\diam(I_{0,F^\kappa}) \le \diam(I_{0,F'})$, so 
 $\: \E \diam(I_{0,F^\kappa})^3 < +\infty$.  Let $\eta > 0$, and consider the event 
 $$ A_N = \{ \omega, \: \text{there exists $I \in CC(F^\kappa(\omega)), \: I \cap Q_N \neq \emptyset, \: I \cap Q_{N(1+\eta)}^c \neq \emptyset \}$}.  $$
 We recall that all inclusions satisfy an inner sphere condition with uniform deterministic radius. Hence, there exists a (deterministic) set of points   $x_1, \dots,  x_{K_N}$  of $\pa Q_{N+\frac{1}{2}}$ with  $K_N \le C N^2$ for a deterministic constant $C$ and such that   any $I \in CC(F^\kappa)$ with $I \cap Q_N \neq \emptyset$, $I \cap Q_{N(1+\eta)}^c \neq \emptyset$  contains at least an $x_i$. It follows that 
 $$ \mathbb{P}(A_N) \le \sum_{i = 1}^{K_N} \mathbb{P}\big(\diam(I_{x_i, F^\kappa}) \ge \eta N\big)  \le C N^2 \mathbb{P}\big(\diam(I_{0, F^\kappa}) \ge \eta N\big) $$
The moment bound implies that  $\sum P(A_N) < +\infty$, and it follows from Borel-Cantelli Lemma that $ \mathbb{P}(\limsup A_N) =0$. In other words, for all $\eta > 0$, for $\omega$ in a set of full measure,  there exists $N$ such that 
$$ \sup_{I \in CC(F^\kappa), I \cap Q_N \neq \emptyset } \diam(I) \le \eta N $$
By taking a countable subset of $\eta$ (and as $\kappa$ describes the countable subset $\mathbb{Q}  \cap  (0,1]$), one can find a set of full measure independent of $\kappa$ and $\eta$, which proves \eqref{uniform_bound_inclusion_kappa}. Let us remark that for $\kappa$ large enough, namely for $\kappa \ge \delta'$ with $\delta'$ the constant in (G1) associated to $F'$, one has $F^\kappa = F'$, so that \eqref{uniform_bound_inclusion_kappa} implies 
\begin{equation} \label{uniform_bound_inclusion}
 \text{ almost surely, } \quad   \sup_{I' \in CC(F'), I' \cap Q_N \neq \emptyset } \diam(I') = o(N)
 \end{equation}

\mspace
We now turn to the proof of the lemma. Let $N \ge 1$, and  $\{b_{IJe}\}$ a family indexed by $I,J \in F^\kappa_N$, $e \in \mathrm{Ed}(F^\kappa_N)$,  $I \overset{e}{\leftrightarrow} J$. By \eqref{uniform_bound_inclusion}, for $N$ large enough each $I \in \mathrm{CC}(F^\kappa_N)$ is included in a connected component of $F'_{2N}$. We define a family 
$\{b'_{I'J'e} \}$ indexed by $I',J' \in F'_{2N}$, $e \in \mathrm{Ed}(F'_{2N})$, $I \overset{e}{\leftrightarrow} J$ in the following way: 
\begin{itemize}
\item if $I'$ or $J'$ does not contain any element of $\mathrm{CC}(F^\kappa_N)$, $b'_{I'J'e} = 0$
\item if $I'$  and $J'$ contain elements of $\mathrm{CC}(F_\kappa^N)$, but $e \not\in \mathrm{Ed}(F^\kappa_N)$,   $b'_{I'J'e} = 0$
\item if  $I'$  and $J'$ contain elements of $\mathrm{CC}(F_\kappa^N)$, and $e \in \mathrm{Ed}(F^\kappa_N)$,  $b'_{I'J'e} = b_{IJe}$, where $I,J$ are the unique elements in $\mathrm{CC}(F^\kappa_N)$ such that $I \overset{e}{\leftrightarrow} J$.
\end{itemize}
  We then introduce the family $\{u'_{I'}\}$ indexed by $\mathrm{CC}(F'_{2N})$ such that 
  $$ \mE(F'_{2N}, \{u'_{I'}\}, \{b'_{I'J'e} \}) = \inf_{\{t'_{I'}\}}  \mE(F'_{2N}, \{t'_{I'}\}, \{b'_{I'J'e} \}). $$
 We then define a family $\{u_I\}$ indexed by $I \in \CC(F^\kappa_N)$, as follows: 
 $$ u_I = u'_{I'} \quad \text{ for $I'$ the single c.c. of  $F'_{2N}$ containing $I$}.    $$
 With this choice, we have 
 $$   \sum_{I',J' \in \mathrm{CC}(F'_{2N})} \sum_{\substack{e \in \mathrm{Ed}(F'_{2N}), \\ I' \overset{e}{\leftrightarrow} J'}} |b_{I'J'e}|^s   \le \sum_{I,J \in \mathrm{CC}(F^\kappa_N)} \sum_{\substack{e \in \mathrm{Ed}(F^\kappa_N), \\ I \overset{e}{\leftrightarrow} J}} |b_{IJe}|^s     $$
 and 
 $$ \sum_{I \in \mathrm{CC(F^\kappa_N)}} |I | \, |u_I|^2 \le C   \sum_{I' \in \mathrm{CC(F'_{2N})}} | I' | \, |u'_{I'}|^2 $$
 and 
 \begin{align*}
  \mE(F^\kappa_N, \{u_I\}, \{b_{IJe}\}) \: & \le \:   C \,  \mE(F'_{2N}, \{u'_{I'}\}, \{b'_{I'J'e}\}) \\ 
&  + \:  \sum_{I' \in \mathrm{CC(F'_{2N})}} \sum_{\substack{I,J \in \mathrm{CC(F^\kappa_N)}, \\ I,J \subset I'}. } \sum_{\substack{e \in \mathrm{Ed}(F^\kappa_N), \\ I \overset{e}{\leftrightarrow} J}}  |b_{IJe} - b_{JIe}|^2  \mu_e.
 \end{align*}
 Now, by definition of $F^\kappa$, connected components $I,J$ of $F^\kappa$  that are included in a single connected component $I'$ of $F'$ are at distance at least $\kappa$, so that $\mu_e \le |\ln \kappa|$. Hence, the last term is bounded by 
\begin{align*}
 C    |\ln \kappa| \sum_{I,J \in  \mathrm{CC(F^\kappa_N)}} \sum_{\substack{e \in \mathrm{Ed}(F^\kappa_N), \\ I \overset{e}{\leftrightarrow} J}}  |b_{IJe}|^2 \:  \le \:  C'    |\ln \kappa|  \,  |Q_N| \, \Big( \frac{1}{|Q_N|}\sum_{I,J \in  \mathrm{CC(F^\kappa_N)}} \sum_{\substack{e \in \mathrm{Ed}(F^\kappa_N), \\ I \overset{e}{\leftrightarrow} J}}  |b_{IJe}|^{s} \Big)^{\frac{2}{s}}.
   \end{align*}
The result follows easily  from assumption \eqref{H2} for $F'$ applied with $2N$ and previous inequalities.

\bspace
We have now all ingredients to perform the proof of our main Theorem \ref{thm_homog}, for a general admissible short $F'$. First, by Remark iii) after Theorem \ref{thm_homog},  $A^0 = (B^0)^{-1}$ is well-defined. As in section \ref{subsec_without}, given $f \in L^{6/5}(U)$, one has  $(u^\eps)_{\eps}$ bounded in  $H^1_0(U)$, and the goal is to show that any weak accumulation point $u^0$ satisfies \eqref{target_system}.

\mspace
Similarly, we introduce $\xi \in \R^3$,   $Z$ the minimizer of problem \eqref{VF_B0}, $\bar{Z}$ its extension by $-\xi$ in $\mF$, and   $\bar{z}(y,\omega) = \bar{Z}(\tau_y(\omega))$.  Let $p^\eps = \na u^\eps$. 
Let $F^{\kappa, \eps}$ defined as in \eqref{defi:rescaled}, replacing $F$ by $F^\kappa$ and $\delta_0$ by $\delta_0/2$ :
\begin{equation*} 
 I^{\kappa, \eps} := \eps I^\kappa \quad \forall I^\kappa \in \mathrm{CC}(F^\kappa), \quad  F^{\kappa, \eps} = \bigcup_{\substack{I^{\kappa,\eps} \subset U, \\ d(I^{\kappa,\eps}, \pa U) \ge \frac{\delta_0}{2} \eps}}  I^{\kappa, \eps}.
 \end{equation*}
We would like to extend the function $p^\eps\vert_{U\setminus F^{\kappa, \eps}}$ into some $P^{\kappa, \eps}$ satisfying 
$$ \div P^{\kappa,\eps} = f 1_{U\setminus F^\eps}  \quad \text{in }  U, $$
relying on  Proposition \ref{prop_div_extensions} and the fact that $F^\kappa$ satisfies \eqref{H2}. Note that for all $I^{\kappa,\eps}$ in  $\CC(F^{\kappa,\eps})$, one has the compability condition 
\begin{align*}
 \int_{\pa I^{\kappa,\eps}} \pa_\nu u^\eps & =  \sum_{I^\eps \in \CC(F^\eps), I^\eps \subset I^{\kappa,\eps}} \int_{\pa I^\eps} \pa_\nu u^\eps + \int_{\pa (I^{\kappa,\eps}\setminus F)} \pa_\nu u^\eps  \\
 &=  \int_{\pa (I^{\kappa,\eps}\setminus F^\eps)} \pa_\nu u^\eps = \int_{I^{\kappa, \eps}\setminus F^\eps} \Delta u^\eps =  \int_{I^{\kappa, \eps}\setminus F^\eps} f = \int_{I^{\kappa, \eps}} f 1_{U\setminus F^\eps}
\end{align*} 
But there is a little technicality here, due to the fact that $U\setminus F^{\kappa, \eps}$ is not necessarily  included in $U \setminus F^\eps$ so that {\it a priori} $\div p^\eps \neq f$ on $U \setminus F^{\kappa,\eps}$. This is due to the connected components $I$  of $F$ contained in connected components $I^\kappa$ of $F^\kappa$, such that $I^\eps \subset F^\eps$ and $I^{\kappa,\eps} \not \subset F^{\kappa, \eps}$.  Still, one can easily replace such connected components $I^{\kappa,\eps}$ by smaller connected closed sets $\tilde{I}^{\kappa,\eps}$ with $I^\eps \subset \tilde{I}^{\kappa,\eps} \subset I^{\kappa,\eps}$,  $d(\tilde{I}^{\kappa,\eps}, \pa U) \ge \frac{\delta_0}{2} \eps$, and such that Proposition \ref{prop_div_extensions} applies to  $\tilde{F}^{\kappa,\eps} = \cup \tilde{I}^{\kappa,\eps}$ instead of $F^{\kappa,\eps}$. Roughly speaking, one has just to erase in $I^{\eps,\kappa}$ the "beads" that intersect $\{x, d(x,\pa U) \ge \frac{\delta_0}{2} \eps\}$.   For brevity, we leave to the reader to verify that no complication occurs replacing $F^{\kappa,\eps}$ by $\tilde{F}^{\kappa, \eps}$, and keep the former notation. 

\mspace
Eventually, by applying Proposition \ref{prop_div_extensions}, we obtain a field $P^{\kappa, \eps} \in L^2(U)$ with 
$$ P^{\kappa,\eps}\vert_{U\setminus F^{\kappa,\eps}} = p^\eps, \quad \div P^{\kappa,\eps} = f 1_{U\setminus F^\eps},  $$
Moreover,  for all $t < s'$, where $s$ in the exponent in \eqref{H2}, if the moment condition   \eqref{moment_bound_F} is satisfied for $p = p(t)$ large enough, one has 
$$ \| P^{\kappa, \eps} \|_{L^{t}(U)} \le C_{\kappa,t}. $$
By diagonal extraction, there exists a subsequence in $\eps$ common to all $\kappa \in \Q \cap (0,1]$, and $P^{\kappa, 0}$ in $L^t(U)$ such that 
ignoring the subsequence in the notation: 
$$ P^{\kappa,\eps} \: \xrightarrow[\eps \rightarrow 0]{} P^{\kappa,0} \: \text{ weakly in } \: L^t(U), \quad \forall \kappa. $$ 
Let $\varphi \in C^\infty_c(U)$. Proceeding exactly  as in Section \ref{subsec_without} for the proof of \eqref{divcurl1}, we find
\begin{equation} \label{divcurl1'}  
\int_{U} \varphi(x) \na u^\eps(x) \cdot (\xi +  \bar{z}(x/\eps))  dx  \: \xrightarrow[\eps \rightarrow 0]{} \: \int_{U}  \varphi(x) \na u^0(x) \cdot \xi dx 
\end{equation}
The novel difficulty lies in the adaptation of the proof of \eqref{divcurl2}. We shall prove that 
\begin{equation} \label{divcurl2'}  
\int_{U} \varphi(x) \na u^\eps(x) \cdot (\xi +  \bar{z}(x/\eps))  dx  \: \xrightarrow[\eps \rightarrow 0]{}  \: \int_{U}  \varphi(x) P^{\kappa,0}(x) \cdot B^0 \xi dx + \eta(\kappa), \quad \eta(\kappa) \xrightarrow[\kappa \rightarrow 0]{}  0. 
\end{equation}
Comparing \eqref{divcurl1'} and \eqref{divcurl2'}, we get 
 $$ \int_{U}  \varphi(x) P^{\kappa,0}(x) \cdot B^0 \xi dx  \xrightarrow[\kappa \rightarrow 0]{} \int_{U}  \varphi(x) \na u^0(x) \cdot \xi dx 
 $$ 
 which shows that $P^{\kappa,0}$ converges in the sense of distributions to $A^0 \na u^0$. But we also have  
 $$ \div P^{\kappa, \eps} = f 1_{U \setminus F^\eps} \quad \text{ in } \: U $$
so that sending $\eps$ to zero,  
 $$ \div P^{\kappa,0} = (1-\lambda) f \quad \text{ in } \: U $$
and finally, sending $\kappa$ to zero, we get \eqref{target_system}. 

\mspace
It remains to show \eqref{divcurl1'}. We take into account that $\na u^\eps(x) = 0$,  $\xi +  \bar{z}(x/\eps) = 0$ in $F^\eps$ , and  write:  
\begin{align*}
& \int_{U} \varphi(x) \na u^\eps(x) \cdot (\xi +  \bar{z}(x/\eps))  dx  \\
 = & \int_{U\setminus F^{\kappa,\eps}} \varphi(x)  P^{\kappa,\eps}   \cdot (\xi +  \bar{z}(x/\eps))   dx +  \int_{F^{\kappa,\eps}} \varphi(x)  \na u^\eps   \cdot (\xi +  \bar{z}(x/\eps))   dx \\
= & \int_{U} \varphi(x)  P^{\kappa,\eps}   \cdot (\xi +  \bar{z}(x/\eps))   dx  + \int_{F^{\kappa,\eps}\setminus F^\eps} \varphi(x)  \na u^\eps   \cdot (\xi +  \bar{z}(x/\eps))  dx - 
  \int_{F^{\kappa,\eps}\setminus F^\eps} \varphi(x)  P^{\kappa,\eps}  \cdot (\xi +  \bar{z}(x/\eps))  dx \\
= & \: I^{\kappa,\eps} + J^{\kappa, \eps} - K^{\kappa, \eps}. 
 \end{align*}
 The first integral can be treated as in Section \ref{subsec_without}, resulting in 
 \begin{equation*}
 I^{\kappa,\eps}   \xrightarrow[\eps \rightarrow 0]{}   \int_{U}  \varphi(x) P^{\kappa,0}(x) \cdot B^0 \xi dx  
 \end{equation*}
 The second integral is bounded by 
 \begin{equation*}
|J^{\kappa,\eps}|   \le  \|\varphi\|_{L^\infty} \,  \|\na u^\eps\|_{L^2(U)} \, \| \xi +  \bar{z}(\cdot/\eps)\|_{L^2(F^{\kappa,\eps}\setminus F^\eps)}  \le C \, \| \xi +  \bar{z}(\cdot/\eps)\|_{L^2(U \cap \eps (F^\kappa \setminus F))} \\ 
\end{equation*}
 where we have used the uniform bound on $\na u^\eps$ in $H^1_0(U)$. From the ergodic theorem, we infer that 
 $$ \limsup_{\eps \rightarrow 0} |J^{\kappa,\eps}|  \le C \Big(  \int_{\Omega} \int_{Q_1} 1_{F^\kappa(\omega)\setminus F(\omega)}(y) \, | \xi + \bar{z}(y,\omega)|^2 \, dy \, d\mathbb{P}(\omega) \Big)^{1/2}. $$
The integral at the right-hand side converges to zero as $\kappa \rightarrow 0$: it follows from the dominated convergence theorem and the pointwise convergence to zero of  $(\omega, y) \rightarrow 1_{Q_1 \cap (F_\kappa(\omega) \setminus F(\omega))}(y)$,  because 
$$ \big( F^\kappa(\omega)\setminus F(\omega)\big) \cap Q_1 = \emptyset  \quad \text{for $\kappa$ large enough.}  $$ 
Hence,
\begin{equation}
\limsup_{\eps \rightarrow 0} |J^{\kappa,\eps}|  = o(\kappa)
\end{equation}

\mspace
We still have to control 
\begin{align*}
K^{\kappa,\eps}  =  \sum_{I^{\kappa,\eps} \in \CC(F^{\kappa,\eps})} \Big( & \int_{I^{\kappa,\eps}\setminus F^\eps} \varphi(x_{I^{\kappa,\eps}})   P^{\kappa,\eps}  \cdot (\xi +  \bar{z}(x/\eps))  dx \\
 + & \int_{I^{\kappa,\eps}\setminus F^\eps} \big( \varphi(x) - \varphi(x_{I^{\kappa,\eps}}) \big)  P^{\kappa,\eps}  \cdot (\xi +  \bar{z}(x/\eps))  dx \Big)  = K^{\kappa,\eps}_1 + K^{\kappa,\eps}_2
\end{align*}
where $x_{I^{\kappa,\eps}}$ is the center of mass of $I^{\kappa,\eps}$. We recall that   $\xi +  \bar{z}(x/\eps) = \na w^\eps(x)$ is a potential field that converges weakly in $L^2(U)$ to $B^0 \xi$ as $\eps \rightarrow 0$. By a proper choice of the additive constant in $w^\eps$, we can always assume that  $w^\eps$ converges weakly in $L^6(U)$ to $x \rightarrow (B^0 \xi) \cdot x$. 
Now, we write 
\begin{align*}
  \int_{I^{\kappa,\eps}\setminus F^\eps}  P^{\kappa,\eps}  \cdot (\xi +  \bar{z}(x/\eps))  dx = &  \int_{I^{\kappa,\eps}}  P^{\kappa,\eps}  \cdot \na w^\eps  dx  = \int_{I^{\kappa,\eps} \setminus F} f w^\eps + \int_{\pa I^{\kappa,\eps}} P^{\kappa,\eps} \cdot \nu  \, w^\eps \\
  = &  \int_{I^{\kappa,\eps} \setminus F^\eps} f w^\eps  + \int_{\pa I^{\kappa,\eps}} \pa_\nu u^\eps w^\eps = \int_{I^{\kappa,\eps} \setminus F^\eps} f  w^\eps  +   \int_{\pa (I^{\kappa,\eps}\setminus F^\eps)} \pa_\nu u^\eps  w^\eps 
 \end{align*}
 Note that for the last equality, we have used that $\int_{I^\eps} \pa_\nu u^\eps = 0$ for all $I^\eps \in \CC(F^\eps)$, and that $w^\eps$ is a constant in each $I^\eps$. Finally, 
 $$   \int_{\pa (I^{\kappa,\eps}\setminus F^\eps)} \pa_\nu u^\eps  w^\eps   = \int_{I^{\kappa,\eps}\setminus F^\eps} \na u^\eps \cdot \na w^\eps =  \int_{I^{\kappa,\eps}\setminus F^\eps} \na u^\eps \cdot (\xi + \bar{z}(\cdot/\eps)) $$
 resulting in 
\begin{align*}
\big| K^{\kappa,\eps}_1 \big| \le  \|\varphi\|_{L^\infty} \Big(    \|w^\eps \|_{L^6(U)} \, \| f 1_{F^{\kappa, \eps} \setminus F^\eps}\|_{L^{6/5}(U)}  +  \|\na u^\eps\|_{L^2(F^{\kappa,\eps} \setminus F^\eps)}  \|\xi + \bar{z}(\cdot/\eps)\|_{L^2(F^{\kappa,\eps} \setminus F^\eps)} \Big).  
\end{align*}
By using the uniform $L^6$ bound on $w^\eps$, the uniform $L^2$ bound on $\na u^\eps$ and the ergodic theorem, we end up with 
\begin{align*}
 \limsup_{\eps \rightarrow 0}  \big| K^{\kappa,\eps}_1 \big| & \le C   \|f\|_{L^{6/5}(U)} \, \Big(   \int_{\Omega} \int_{Q_1} 1_{F^\kappa(\omega)\setminus F(\omega)}(y)  \, dy \, d\mathbb{P}(\omega) \Big)^{5/6}     \\
 & + \int_{\Omega}  \int_{Q_1} 1_{F^\kappa(\omega)\setminus F(\omega)}(y) \, | \xi + \bar{z}(y,\omega)|^2 \, dy \, d\mathbb{P}(\omega) \Big)^{1/2} = o(\kappa)
 \end{align*}
 as seen above. It remains to treat 
 \begin{align*}
\big| K^{\kappa,\eps}_2 \big|  & \le \| \na \varphi \|_{L^\infty} \, \eps \,  \|P^{\kappa,\eps}\|_{L^2(F^{\kappa,\eps} \setminus F^\eps)}  \|\xi + \bar{z}(\cdot/\eps)\|_{L^2(F^{\kappa,\eps} \setminus F^\eps)}.  \\
& \le  C' \eps  \|P^{\kappa,\eps}\|_{L^2(F^{\kappa,\eps} \setminus F^\eps)} 
\end{align*}
using the ergodic theorem to bound the factor   $\|\xi + \bar{z}(\cdot/\eps)\|_{L^2(F^{\kappa,\eps} \setminus F^\eps)}$. The last difficulty is to bound $\|P^{\kappa,\eps}\|_{L^2(F^{\kappa,\eps} \setminus F^\eps)}$, because we only have so far a control in $L^t$, for any $t < s'$.   Still,  under a large moment bound on $\diam(I_{0,F})$, {\it cf. } \eqref{moment_bound_F},  we will now show that 
\begin{equation} \label{extra_int_Zkappa} 
 \|P^{\kappa,\eps}\|_{L^2(F^{\kappa, \eps}\setminus F^\eps)} \le C \eps^{3/2 - 3/t}. 
 \end{equation}
Indeed, following the proof of Proposition \ref{prop_div_extensions}, we see  that inside each inclusion $I^{\kappa, \eps}$, one has $P^{\kappa,\eps} = \na w +  \na v^{\kappa, \eps} $, where 
 $w$ solves 
 $$ \Delta w = f 1_{U \setminus F^\eps} \quad \text{ in } U, \quad w\vert_{\pa U} = 0 $$
 so that  in particular
 $$ \|\na w \|_{L^2(U)} \le C \|f\|_{L^{6/5}(U)},$$ 
while  $v^{\kappa, \eps}$ is the solution, mean-free over  $\mathring{I}^{\kappa,\eps}$, of 
 $$ \Delta  v^{\kappa, \eps} = 0 \quad \text{ in } \mathring{I}^{\kappa,\eps}, \quad \pa_\nu v^{\kappa,\eps}\vert_{\pa I^{\kappa,\eps}} = \pa_\nu (u^\eps-w)  $$
A crucial point is that $I^\kappa \setminus F$ is a union  of bridges, that are  at some uniform distance $r > 0$ from all other inclusions. Denoting $I^\kappa_j$, $j = 1 \dots$ such bridges, and introducing for all $j$, the  $r/2$ neighborhood  $V^\kappa_j$ of $I^\kappa_j$, we claim that 
 \begin{align*}
  \|\na v^{\kappa,\eps}\|^2_{L^2(\eps I^\kappa_j)} & \le  C   \left(  \eps^{-2} \|v^\kappa\|^2_{L^{2}(\eps(V^\kappa_j \cap I^\kappa))}   + \| \pa_\nu (u^\eps-w) \|_{H^{-1/2}(\eps ( \pa I^{\kappa} \cap \,  V^\kappa_j))}  \right) \\ 
  & \le  C'   \left(  \eps^{-2} \|v^\kappa\|^2_{L^{2}(\eps(V^\kappa_j \cap I^\kappa))}   +  \|\na (u^\eps-w)\|^2_{L^2(\eps (V^\kappa_j \setminus I^\kappa))} \right)  
  \end{align*} 
 Indeed, by a scaling argument, it is enough to consider the case $\eps = 1$. The first inequality follows then from a  standard elliptic regularity result, while the second one follows from the usual bound 
 $$ \| U \cdot \nu \|_{L^2(\pa \mathcal{O})} \le C \left(  \| \div U \|_{L^2(\mathcal{O})} + \| U \|_{L^2(\mathcal{O})} \right) $$ 
 applied in the domain $\mathcal{O} = V^\kappa_j \setminus I^\kappa$. Here, we rely on the fact that such domains are far away from other inclusions, so that the constant can be taken uniform in $j$. 
 
 \mspace
 Summing over all $j$'s and all inclusions $I^\kappa$, we end up with 
\begin{align*}
 \|\na v^{\kappa,\eps}\|^2_{L^2(F^{\kappa,\eps} \setminus F^\eps)} & \le  C   \left(  \eps^{-2} \sum_{I^{\kappa,\eps}} \|v^\kappa\|^2_{L^2(I^{\kappa,\eps})}   +  \|\na (u^\eps-w)\|_{L^2(U)}^2  \right) \\
 & \le C'   \Big(  \eps^{-2} \sum_{I^{\kappa,\eps}} \|v^\kappa\|^2_{L^2(I^{\kappa,\eps})}   +  1 \Big) 
 \end{align*}
 As $s > 6$, one has $s' > \frac{6}{5}$. Hence, for $t$ close enough to $s'$, $W^{1,t}(I^{\kappa,\eps}) \subset L^2(I^{\kappa,\eps})$, with 
 $$ \|v^\kappa\|_{L^2(I^{\kappa,\eps})} \le C \, \diam(I^\kappa)^r \, \eps^{\frac{5}{2} - \frac{3}{t}} \|\na v^\kappa\|_{L^t(I^{\kappa,\eps})}  $$
 Again, the power of $\eps$ is deduced from a scaling argument, while the factor $\diam(I^\kappa)^r$ comes from the Poincaré inequality \eqref{PW} applied to the mean-free function $v^\kappa$. We get eventually  
\begin{align*} 
\|P^{\kappa,\eps}\|^2_{L^2(F^{\kappa, \eps}\setminus F^\eps)} & \le C  \Big(  \eps^{3 - \frac{6}{t}}  \sum_{I^{\kappa,\eps}} \diam(I^\kappa)^{2r} \|P^{\kappa,\eps}\|_{L^t(I^{\kappa,\eps})}^2  + 1 \Big)  \\
& \le C  \Big(  \eps^{3 - \frac{6}{t}}      \sum_{I^{\kappa,\eps}} \diam(I^\kappa)^{2r} \|P^{\kappa,\eps}\|_{L^t(I^{\kappa,\eps})}^t \| P^{\kappa, \eps} \|_{L^t(U)}^{2-t}   + 1 \Big) \\
& \le C'  \Big(  \eps^{3 - \frac{6}{t}}   \sum_{I^{\kappa,\eps}} \diam(I^\kappa)^{2r} \|P^{\kappa,\eps}\|_{L^t(I^{\kappa,\eps})}^t + 1 \Big)
\end{align*}   
Using H\"older inequality, we find for any $\tilde t$ such that $t < \tilde t < s'$: 
$$  \sum_{I^{\kappa,\eps}} \diam(I^\kappa)^{2r} \|P^{\kappa,\eps}\|_{L^t(I^{\kappa,\eps})}^t  \le C \Big( \int_{U} |\diam(I_{x/\eps, F^\kappa})|^p\Big)^{2/p}  \|P^{\kappa,\eps}\|_{L^{\tilde t}(U)}^{t/\tilde t} $$ 
  where $p = 2 r (\tilde t/t)'$. The first factor is bounded thanks to the ergodic theorem and \eqref{moment_bound_F}, while the second one is bounded thanks to the uniform bound for $P^{\kappa, \eps}$ in $L^{\tilde t}(U)$. Inequality    \eqref{extra_int_Zkappa} follows. 
  
\mspace
Back to $K^{\kappa,\eps}_2$, we deduce that,
$$ |K^{\kappa,\eps}_2| \le C \eps^{5/2 - 6/t} $$
which goes to zero taking $t$ close enough to $s'$, by the condition $s < 6$. This concludes the proof of the theorem.

\section{Discussion of the assumptions} \label{section_disc_assumptions}
We start here an extended discussion of the  assumptions \eqref{H1}-\eqref{H2}. We remind the definition of $\mE$: 
\begin{equation*}
\mE\Big(F, \{u_I\}, \{b_{IJe}\}\Big) =   \sum_{I,J \in \mathrm{CC}(F)} \sum_{\substack{e \in \mathrm{Ed}(F), \\ I \overset{e}{\leftrightarrow} J}}  \mu_e  |b_{IJe} - b_{JIe} + u_I - u_J |^2 + \sum_{I \in \mathrm{CC}(F)}| I | \,   |u_I|^2. 
\end{equation*}
It follows from this definition that for  closed sets  $\overline{F} = F \cup G$ where the union is disjoint (so that  $Gr(\overline{F})$ is deduced from $Gr(F)$ by the addition of nodes and/or edges), one has 
\begin{equation} \label{prop_E_1}
\mE\Big(F, \{u_I\}, \{b_{IJe}\}\Big) \le \mE\Big(\overline{F}, \{\overline{u}_I\}, \{\overline{b}_{I,J,e}\}\Big)
\end{equation}
for any extensions $\{\overline{u}_I\}, \{\overline{b}_{I,J,e}\}$ of $\{u_I\},  \{b_{IJe}\}$, meaning that 
$$ \overline{u}_I = u_I, \quad \forall I \in \mathrm{CC}(F), \quad \overline{b}_{I,J,e} = b_{IJe}, \quad \forall I,J \in \mathrm{CC}(F), \quad e \in \mathrm{Ed}(F), \: 
I \overset{e}{\leftrightarrow} J. $$
Indeed, the sum in the right-hand side of \eqref{prop_E_1} has more (positive) terms than the one at the left-hand side. 

\mspace
\subsection{Discussion around \eqref{H1}}

\subsubsection*{\eqref{H1} for a short implies \eqref{H1}}
An important property of the discrete energy above concerns closed sets $F \subset F'$ with $F'$ a short of $F$, {\it cf.} Definition \ref{defi:short_inclusions}. Given  $\{u'_{I'}\}$, indexed by $I' \in \mathrm{CC}(F')$,  a family associated to $F'$, one can associate a family $\{u_I\}$ indexed by $I \in \mathrm{CC}(F)$ as follows: 
\begin{equation} \label{defuiu'i}
u_I  :=  \xi \cdot x_I + u'_{I'} - \xi \cdot x_{I'}, \quad \forall I \in \mathrm{CC}(F), \: I' \in \mathrm{CC}(F') \quad \text{such that } \: I \subset I'. 
\end{equation}
We remind that for any set $S$, $x_S$ is the center of mass of $S$.  Note that by definition of a short, any connected component of $F'$ contains at least one connected component of $F$. We then claim that  for $\mathrm{CC}(F)$ finite, 
\begin{equation} \label{prop_E_2}
\mE\Big(F, \{u_I\}, \{\xi \cdot x_I \}\Big) \le C  \Big( \mE\Big(F', \{u'_{I'}\}, \{\xi \cdot x_{I'}\}\Big)  +  \sum_{I' \in \mathrm{CC}(F')}  | I' | \, \diam(I')^2   \Big). 
\end{equation}
Indeed, introducing $v_I := u_I - \xi \cdot x_I$, resp. $v'_{I'} = u_{I'} - \xi \cdot x_{I'}$, we write 
\begin{align*}
\mE\Big(F, \{u_I\}, \{\xi \cdot x_I \}\Big) =  \sum_{I,J \in \mathrm{CC}(F)}  \mu_{I,J}  |v_I - v_J |^2 + \sum_{I \in \mathrm{CC}(F)} | I | \,  |v_I + \xi \cdot x_I|^2  =: E_1 + E_2, 
\end{align*}
where $\mu_{I,J} =  \sum_{e \in \mathrm{Ed}(F),  I \overset{e}{\leftrightarrow} J} \mu_e$. 
As $v_I = v_J$ when $I,J$ are included in the same inclusion of $F'$, the first term can be bounded by 
\begin{align*}
E_1 & \le  \sum_{I' \neq J' \in \mathrm{CC}(F')} \sum_{\substack{I,J \in \mathrm{CC}(F), \\ I \subset I', J \subset J'  }}   \mu_{I,J}  |v_I - v_J |^2 \\
& \le \sum_{I' \neq J' \in \mathrm{CC}(F')} \mu_{I',J'} \sum_{\substack{I,J \in \mathrm{CC}(F), \\ I \subset I', J \subset J', I \leftrightarrow J }}  |v_I - v_J |^2 
& \le C \sum_{I' \neq J' \in \mathrm{CC}(F')}  \mu_{I',J'}  |v'_{I'} - v'_{J'} |^2 \\
& \le  C \mE\Big(F', \{u'_{I'}\}, \{\xi \cdot x_{I'}\}\Big).
\end{align*}
For the third inequality, we have used the fact that $v_I = v'_{I'}$, $v_J = v'_{J'}$ by our definition of the family $\{u_I\}$, as well as assumption (G2):  the number of gaps between two inclusions $I'$ and $J'$ is finite,  so that 
$$ \sharp \{ I,J \in \mathrm{CC}(F), \quad  I \subset I', J \subset J' , I \leftrightarrow J \} \le C, \quad \text{ for some uniform constant $C$.} $$
 Moreover,  $v_I = v'_{I'}$, $v_J = v'_{J'}$ by our definition of the family $\{u_I\}$). Eventually, 
\begin{align*}
E_2 & = \sum_{I' \in \mathrm{CC}(F')} \sum_{I \in \mathrm{CC}(F), I \subset I'} | I | \, |v'_{I'} + \xi \cdot x_I|^2 \\ 
& \le  2 \sum_{I' \in \mathrm{CC}(F')} \sum_{I \in \mathrm{CC}(F), I \subset I'}  | I | \, ( |v_{I'} + \xi \cdot x_{I'}|^2  +  |\xi \cdot x_{I} - \xi \cdot x_{I'}|^2 ) \\
& \le C \Big( \sum_{I' \in \mathrm{CC}(F')}  | I' | \,  |v_{I'} + \xi \cdot x_{I'}|^2  +  | I' | \, \diam(I')^2 \Big)\\
\end{align*}
which yields \eqref{prop_E_2}. We are now ready to show

\begin{lemma}  \label{lemma_H1H2}

\sspace
Let $F'$ a short of $F$ satisfying 
\begin{align*}
& \limsup_{N \rightarrow +\infty} \frac{1}{|Q_N|} \sum_{I' \in \CC(F'_N)} | I' | \, \diam(I')^2 < +\infty \\
& \sup_{I' \in CC(F'), I' \cap Q_N \neq \emptyset } \diam(I') = o(N)
\end{align*}
 as well as  \eqref{H1}. Then, $F$ itself satisfies  \eqref{H1}.
\end{lemma}

\begin{remark} \label{rem_H1_H1short}
In the case where $F = F(\omega)$ is an admissible set of inclusions, and $F'$ is an admissible short of $F$, see Definition \ref{defi:short_inclusions}, the first two conditions above are consequences of the relation 
$$\E \, \diam(I_{0,F'})^3 < +\infty,$$ 
see the proof of   \eqref{uniform_bound_inclusion}. 
\end{remark}

\noindent
{\em Proof}. One must realize once again that for $I \in \mathrm{CC}(F_N)$, and $I' \in \mathrm{CC}(F')$ such that $I \subset I'$, one does not have necessarily $I' \in \mathrm{CC}(F'_N)$. Indeed, it may happen that $I$ is contained in $Q_N$, while $I'$ crosses $\partial Q_N$.  Still, by the second assumption of the lemma,  $I' \in \mathrm{CC}(F'_{2N})$, for $N$ large enough. It follows that $F'_{2N}$ can be seen as the short of a closed set $\overline{F}_N$ with $Q_{2N} \supset \overline{F}_N$, and $\overline{F}_N = F_N \cup G_N$ with a disjoint union.  

\mspace 
Let now $\{u'_{I'}\}$, indexed by $I' \in \mathrm{CC}(F'_{2N})$, satisfying 
$$ \mE(F'_{2N}, \{u'_{I'}\}, \{ \xi \cdot x_{I'} \}) = \inf_{\{t'_{I'}\}}  \mE(F'_{2N}, \{t'_{I'}\}, \{ \xi \cdot x_{I'} \}).   $$
We associate to $\{u'_{I'}\}$ the family $\{u_{I}\}$, indexed by $I \in \overline{F}_N$,  as in \eqref{defuiu'i} (replacing $F$ by $\overline{F}_N$ and $F'$ by $F'_{2N}$).  By \eqref{prop_E_2}, we have 
\begin{align*}
  \mE(\overline{F}_{N}, \{u_I\}, \{ \xi \cdot x_{I} \}) & \le C \left(  \mE(F'_{2N}, \{u'_{I'}\}, \{ \xi \cdot x_{I'} \})  + \sum_{I' \in \CC(F'_{2N})} | I' | \, \diam(I')^2 \right).
  \end{align*}
Furthermore,  by using \eqref{prop_E_1}, we have 
  $$  \mE(F_N, \{u_{I}\}, \{ \xi \cdot x_{I} \})  \le  \mE(\overline{F}_{N}, \{u_{I}\}, \{ \xi \cdot x_{I} \})  $$
  so that combining everything we find  
 \begin{align*}
 \inf_{\{t_I\}}  \mE(F_{N}, \{t_{I}\}, \{ \xi \cdot x_{I'} \}) & \le \mE(F_{N}, \{u_{I}\}, \{ \xi \cdot x_{I} \})  \le  \mE(\overline{F}_{N}, \{u_{I}\}, \{ \xi \cdot x_{I} \})   \\
 & \le C \Big(  \mE(F'_{2N}, \{u'_{I'}\}, \{ \xi \cdot x_{I'} \})  +  \sum_{I' \in \CC(F'_{2N})} | I' | \, \diam(I')^2 \Big) \\
 & \le C \Big( \inf_{\{t'_{I'}\}}  \mE(F'_{N+2D'}, \{t'_{I'}\}, \{ \xi \cdot x_{I'} \}) + \sum_{I' \in \CC(F'_{2N})} | I' | \, \diam(I')^2 \Big). 
\end{align*}
As $F'$ satisfies the first assumption of the lemma and  \eqref{H1}, dividing by $|Q_N|$ and sending $N$ to infinity, we see that $F$ satisfies \eqref{H1}. 

\subsubsection*{Clusters with a moment bound on the diameter}
We prove here 
\begin{lemma}
Let $F = F(\omega)$ an admissible set of inclusions  satisfying the moment bound
$$ \E  \, \diam(C_{0,F})^2 < +\infty $$
where $C_{0,F}$ is the cluster of $F$ containing $0$, {\it cf. }Remark \ref{rem_I0_C0}. Then, $F$ satisfies \eqref{H1}.
\end{lemma}

\begin{proof}
We want to consider the discrete energy $\mE(F_N, \{ u_I\}, \{\xi \cdot x_I\})$ for a suitable choice of $u_I$'s. For each cluster $\mC$ of $F$, and for each $I \in \CC(F_N)$, $I \subset \mC$, we set 
$u_I := -\xi \cdot (x_I - x_{\mC})$, where as before $x_{S}$ is the center of mass of $S$. Then, 
\begin{align*} 
& \frac{1}{|Q_N|} \mE(F_N, \{ u_I\}, \{\xi \cdot x_I\})   = \frac{1}{|Q_N|} \sum_{I \in \CC(F_N)} | I | \,  |u_I|^2\\
& \leq \frac{|\xi|^2}{|Q_N|} \sum_{\mC \atop \text{cluster of $F$}} \sum_{\substack{I \in \CC(F_N) \\ I \subset \mC }}  | I | \, |x_I-x_{\mC}|^2 
 \leq \frac{|\xi|^2}{|Q_N|} \sum_{\mC \atop \text{cluster of $F$}} \sum_{\substack{I \in \CC(F_N) \\ I \subset \mC }} | I | \, \diam(\mC)^2 \\
& \leq  \frac{|\xi|^2}{|Q_N|} \sum_{\mC \atop \text{cluster of $F$}} \sum_{\substack{I \in \CC(F_N) \\ I \subset \mC }} \int_I \diam(C_{x,F})^2 dx  \le  \frac{|\xi|^2}{|Q_N|} \int_{Q_N} \diam(C_{x,F})^2 dx
\end{align*}
Sending $N$ to infinity, we end up with
$$ \limsup_N \frac{1}{|Q_N|} \mE(F_N, \{ u_I\}, \{\xi \cdot x_I\}) \le |\xi|^2 \, \E \diam(C_{0,F})^2 < +\infty $$
which shows that $F$ satisfies \eqref{H1} and concludes the proof of the lemma. 
\end{proof}

 \subsubsection*{Link with the graph laplacian}
\bspace
Another interesting result starts with the following observation. Assumption \eqref{H1} is verified by $F$ if almost surely, there exists $M = M(\omega)$ such that for any $N>0$,  any $\xi \in \R^3$, 
$$\inf_{\{u_I\}} \: \frac{1}{|Q_N|} \mE\Big(F_N, \{u_I\}, \{\xi \cdot x_I\} \Big) \leq M |\xi|^2.$$
Writing the Euler equations of the minimization problem at the left-hand side leads  to the following linear system 
$$u_I \: + \:  \sum_{J \in \CC(F_N)} \mu_{I,J} (u_I - u_J)  = -  \sum_{J \in \CC(F_N)} \mu_{I,J} (\xi \cdot x_I -  \xi \cdot x_J) \quad \forall I \in \CC(F_N)$$
where we remind that $\mu_{I,J}:=  \sum_{e \in \Ed(F),  I \overset{e}{\leftrightarrow} J} \mu_e$ (and is therefore zero if $I$ and $J$ are not linked by an edge). This linear system can be written into the matrix form :
\begin{equation*}
    (\mathbb{L}_{F_N}+\mathbb{I}) \mathbb{U} = - \mathbb{L}_{F_N} \mathbb{S}
\end{equation*}
where  $\mathbb{U}=((u_I)_{I \in \CC(F_N)})^t$,  $\mathbb{S} =(\xi \cdot x_I)_{I \in \CC(F)})^t$ and $\mathbb{L}_{F_N}$ is a symmetric matrix of size $|\CC(F_N)| \times |\CC(F_N)|$ defined by
\begin{align*}
   \Big[ \mathbb{L}_{F_N} \Big]_{IJ}= \begin{cases}
       \sum_{K \in \CC(F_N)} \mu_{I,K} & \text{if $I=J$}\\
         -\mu_{I,J} &   \text{if $I \neq J$}
    \end{cases}.    
\end{align*}
This kind of matrix arises in the graph literature as the {\em weighted laplacian matrix} for the pondered unoriented graph $Gr(F_N)$, see the first section of \cite{berkolaiko2013introduction}. It can be seen as  a discrete version of a continuous problem of the form
\begin{equation*}
    \left\{
      \begin{aligned}
        - \div(\mu \na u_N) + u_N=   \div(\mu \xi)    \ \text{in} \  Q_N\\
        \mu \na u_N \cdot \nu = 0 \text{ on }  \partial Q_N
      \end{aligned}
    \right. \iff u_N=\argmin_{v \in H^1(Q_N)} \int_{Q_N} \big(\mu \nabla v \cdot \na v + |v|^2 -  2\xi \cdot \na v \big) \dx.
\end{equation*}
The energy of this problem is a {\em superadditive quantity over sets} and one can expect our discrete minimization problem to verify a similar property. We state

\begin{lemma}
Let $\P$ a bounded set of $\R^3$ and denote 

$$F_{\P} = \bigcup_{\substack{I \in CC(F) 
, \\ I \subset \P}} I, \quad Gr(F_\P)= (\CC(F_\P), \Ed(F_\P)) .$$
The quantity $\mathcal{H}(\P):=  \inf_{\{u_I\}}\mE( F_\P, \{u_I\}, \{\xi \cdot x_I \})$ is superadditive over sets, which means that for any decomposition  $\P= \cup_{k=1}^M \P_k$ where $\{\P_1, \dots, \P_M \}$ are pairwise disjoint, one has almost surely

$$\mathcal{H}(\P) \geq \sum_{k=1}^M \mathcal{H}(\P_k).$$

\end{lemma}

\begin{proof}
It is enough to prove the results for a simple decomposition  $\P=\P_1 \cup \P_2$, with a boundary $\Sigma$ between $\P_1$ and $\P_2$. This leads to the following decomposition of the nodes of the graph

$$ \CC(F_\P) = \CC(F_{\P_1}) \cup \CC(F_{\P_2}) \cup \CC_\Sigma$$
where $\CC_\Sigma$ is the set of  all  connected of components of $F_\P$ that intersect the boundary $\Sigma$ without being included in $\P_1$ or $\P_2$. The following figure explains the decomposition. In white is the graph $Gr(F_{\P_1})$ and in black is the graph $Gr(F_{\P_2})$. In dotted lines is what remains from the graph of   $\CC(F_\P)$.  
\definecolor{gray1}{gray}{0.85}
\definecolor{gray2}{gray}{0.67}
\definecolor{gray3}{gray}{0.55}
\definecolor{gray4}{gray}{0.90}
\begin{figure}[H]
    \centering
    \psset{unit=0.8cm}
\begin{pspicture}(-4,-3)(4,3)
\psframe[fillstyle=solid,linewidth=1.5pt,linecolor=black, fillcolor=gray4] (-4,-3)(4,3)
\pscircle[fillstyle=solid,fillcolor=gray, linecolor=black, linewidth=1.2pt](1,0){1}
\pscircle[fillstyle=solid,fillcolor=gray, linecolor=black, linewidth=1.2pt](2,1.8){1}
\pscircle[fillstyle=solid,fillcolor=gray, linecolor=black, linewidth=1.2pt](-2,-1.8){1}
\pscircle[fillstyle=solid,fillcolor=gray, linecolor=black, linewidth=1.2pt](-2.9,1){0.9}
\pscircle[fillstyle=solid,fillcolor=gray, linecolor=black, linewidth=1.2pt](2.5,-1){0.7}
\pscircle[fillstyle=solid,fillcolor=gray, linecolor=black, linewidth=1.2pt](-0.3,1.6){0.9}
\pscircle[fillstyle=solid,fillcolor=gray, linecolor=black, linewidth=1.2pt](-1.3,0){0.9}
\pscircle[fillstyle=solid,fillcolor=gray, linecolor=black, linewidth=1.2pt](1,-2){0.9}
\pscircle[fillstyle=solid,fillcolor=gray, linecolor=black, linewidth=1.2pt](3.2,0.2){0.65}

\pscurve[showpoints=false,linewidth=3pt,linecolor=black,linestyle=dotted] (1,0) (2,1.8) (2,1.8)
\pscurve[showpoints=false,linewidth=3pt,linecolor=black,linestyle=dotted] (1,0) (1,-2) (1,-2)
\pscurve[showpoints=false,linewidth=4pt,linecolor=black] (2.5,-1) (3.2,0.2) (3.2,0.2)
\pscurve[showpoints=false,linewidth=4pt,linecolor=white] (-2.9,1) (-1.3,0) (-1.3,0)
\pscurve[showpoints=false,linewidth=4pt,linecolor=white] (-1.3,0) (-2,-1.8) (-2,-1.8)
\pscurve[showpoints=false,linewidth=4pt,linecolor=white] (-1.3,0) (-0.3,1.6) (-0.3,1.6)
\pscurve[showpoints=false,linewidth=3pt,linecolor=black,linestyle=dotted] (1,0) (-0.3,1.6)  (-0.3,1.6) 
\pscurve[showpoints=false,linewidth=3pt,linecolor=black,linestyle=dotted] (1,0) (2.5,-1) (2.5,-1)

\psline[linewidth=1pt,linestyle=dashed] (1,-3) (2.5,3)
\rput(1.4,-3.3){\large$\Sigma$ } 
\rput(-3.2,-2.5){\large$\P_1$ } 
\rput(3,-2.5){\large$\P_2$ } 
\psline[linewidth=1pt] (4,2.2) (4.5,2.2)
\rput(4.8,2.2){\large$\P$} 
\end{pspicture}
    \caption{Separation of the domain $\P= \P_1 \cup \P_2$ with the boundary $\Sigma$.}
    \label{fig:separation of the domain}
\end{figure}
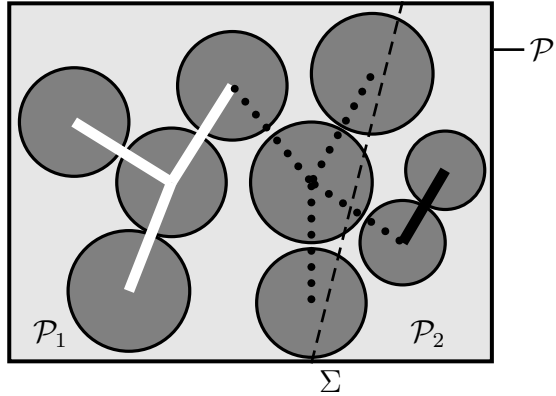
\noindent
Let $(u_I)_{I \in \CC(F_\P)}$ the solution that minimizes $\mathcal{H}(\P)$, that we split, up to some permutations, into a vector of the form $(u^1,u^2,u^\Sigma) \in \R^{|\CC(F_{\P_1})|} \times \R^{|\CC(F_{\P_2})|}  \times \R^{|\CC_\Sigma|}$. We compute
\begin{align*}
 \mathcal{H}(\P) -  \mathcal{H}(\P_1) - \mathcal{H}(\P_2) \:  \ge \: &  \:  \mathcal{H}(\P) - \mE( F_{\P_1}, u^1, \{\xi \cdot x_I \}) -  \mE( F_{\P_2}, u^2, \{\xi \cdot x_I \}) \\
    \ge  \: &\sum_{\substack{I \in \CC(F_{\P_1}), \\J \in \CC_\Sigma}}  \mu_{I,J} | \xi \cdot (x_I-x_J) + u_I^1 - u_J^\Sigma|^2\\
    + \: &\sum_{\substack{I \in \CC(F_{\P_2}), \\J \in \CC_\Sigma}} \mu_{I,J} |\xi \cdot (x_I-x_J) + u_I^2 - u_J^\Sigma|^2 + \sum_{I \in \CC_\Sigma}| \, I \, | \,  |u_I^\Sigma|^2 \: \geq \: 0
\end{align*}
which ends the argument.
\end{proof}
\noindent
We can then use the superadditive ergodic theorem ({\it cf.} \cite{akcoglu1981ergodic,modica1986nonlinear}), which yields a sufficient condition for \eqref{H1} to hold :
\begin{proposition}
\label{prop_super_additivity_H1}
Assume that 
$$ \sup_{N}  \mathbb{E} \frac{1}{|Q_N|} \mathcal{H}(Q_N)  < + \infty,$$
 then the admissible set $F$ verifies assumption \eqref{H1}. 
\end{proposition}

%
%
%
%
%
%

\subsection{Discussion around \eqref{H2}}

\subsubsection*{Logarithmic moment bound}
\begin{lemma} \label{lemma_logH2}
Let $F$ a random closed set satisfying (G1)-(G2). Let  $s \in (2,\infty)$ and $k= \big(\frac{s}{2}\big)'$. If 
\begin{equation} \label{log_bound_mue_bis}
  \limsup_{N \rightarrow +\infty} \frac{1}{|Q_N|} \sum_{e \in \mathrm{Ed}(F_N)}   \mu_e^k < +\infty 
 \end{equation}
then $F$ verifies \eqref{H2} with exponent $s$.
\end{lemma}

\begin{proof}
Let $N \in \N^*$, and $\{b_{IJe}\}$ a family indexed by $I,J$ in $\CC(F_N)$ and $e \in \Ed(F_N)$ such as $I \overset{e}{\leftrightarrow } J$.  We want to control the quantity  $\mE\Big(F_N, \{u_I\}, \{b_{IJe}\}\Big)$ by $$|Q_N|\Big( \frac{1}{|Q_N|} \sum_{I,J \in \CC(F_N)} \sum_{\substack{e \in \Ed(F_N), \\ I \overset{e}{\leftrightarrow} J}} |b_{IJe}|^s \Big)^{2/s}$$ for a suitable choice of $\{u_I\}_{I \in \CC(F_N)}$. With our logarithmic bound, we may  take  $u_I=0$ for all $I$. We find  
$$\mE\Big(F_N, \{u_I\}, \{b_{IJe}\}\Big) = \sum_{I,J \in \CC(F_N)} \sum_{\substack{e \in \Ed(F_N), \\ I \overset{e}{\leftrightarrow} J}} |b_{IJe}-b_{JIe}|^2 \mu_{e}.$$
Using Hölder inequality with $k=\frac{s}{s-2}$ and $k' = \frac{s}{2}$, we have
\begin{align*}
    & \sum_{I,J \in \CC(F_N)} \sum_{\substack{e \in \Ed(F_N), \\ I \overset{e}{\leftrightarrow} J}} |b_{IJe}-b_{JIe}|^2 \mu_{e} \\
    & \leq  \Big( \sum_{I,J \in \CC(F_N)} \sum_{\substack{e \in \Ed(F_N), \\ I \overset{e}{\leftrightarrow} J}} \mu_e^k \Big)^{\frac{1}{k}}   \Big(\sum_{I,J \in \CC(F_N)} \sum_{\substack{e \in \Ed(F_N), \\ I \overset{e}{\leftrightarrow} J}} |b_{IJe}-b_{JIe}|^s\Big)^{\frac{2}{s}}\\
       & \le C |Q_N| \Big( \frac{1}{|Q_N|}\sum_{e \in \Ed(F_N)}  \mu_e^{k} \Big)^{\frac{s-2}{s}}  \Big( \frac{1}{|Q_N|} \sum_{I,J \in \CC(F_N)} \sum_{\substack{e \in \Ed(F_N), \\ I \overset{e}{\leftrightarrow} J}} |b_{IJe}|^s\Big)^{\frac{2}{s}} 
\end{align*}
which together with the bound  \eqref{log_bound_mue_bis}  gives the expected result.
\end{proof}

\subsubsection*{\eqref{H2}  implies \eqref{H1}}
\begin{lemma} \label{lem_H2_implies_H1}
Let $F$ an admissible set of inclusions, $F'$ an admissible short of $F$, $s > 3$. If
$$ \E \diam(I_{0,F'})^s  < +\infty $$
and  if \eqref{H2} is satisfied by $F'$, then \eqref{H1} is satisfied by $F$.
\end{lemma}
\begin{proof}
By Lemma  \ref{lemma_H1H2} and Remark  \ref{rem_H1_H1short}, it is enough to show that  $F'$ satisfies \eqref{H1}.  
 This can be seen by setting '
$$ b_{IJe} = \xi \cdot \big(x_I - \frac{x_{I,\alpha}+ x_{J,\beta}}{2}\big), \quad \forall I,J \in \mathrm{CC}(F'_N), \quad e = [x_{I,\alpha}, x_{J,\beta}] \in \mathrm{Ed}(F'_N), \\ I \overset{e}{\leftrightarrow} J.$$
Clearly,  from assumptions (G1)-(G2), one has $|b_{IJe}| \le C|\xi| (\diam(I) + \delta) \le C' \diam(I)$, and
$\sharp \mathrm{Ed}(F'_N) \le C |Q_N|$. It follows that 
\begin{align*} 
\frac{1}{|Q_N|} \sum_{I,J \in \mathrm{CC}(F'_N)} \sum_{\substack{e \in \mathrm{Ed}(F'_N), \\ I \overset{e}{\leftrightarrow} J}} |b_{IJe}|^s  & \le \frac{C}{|Q_N|}  \, |\xi| \sum_{I \in \mathrm{CC}(F_N)} \diam(I)^{s}  \\
& \le \frac{C'}{|Q_N|}  \, |\xi| \sum_{I \in \mathrm{CC}(F_N)} |\, I \, | \, \diam(I)^{s}  \\
& \le \frac{C'}{|Q_N|} \, |\xi|  \int_{Q_N} \diam(I_{x,F})^s dx 
\end{align*}
Thanks to the ergodic theorem, we end up with 
$$ \limsup_N \, \frac{1}{|Q_N|} \sum_{I,J \in \mathrm{CC}(F_N)} \sum_{\substack{e \in \mathrm{Ed}(F_N), \\ I \overset{e}{\leftrightarrow} J}} |b_{IJe}|^s \le C' |\xi| \, \E \, I_{0,F}^s. $$
\end{proof}

\subsubsection*{Cycle-free graphs}
The previous lemma comes from a trivial choice of the family $\{u_I\}$. We will show that if the multigraph of inclusions is cycle-free, there is a better choice,   that enables to relax the logarithmic moment bound, and prove Corollary  \ref{corollary_cyclefree_new}.   

\begin{proof}[Proof of Corollary \ref{corollary_cyclefree_new}]
Let $b_{IJe}$ a family indexed by the triplet $I,J \in \CC(F_N)$, $e \in \Ed(F_N)$, with $I \overset{e}{\leftrightarrow} J$. As $Gr(F)$ is cycle-free, there is a single edge $e$ linking the nodes $I$ and $J$, so that we can note $b_{IJ}$ instead of  $b_{IJe}$ for brevity. 
Given an arbitrary reference inclusion $I_{\mC_N}$  in each cluster $\mC_N$ of $F_N$, we then define a family $u_I$, $I \in \CC(\mC_N)$, as follows.  For all such inclusion, there is a unique integer $k_I \in \N$ and a unique branch $I_0 = I_{\mC_N}, \dots,  I_{k_I} = I$ connecting $I_{\mC_N}$ to $I$ (with $k_I=0$ in the case $I = I_{\mC_N}$). We define 
$$u_I := \sum_{j=0}^{k_I-1} b_{I,j}' ,  \ b_{I,j}' :=b_{I_j I_{j+1}} - b_{I_{j+1} I_j}$$
Note that in particular, $u_{I_{\mC_N}} = 0$. Doing this for all cluster $\mC_N$, we get a family $(u_I)_{I \in \CC(F_N)}$.  We then compute 
\begin{align*}
    \frac{1}{|Q_N|} \mE(F_N,\{u_I\},\{ b_{IJ} \} ) & = \frac{1}{|Q_N|} \sum_{I \in \CC(F_N)} | I | \, |u_I|^2\\
    & = \frac{1}{|Q_N|} \sum_{\mC_N} \sum_{I \in \CC(\mC_N)}   | I | \, \left|\sum_{j=0}^{k_I-1}  b_{I,j}' \right|^2\\
    & \le  \frac{1}{|Q_N|}  \sum_{\mC_N} \sum_{I \in \CC(\mC_N)}   | I | \,  k_I \sum_{j=0}^{k_I-1} |b_{I,j}'|^2
\end{align*}
Using that $k_I \leq \sharp \mC_N$, and the bound 
$$  \sum_{j=0}^{k_I-1} |b_{I,j}'|^2 \le \sum_{[I,J] \in \Ed(\mC_N)} |b_{IJ} - b_{JI}|^2$$
we find  
\begin{align*}
    \frac{1}{|Q_N|} \mE(F_N,\{u_I\},\{ b_{IJ} \} ) & \le  \frac{1}{|Q_N|}  \sum_{\mC_N} (\sharp \mC_N) \sum_{[I,J] \in \Ed(\mC)} |b_{IJ} - b_{JI}|^2 
    \sum_{I \in \CC(\mC_N)}   | I | \\
    & \le  \frac{1}{|Q_N|}  \sum_{\mC_N} (\sharp \mC_N) |\mC_N| \sum_{[I,J] \in \Ed(\mC)} |b_{IJ} - b_{JI}|^2 
 \end{align*}
We get, for any $p > 2$, denoting $\tilde{s}:=\frac{p}{p-2}$ the conjugate exponent of $\frac{p}{2}$ : 
\begin{align*}
   &  \frac{1}{|Q_N|} \mE(F_N,\{u_I\},\{ b_{IJe} \} )  \\
     & \le \left( \frac{1}{|Q_N|} \sum_{\mC_N} \sum_{[I,J] \in \Ed(\mC_N)}|b_{IJ} - b_{JI}|^{2 \tilde{s}} \right)^{\frac{1}{\tilde{s}}} \left( \frac{1}{|Q_N|} \sum_{\mC_N} \sum_{[I,J] \in \Ed(\mC_N)} (\sharp \mC_N)^{p/2} \, |\mC_N|^{p/2} \right)^{\frac{2}{p}}
\end{align*}
Thanks to the cycle-free hypothesis, we get that $\sharp \Ed(\mC_N)+1= \sharp \mC_N$, so much that 
\begin{align*}
     \frac{1}{|Q_N|} \mE(F_N,\{u_I\},\{ b_{IJe} \} )  & \le C \left( \frac{1}{|Q_N|}  \sum_{[I,J] \in\Ed(F_N)}  |b_{IJ}|^{2 \tilde{s}} \right)^{\frac{2}{\tilde{s}}} \left( \frac{1}{|Q_N|} \sum_{\mC_N} (\sharp \mC_N)^{p/2+1} \, |\mC_N|^{p/2}  \right)^{\frac{2}{p}}.
\end{align*}
Finally, we notice that $|\mC_N| \le  C \, \sharp \mC_N$, so that  
\begin{align*}
\limsup_{N \rightarrow +\infty} \frac{1}{|Q_N|} \sum_{\mC_N} (\sharp \mC_N)^{p/2+1} \, |\mC_N|^{p/2} \: \le \:  & C \limsup_{N \rightarrow +\infty} \frac{1}{|Q_N|} \sum_{\mC_N}  (\sharp \mC_N)^{p} \, |\mC_N|  \\
\: \le \:  & C  \limsup_{N \rightarrow +\infty}  \frac{1}{|Q_N|} \int_{Q_N} (\sharp C_{x,F_N})^p dx \\
\: \le \: & C  \limsup_{N \rightarrow +\infty}  \frac{1}{|Q_N|} \int_{Q_N} (\sharp C_{x,F})^p dx = C \,  \E \,  \sharp C_{0,F}^p 
\end{align*}
This concludes the proof.
\end{proof}

\section*{Acknowledgements}
The authors acknowledge the support of the SingFlows project, grant ANR-18-CE40-0027 of the French National Research Agency (ANR). David G\'erard-Varet acknowledges  the support of the Institut Universitaire de France. Alexandre Girodroux-Lavigne acknowledges the support of the DIM Math Innov de la Région Ile-de-France.

\appendix
\section{Keller function}

The object of this appendix is the following 

\begin{lemma}
\label{lemma::bridge function}
Let $P_1$ and $P_2$ be two paraboloids be defined in cylindrical coordinates as 
\begin{align*}
    & P_1 := \{ z \leq -a r^2  \}, \quad  P_2 := \{ a r^2 + \nu \leq z  \}, \quad  \nu > 0 \text{ small}. 
\end{align*}
Furthermore we note $F_{12}(d)$ the gap of width $d>0$ between the two paraboloids defined by
$$F_{12}(d):= \{ r^2 \leq  d^2, \quad  -a r^2  \leq z \leq \nu + a r^2    \}.$$
(see the picture below). 
Then there exists $w_{\nu} \in  H^{1}(F_{12}(d) \cup P_1 \cup P_2)$ such that 
\begin{align*}
& 0 \leq w_{\nu} \leq 1, \quad  w_{\nu}|_{P_1}=1, \quad w_{\nu}|_{P_2}=0, \\
& \int_{F_{12}(d)} |\na w_{\nu}|^2 \dx \leq C \ln{\frac{1}{\nu}}, \quad  \int_{F_{12}(d)} |\na w_{\nu}|^2 |x|^{2\gamma} \dx \le C_\gamma, \quad \forall \gamma > 0, 
\end{align*}
where $C, C_\gamma$ are constants independent of $\nu$.
\end{lemma}

\definecolor{gray1}{gray}{0.85}
\definecolor{gray2}{gray}{0.67}
\definecolor{gray3}{gray}{0.55}
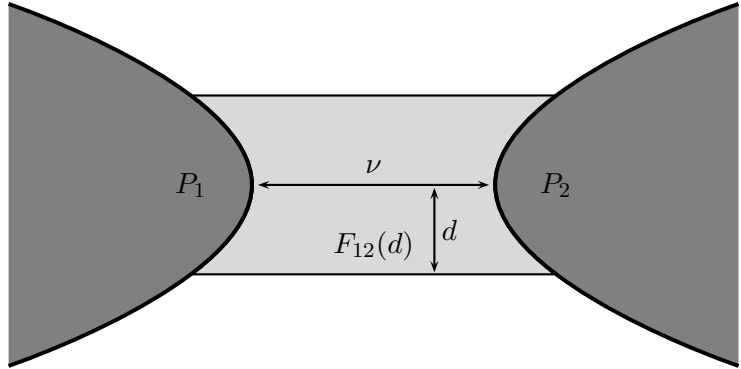
\begin{figure}[h!]
    \centering
    \psset{unit=0.8cm}
\begin{pspicture}(-4,-2.4)(4,3.2)

\begin{psclip}{\psframe[fillstyle=solid,fillcolor=gray1](-3,-1.5) (3,1.5)} 

\pscurve[fillstyle=solid,fillcolor=gray,showpoints=false,linewidth=1.5pt,linecolor=black] (-6,3 ) (-2,0) (-6,-3) 

\pscurve[fillstyle=solid,fillcolor=gray,showpoints=false,linewidth=1.5pt,linecolor=black] (6,3) (2,0) (6,-3) 
\end{psclip}
\pscurve[fillstyle=solid,fillcolor=gray,showpoints=false,linewidth=1.5pt,linecolor=black] (-6,3 ) (-2,0) (-6,-3) 
\pscurve[fillstyle=solid,fillcolor=gray,showpoints=false,linewidth=1.5pt,linecolor=black] (6,3 ) (2,0) (6,-3)

\rput(-3,0){$P_1$}
\rput(3,0){$P_2$}
\rput(0,-1){$F_{12}(d)$}

\psline{<->}(-1.9,0) (1.9,0)
\rput(0,0.3){$\nu$}

\psline{<->}(1,-0.05) (1,-1.45)
\rput(1.25,-0.75){$d$}
\end{pspicture}
    \caption{Geometry of the gap}
    \label{fig:Geometry of the gap}
\end{figure}

\begin{proof}
For $x \in F_{12}(d)$, we set 
$$w_{\nu}(x)= \frac{ar^2+\nu-z}{2 ar^2+\nu}$$
that we extend by $1$ on $P_1$ and $0$ on $P_2$. Clearly, $w_\nu$ is smooth, with values in $[0,1]$. 
As regards the $L^2$ bound on the gradient, it is easily verified that all derivatives are bounded uniformly in $\nu$  in $L^2$ except for  $\partial_z w_\nu$, for which one can compute:

\begin{align*}
\int_{F_{12}(d)} |\partial_z w_\nu|^2  & = 2\pi \int_{0}^d \int_{-ar^2}^{\nu+ar^2} \frac{r}{(2 ar^2 + \nu)^2}  \ \mathrm{d}r \mathrm{d}z\\
& = 2 \pi \int_0^d \frac{r}{(2 ar^2 + \nu)} \ \mathrm{d}r \\
& \le C \ln \frac{1}{\nu} \end{align*}
by direct computation of the integral. The other bounds can be computed similarly.  Note that functions of the type of $w_\nu$ are sometimes referred as Keller functions, cf \cite{keller1964theorem}.
\end{proof}

\section{Poincaré-Wirtinger inequality} \label{appendix_PW}
We explain here how to obtain inequality \eqref{PW} with $r=12$ for inclusions $I$ satisfying (G1). A crucial point, beside the regularity of the boundary, is that the inclusions do not shrink, thanks to the interior ball condition with uniform radius $d$. 

\mspace
Our starting point is the following statement: given an open set $U$ Lipschitz-diffeomorphic to the unit ball, with diffeomorphism $\phi$ satisfying 
$\| \na \phi \|_{L^\infty}  + \|\na \phi^{-1} \|_{L^\infty} \le M$, one has 
$$ \forall p \in [1,\infty), \quad \int_{U \times U} |f(x)- f(y)|^p dx dy \le C_M \int_{U} |\na f(x)|^p dx $$
where $C_M$ depends only on $M$. This can be seen by a direct adaptation of the proof of the Poincaré-Wirtinger inequality for convex domains given in  \cite[chapter 3, page 5]{Mischler}: just take $z_t = \phi^{-1}\big((1-t) \phi(x) + t \phi(y)\big)$ in this proof, instead of $z_t = (1-t) x + t y$.  It follows easily that for all $A,A' \subset U$, with $|A|, |A'| \ge \eta > 0$, 
\begin{equation} \label{PW1}
  \|f - (f)_A\|_{L^p(U)}  \le C \| \na f \|_{L^p(U)},   
 \end{equation} 
and then
 \begin{equation} \label{PW2}
|(f)_A - (f)_{A'}| \le C \| \na f \|_{L^p(U)}. 
 \end{equation} 
where $C$ just depends on $\eta$ (besides $M$).

\mspace
To prove \eqref{PW}, we first introduce a covering  $\mathring{I} = \cup_{k=1}^K O_k$, satisfying the following properties: 
\begin{itemize}
\item[i)]   for each $k$, $O_k$ is an open set  Lipschitz-diffeomorphic to the unit ball, with diffeomorphisms $\phi_k$ s.t. $\| \na \phi_k \|_{L^\infty}  + \|\na \phi_k^{-1} \|_{L^\infty} \le M$ for some $M$ independent of $k$ (and $I$). 
 \item[ii)]  The cardinal $K$ of the covering  is bounded by  $C \, \diam(I)^3$.  
\end{itemize}
There are of course many possible choices to satisfy such properties. As regards i), our assumption (G1) ensures that we can cover  a vicinity of the boundary  by such type of open sets,  while the remaining part of $I$ can be covered directly by balls. As regards ii), if the covering satisfies i) and is uniformly locally finite, meaning: 
$$ \exists M_0 > 0, \quad \forall k, \quad  \sharp \{ k', O_k \cap O_{k'} \neq \emptyset  \} \le M_0   $$
then the cardinal $K$ is comparable to  $| \, I \, |$, hence bounded by $C \, \diam(I)^3$.

\mspace
Then, given this covering, we write 
\begin{align*}
\int_{I} |f - (f)_I|^p \le \frac{1}{| I |} \int_{I \times I} |f(x) - f(y)|^p dx dy & =  \frac{1}{| I |} \sum_{k,k'}  \int_{O_k \times O_{k'}} |f(x) - f(y)|^p dx dy \\
& \le C \, \diam(I)^6 \, \sup_{k,k'} \int_{O_k \times O_{k'}} |f(x) - f(y)|^p dx dy
\end{align*}
where we used property ii) and  the fact that $\frac{1}{|\,I \,|} \ge d^{-3}$, with the constant $d$ in (G1).  
Now, for any fixed couple $k,k'$, we take a sequence $ U_0 = O_k, \: U_1, \dots, \: U_{J} = O_{k'}$ such that 
\begin{itemize}
\item   $U_j$ is a ball for all $j = 1, \dots, J-1$, whose radius is   bounded by some  $R > 0$ uniform in $j$ and  $(k,k')$. 
\item   $ |U_j \cap U_{j+1}| \ge \eta, \quad \forall j=0, \dots, J-1$, for some $\eta$ uniform in $j$ and $k,k'$.  
\item The cardinal $J$ is  bounded by $C' \diam(I)^3$ for some constant uniform in $k,k'$.    
\end{itemize}
We finally write: for all $k,k'$, 
\begin{align*}
& \int_{O_k \times O_{k'}} |f(x) - f(y)|^p dx dy  \\
 = & \int_{O_k \times O_{k'}} | f(x) - (f)_{U_0 \cap U_1} + (f)_{U_0 \cap U_1} - (f)_{U_1 \cap U_2} + \dots + (f)_{U_{J-1} \cap U_J} - f(y)|^p  \\
\le & C J^p \Big(  \int_{O_k}  | f(x) - (f)_{U_0 \cap U_1}|^p dx + \sum_{j=0}^{J-1}  |(f)_{U_j \cap U_{j+1}} - (f)_{U_{j+1} \cap U_{j+2}}|^p +  \int_{O_{k'}}  |(f)_{U_{J-1} \cap U_J} - f(y)|^p dy \Big) \\ 
\le & C' J^p \sum_{j=0}^J  \| \na f \|^p_{L^p(U_j)} \\
 \le & C' J^{p+1}  \| \na f \|^p_{L^p(I)}  \le C'' \diam(I)^{3p+3}  \| \na f \|^p_{L^p(I)}. 
\end{align*}  
 where we used \eqref{PW1}-\eqref{PW2}   for the second inequality.

{\footnotesize
\bibliographystyle{siam}

 \bibliography{biblio_graph}

\begin{thebibliography}{10}

\bibitem{akcoglu1981ergodic}
{\sc M.~A. Akcoglu and U.~Krengel}, {\em Ergodic theorems for superadditive
  processes}, J. Reine Angew. Math., 323 (1981), pp.~53--67.

\bibitem{AGKL}
{\sc H.~Ammari, P.~Garapon, H.~Kang, and H.~Lee.}, {\em Effective viscosity
  properties of dilute suspensions of arbitrarily shaped particles.}, Asymptot.
  Anal., 80(3-4),  (2012), pp.~189,211.

\bibitem{berkolaiko2013introduction}
{\sc G.~Berkolaiko and P.~Kuchment}, {\em Introduction to quantum graphs},
  no.~186, American Mathematical Soc., 2013.

\bibitem{BeBoPa}
{\sc L.~Berlyand, L.~Borcea, and A.~Panchenko}, {\em Network approximation for
  effective viscosity of concentrated suspensions with complex geometry}, SIAM
  J. Math. Anal., 36 (2005), pp.~1580--1628.

\bibitem{MR2525112}
{\sc L.~Berlyand, Y.~Gorb, and A.~Novikov}, {\em Fictitious fluid approach and
  anomalous blow-up of the dissipation rate in a two-dimensional model of
  concentrated suspensions}, Arch. Ration. Mech. Anal., 193 (2009),
  pp.~585--622.

\bibitem{MR1857272}
{\sc L.~Berlyand and A.~Kolpakov}, {\em Network approximation in the limit of
  small interparticle distance of the effective properties of a high-contrast
  random dispersed composite}, Arch. Ration. Mech. Anal., 159 (2001),
  pp.~179--227.

\bibitem{BeKoNo}
{\sc L.~Berlyand, A.~G. Kolpakov, and A.~Novikov}, {\em Introduction to the
  network approximation method for materials modeling}, vol.~148 of
  Encyclopedia of Mathematics and its Applications, Cambridge University Press,
  Cambridge, 2013.

\bibitem{BeMi}
{\sc L.~Berlyand and V.~Mityushev}, {\em Increase and decrease of the effective
  conductivity of two phase composites due to polydispersity}, J. Stat. Phys.,
  118 (2005), pp.~481--509.

\bibitem{berlyand2002error}
{\sc L.~Berlyand and A.~Novikov}, {\em Error of the network approximation for
  densely packed composites with irregular geometry}, SIAM journal on
  mathematical analysis, 34 (2002), pp.~385--408.

\bibitem{MR2319721}
{\sc L.~Berlyand and A.~Panchenko}, {\em Strong and weak blow-up of the viscous
  dissipation rates for concentrated suspensions}, J. Fluid Mech., 578 (2007),
  pp.~1--34.

\bibitem{bart}
{\sc B.~Blaszczyszyn}, {\em Lecture notes on random geometric models. random
  graphs, point processes and stochastic geometry}.
\newblock Doctoral. Japan, 2017.

\bibitem{Bollobas}
{\sc B.~Bollob\'{a}s}, {\em Modern graph theory}, vol.~184 of Graduate Texts in
  Mathematics, Springer-Verlag, New York, 1998.

\bibitem{BoPa}
{\sc L.~Borcea and G.~C. Papanicolaou}, {\em Network approximation for
  transport properties of high contrast materials}, SIAM J. Appl. Math., 58
  (1998), pp.~501--539.

\bibitem{modica1986nonlinear}
{\sc G.~Dal~Maso and L.~Modica}, {\em Nonlinear stochastic homogenization},
  Ann. Mat. Pura Appl. (4), 144 (1986), pp.~347--389.

\bibitem{duerinckx2020effective}
{\sc M.~Duerinckx}, {\em Effective viscosity of random suspensions without
  uniform separation}, 2020.

\bibitem{DuerinckxGloria}
{\sc M.~Duerinckx and A.~Gloria}, {\em Corrector equations in fluid mechanics:
  Effective viscosity of colloidal suspensions}, Arch. Ration. Mech. Anal. in
  press,  (2020).

\bibitem{DueGloria}
\leavevmode\vrule height 2pt depth -1.6pt width 23pt, {\em On einstein's
  effective viscosity formula}.
\newblock arXiv:2008.03837, August 9th 2020.

\bibitem{DuGloperco}
\leavevmode\vrule height 2pt depth -1.6pt width 23pt, {\em Continuum
  percolation in stochastic homogenization and the effective viscosity
  problem}.
\newblock arXiv:2108.09654, September 2021.

\bibitem{gerard2021derivation}
{\sc D.~G{\'e}rard-Varet}, {\em Derivation of the batchelor-green formula for
  random suspensions}, Journal de Math{\'e}matiques Pures et Appliqu{\'e}es,
  (2021).

\bibitem{GVH}
{\sc D.~G\'erard-Varet and M.~Hillairet}, {\em Analysis of the viscosity of
  dilute suspensions beyond einstein's formula}, Arch. Ration. Mech. Anal., 238
  (2020), pp.~1349--1411.

\bibitem{GVRH}
{\sc D.~G\'erard-Varet and R.~H\"ofer}, {\em Mild assumptions for the
  derivation of einstein's effective viscosity formula}, Comm. Partial Diff.
  Eq. in press,  (2020).

\bibitem{gerard2020correction}
{\sc D.~G{\'e}rard-Varet and A.~Mecherbet}, {\em On the correction to
  einstein's formula for the effective viscosity}, arXiv preprint
  arXiv:2004.05601,  (2020).

\bibitem{Haines&Mazzucato}
{\sc B.~M. Haines and A.~L. Mazzucato}, {\em A proof of einstein's effective
  viscosity for a dilute suspension of spheres.}, SIAM J. Math. Anal., 44(3),
  (2012), pp.~[2120,2145].

\bibitem{MR4098775}
{\sc M.~Hillairet and D.~Wu}, {\em Effective viscosity of a polydispersed
  suspension}, J. Math. Pures Appl. (9), 138 (2020), pp.~413--447.

\bibitem{book:ZKO}
{\sc V.~Jikov, S.~Kozlov, and O.~Oleinik}, {\em Homogenization of Differential
  Operators}, Springer-Verlag Berlin Heidelberg, 1994.

\bibitem{keller1964theorem}
{\sc J.~B. Keller}, {\em A theorem on the conductivity of a composite medium},
  Journal of Mathematical Physics, 5 (1964), pp.~548--549.

\bibitem{Leoni}
{\sc G.~Leoni}, {\em A first course in {S}obolev spaces}, vol.~181 of Graduate
  Studies in Mathematics, American Mathematical Society, Providence, RI,
  second~ed., 2017.

\bibitem{Mischler}
{\sc S.~Mischler}, {\em An introduction to evolution pdes}.
\newblock Academic master 2nd year, 2020.

\bibitem{MR4102716}
{\sc B.~Niethammer and R.~Schubert}, {\em A local version of {E}instein's
  formula for the effective viscosity of suspensions}, SIAM J. Math. Anal., 52
  (2020), pp.~2561--2591.

\bibitem{stein1970singular}
{\sc E.~M. Stein}, {\em Singular integrals and differentiability properties of
  functions}, vol.~2, Princeton university press, 1970.

\end{thebibliography}
 }
 
 \end{document}